\documentclass[11pt,reqno]{amsart}
\usepackage[T1]{fontenc}
\usepackage{lmodern}
\usepackage[a4paper,margin=28mm]{geometry}
\usepackage{amsmath,amssymb,amsthm,mathtools}
\usepackage[expansion=false]{microtype}
\usepackage[hidelinks]{hyperref}
\allowdisplaybreaks[2]
\numberwithin{equation}{section}
\newtheorem{theorem}{Theorem}[section]
\newtheorem{proposition}[theorem]{Proposition}
\newtheorem{lemma}[theorem]{Lemma}
\newtheorem{corollary}[theorem]{Corollary}
\theoremstyle{definition}
\newtheorem{definition}[theorem]{Definition}
\newtheorem{example}[theorem]{Example}
\theoremstyle{remark}
\newtheorem{remark}[theorem]{Remark}
\newcommand{\g}{\mathfrak g}

\newcommand{\F}{\mathbb F}
\newcommand{\id}{\mathrm{id}}
\newcommand{\Hom}{\operatorname{Hom}}
\newcommand{\End}{\operatorname{End}}
\newcommand{\Der}{\operatorname{Der}}
\newcommand{\Aut}{\operatorname{Aut}}
\newcommand{\Ext}{\operatorname{Ext}}

\newcommand{\ad}{\operatorname{ad}}
\newcommand{\im}{\operatorname{im}}

\newcommand{\tr}{\operatorname{tr}}
\newcommand{\nab}{\mathrm{nab}}
\newcommand{\ab}{\mathrm{ab}}
\newcommand{\Ob}{\operatorname{Ob}}

\newcommand{\At}{\mathcal A_{\vartheta}}
\newcommand{\Ft}{\mathcal F_{\vartheta}}
\newcommand{\Ht}{\mathcal H^2_{\rho,\vartheta}}
\newcommand{\Zt}{Z^2_{\rho,\vartheta}}

\title[Wells exact sequences for $n$-Lie algebras]
{Wells exact sequences for $n$-Lie algebras}
\date{}
\hypersetup{pdftitle={Wells exact sequences for n-Lie algebras}}

\author{Tao Zhang}
\address{School of Mathematics and Statistics (School of Cryptology), Henan Normal University,
Xinxiang 453007, P. R. China}
\email{zhangtao@htu.edu.cn}

\newcommand{\Inn}{\operatorname{Inn}}
\newcommand{\Out}{\operatorname{Out}}
\newcommand{\ext}{\mathrm{ext}}

\allowdisplaybreaks
\begin{document}
\begin{abstract}
We study derivations and automorphisms of arbitrary non-abelian extensions of $n$-Lie algebras through their full crossed-product data.
Their relative center is a canonical coefficient module, and a third cohomology class obstructs completion of data satisfying the mixed Filippov identities. When this class vanishes, a relative second cohomology space classifies the remaining extension data. We prove a non-abelian good criterion, construct the Lie algebra and group of compatible pairs, and obtain derivation and automorphism Wells sequences with explicit lifting equations. The derivation obstruction is a Lie algebra $1$-cocycle; the additive automorphism obstruction is a crossed homomorphism. We distinguish the ordinary center from the relative center, and outer derivations of the kernel from the mixed action required. Sparse crossed products are treated, including the representation class of their second mixed operation and the additional equations governing changes of sparse sections.
\end{abstract}
\maketitle

\section{Introduction}

An extension separates an algebra into an ideal and a quotient.
The extension problem asks which brackets combine these two
algebras; the lifting problem asks which derivations or automorphisms
of the two factors preserve the resulting extension. Wells'
sequence~\cite{Wells} expresses the latter problem as an obstruction
to preserving an extension class. Its Lie algebra and $3$-Lie algebra
forms were studied in~\cite{BardakovSingh,TanXu,Xu}.
For an arbitrary non-abelian kernel, the obstruction must account
for both the quotient bracket and every mixed bracket.

The theory of $n$-Lie algebras began with Filippov~\cite{Filippov}.
Its representation and cohomology theories are developed
in~\cite{Kasymov,DaletskiiTakhtajan,Takhtajan,CasasLodayPirashvili}.
The structural and deformation aspects appear in~\cite{Ling,AzcarragaIzquierdo}.
The extending structures for \(3\)-Lie algebras was given in~\cite{Zhang3Lie}.
Non-abelian extensions and their equivalences were investigated
in~\cite{SongMakhloufTang,AfiBasdouri}. We use the full crossed-product
construction of~\cite{ZhangExtending} for $n$-Lie algebras.
An extension of $\g$ by $V$ is an exact sequence of $n$-Lie algebras
\begin{equation}
\mathcal E:\quad
0\longrightarrow V\xrightarrow{i}E\xrightarrow{p}\g
\longrightarrow0.
\end{equation}
With quotient $\g$ and kernel $V$, a section determines
\[
\omega:\bigwedge^n\g\longrightarrow V,\qquad
\rho_r:\bigwedge^{n-r}\g\otimes\bigwedge^rV\longrightarrow V,
\quad 1\leq r\leq n-1.
\]
For $n\geq4$, the intermediate maps $\rho_2,\ldots,\rho_{n-2}$
carry information that is absent from the ternary case.

Classical non-abelian Lie extension theory has two related features.
An outer action on the kernel induces an action on its center;
a class in third cohomology obstructs existence, and second
cohomology describes extensions once the obstruction vanishes.
Alekseevsky, Michor, and Ruppert~\cite{AlekseevskyMichorRuppert}
establish this picture for super Lie algebras.
Fr\'egier~\cite{Fregier} treats ordinary Lie extensions through
non-abelian cohomology and its gauge equivalence.
Inassaridze, Khmaladze, and Ladra~\cite{InassaridzeKhmaladzeLadra}
develop a different coefficient theory using crossed modules.

For the $n$-Lie algebras case, the familiar formula
$\Out(V)=\Der(V)/\Inn(V)$ remains a quotient of ordinary Lie
algebras. It does not by itself encode an arbitrary $n$-Lie
extension. Indeed, changing a section can change $\rho_1$ by
mixed derivations that are not inner derivations of $V$.
Moreover, the relevant coefficient module is generally
\[
A_\rho=\{a\in V:[a,b,z_3,\ldots,z_n]_E=0
\text{ for all }b\in V,\ z_j\in E\},
\]
which can be strictly smaller than $Z(V)$.
These distinctions determine the form of the obstruction theory.

The main results are as follows. First, we construct a third
cohomology obstruction for a full mixed family for which all mixed
Filippov identities can be solved. A finite bracket calculation
shows that the remaining defect is an $A_\rho$-valued cocycle.
The cohomology convention retains the alternating primitive
required to change an $n$-ary bracket. Vanishing is equivalent to
existence of an extension with that mixed family.
Second, we identify exactly those changes of section preserving
the mixed brackets. They form a vector space $\mathcal F_\rho$,
and their pure changes are the finite coboundaries $d_\rho f$.
Consequently, the extension classes with fixed mixed maps form
an affine space over
\[
\mathcal H_\rho^2
=Z^2_{\ab}(\g,A_\rho)/d_\rho\mathcal F_\rho.
\]
The difference between $\mathcal F_\rho$ and
$\Hom(\g,A_\rho)$ explains why this space can be a proper quotient
of ordinary alternating extension cohomology.

Third, compatible derivation pairs form a Lie algebra
$\mathfrak G_\rho$ acting on $\mathcal H_\rho^2$.
Solving their mixed equations gives the obstruction
\[
\mathcal W_{\mathcal E}(D_V,D_\g)
=[L_{(D_V,D_\g)}\omega-d_\rho\lambda].
\]
Its vanishing characterizes extensibility, and it satisfies the
Lie algebra cocycle identity. Automorphism pairs have a pointed
Wells map on the full non-abelian cohomology set and an additive
Wells map on the group $\mathcal C_\rho$ of mixed-compatible pairs.
We determine all lifts, the kernels of the sequences, and the
factor sets controlling their splitting.

The full structural criterion and the classification by section
changes are recalled with citations. The proofs concern the
central obstruction, section invariance, and the two lifting
problems. Sparse crossed products occur only in
Section~\ref{sec:sparse}. Explicit examples show why zero
intermediate brackets and the distinction between $A_\rho$ and
$Z(V)$ cannot be omitted.

Throughout, $\F$ is a field of characteristic zero, $n\geq2$,
and all maps are $\F$-linear. The sparse specialization assumes
$n\geq3$. The quotient is always $\g$ and the kernel is always
$V$. We write $x_i, y_i\in\g$ and $a_i, a, b\in V$.  An empty sum is zero. Cohomological degrees used for
extensions are specified in Section~\ref{sec:centers-obstruction}.

\section{Full crossed products and non-abelian extensions}\label{sec:data}

\subsection{Preliminaries}
An $n$-Lie algebra is a vector space $L$ with an alternating
$n$-linear bracket satisfying
\begin{equation}\label{eq:FI}
[u_1,\ldots,u_{n-1},[v_1,\ldots,v_n]]
 =\sum_{i=1}^n
 [v_1,\ldots,[u_1,\ldots,u_{n-1},v_i],\ldots,v_n].
\end{equation}
A subspace $I$ is an ideal if $[I,L,\ldots,L]\subseteq I$.
A derivation $D$ satisfies
\begin{equation}\label{eq:derivation}
D[z_1,\ldots,z_n]
 =\sum_{i=1}^n[z_1,\ldots,Dz_i,\ldots,z_n].
\end{equation}
In particular, $\ad(u_1,\ldots,u_{n-1})z
=[u_1,\ldots,u_{n-1},z]$ is a derivation.
We use
\[
Z(L)=\{a\in L:[a,L,\ldots,L]=0\},\qquad
[L,\ldots,L]=\operatorname{span}\{[z_1,\ldots,z_n]\}.
\]
For an $n$-Lie algebra $L$, a representation on $M$ is a map
$\rho_M:\bigwedge^{n-1}L\to\End(M)$ satisfying
\begin{align}
[\rho_M(X),\rho_M(y_1,\ldots,y_{n-1})]
 &=\sum_{i=1}^{n-1}
 \rho_M(y_1,\ldots,[X,y_i],\ldots,y_{n-1}),\notag\\
\rho_M(x_1,\ldots,x_{n-2},[y_1,\ldots,y_n])
 &=\sum_{i=1}^n(-1)^{n-i}
 \rho_M(y_1,\ldots,\widehat y_i,\ldots,y_n)
 \rho_M(x_1,\ldots,x_{n-2},y_i).
\label{eq:representation-definition}
\end{align}
Here $X=(x_1,\ldots,x_{n-1})$ and $[X,y]=[x_1,\ldots,x_{n-1},y]$.
These are the standard representation identities; see
\cite{Kasymov,Takhtajan}.

\begin{definition}\label{def:extension}
An extension of $\g$ by $V$ is an exact sequence of $n$-Lie algebras
\begin{equation}\label{eq:extension}
\mathcal E:\quad
0\longrightarrow V\xrightarrow{i}E\xrightarrow{p}\g
\longrightarrow0.
\end{equation}
We identify $V$ with $i(V)$. Two extensions are equivalent if there
is an $n$-Lie isomorphism between their middle algebras that is the
identity on $V$ and induces the identity on $\g$.
Their equivalence classes form $\Ext_{\nab}(\g,V)$.
Here the term non-abelian allows an arbitrary kernel; it does not
require that the kernel bracket be nonzero.
\end{definition}

Fix a linear section $s$ of $p$. Every $e\in E$ has a unique expression
\[
e=s(x)+a,\qquad x=p(e),\quad a=e-s(p(e)).
\]
The coordinate isomorphism is
\begin{equation}\label{eq:coordinate-isomorphism}
F_s:\g\oplus V\longrightarrow E,\qquad F_s(x+a)=s(x)+a.
\end{equation}
Since $p$ preserves brackets, the following maps take values in $V$:
\begin{align}
\omega(x_1,\ldots,x_n)
 &=[s(x_1),\ldots,s(x_n)]-s[x_1,\ldots,x_n]_{\g},
 \label{eq:omega-section}\\
\rho_r(x_1,\ldots,x_{n-r})(a_1,\ldots,a_r)
 &=[s(x_1),\ldots,s(x_{n-r}),a_1,\ldots,a_r],
 \quad 1\leq r\leq n-1.
 \label{eq:rho-section}
\end{align}
Each map is alternating separately in its two displayed groups of
arguments. We set
\begin{equation}\label{eq:rho-n}
\rho=\rho_1,\qquad
\rho_n(a_1,\ldots,a_n)=[a_1,\ldots,a_n]_{V},
\qquad
\Omega=(\omega,\rho_1,\ldots,\rho_{n-1}).
\end{equation}
Thus $\rho_n$ is fixed, not part of the variable extension datum.
An empty list of $\g$-arguments is suppressed.

\subsection{The full bracket}
Conversely, let $\omega:\bigwedge^n\g\to V$ and
\[
\rho_r:\bigwedge^{n-r}\g\otimes\bigwedge^rV\longrightarrow V,
\qquad 1\leq r\leq n-1,
\]
be given. On $E=\g\oplus V$ define an alternating operation by
\begin{align}
[x_1,\ldots,x_n]_{\Omega}
 &=[x_1,\ldots,x_n]_{\g}+\omega(x_1,\ldots,x_n),
 \label{eq:bracket-pure}\\
[x_1,\ldots,x_{n-r},a_1,\ldots,a_r]_{\Omega}
 &=\rho_r(x_1,\ldots,x_{n-r})(a_1,\ldots,a_r),
 \quad 1\leq r\leq n.
 \label{eq:bracket-mixed}
\end{align}
These formulas specify the bracket on every pure input type.
For arbitrary $x_i+a_i$, the operation is obtained by multilinearity;
each occurrence of $a_i$ remains in its original slot until
alternation moves it to the right. No mixed component is set to zero.

For example, when $n=3$, write $\vartheta=\rho_2$. Then
\begin{align}\label{eq:ternary-bracket}
[x_1+a_1,x_2+a_2,x_3+a_3]_{\Omega}
={}&[x_1,x_2,x_3]_{\g}+\omega(x_1,x_2,x_3)\notag\\
&+\rho(x_1,x_2)a_3-\rho(x_1,x_3)a_2+\rho(x_2,x_3)a_1\notag\\
&+\vartheta(x_1)(a_2,a_3)-\vartheta(x_2)(a_1,a_3)
  +\vartheta(x_3)(a_1,a_2)\notag\\
&+[a_1,a_2,a_3]_{V}.
\end{align}
This is the full ternary bracket of~\cite{SongMakhloufTang,Zhang3Lie}.
For $n=4$ the two-kernel terms, for instance, are
\begin{align}\label{eq:four-mixed}
&\rho_2(x_1,x_2)(a_3,a_4)
-\rho_2(x_1,x_3)(a_2,a_4)
+\rho_2(x_1,x_4)(a_2,a_3)\notag\\
&\quad+\rho_2(x_2,x_3)(a_1,a_4)
-\rho_2(x_2,x_4)(a_1,a_3)
+\rho_2(x_3,x_4)(a_1,a_2).
\end{align}
They are additional to the one-kernel and three-kernel terms.

\subsection{The cocycle equations}
In the next formulas,
$X=(x_1,\ldots,x_{n-1})$,
$Y=(y_1,\ldots,y_n)$, and
$[X,y]=[x_1,\ldots,x_{n-1},y]_{\g}$.
The notation $\omega(X,y)$ has the same literal meaning.
For formulas involving $r$ outer kernel entries, put
$p=n-1-r$; for $s$ inner kernel entries, put $q=n-s$.
All sums with no indices are zero.

\begin{theorem}[\cite{ZhangExtending}]\label{thm:cocycle-criterion}
The operation \eqref{eq:bracket-pure}--\eqref{eq:bracket-mixed}
is an $n$-Lie bracket if and only if the following four families
hold:
\begin{align}
\rho(X)\omega(Y)+\omega(X,[y_1,\ldots,y_n]_{\g})
={}&\sum_{i=1}^n
 \omega(y_1,\ldots,[X,y_i],\ldots,y_n)\notag\\
&+\sum_{i=1}^n(-1)^{n-i}
 \rho(y_1,\ldots,\widehat y_i,\ldots,y_n)\omega(X,y_i);
\label{eq:NA0}
\end{align}
for $1\leq s\leq n$, with $q=n-s$,
\begin{align}
&\rho(X)\rho_s(y_1,\ldots,y_q)(b_1,\ldots,b_s)\notag\\
&=\sum_{i=1}^{q}
 \rho_s(y_1,\ldots,[X,y_i],\ldots,y_q)(b_1,\ldots,b_s)\notag\\
&\quad+\sum_{i=1}^{q}(-1)^{q-i}
 \rho_{s+1}(y_1,\ldots,\widehat y_i,\ldots,y_q)
       (\omega(X,y_i),b_1,\ldots,b_s)\notag\\
&\quad+\sum_{j=1}^{s}
 \rho_s(y_1,\ldots,y_q)
       (b_1,\ldots,\rho(X)b_j,\ldots,b_s);
\label{eq:NA1}
\end{align}
for $1\leq r\leq n-1$, with $p=n-1-r$,
\begin{align}
&(-1)^r\rho_r(x_1,\ldots,x_p,[y_1,\ldots,y_n]_{\g})(a_1,\ldots,a_r)\notag\\
&\quad+\rho_{r+1}(x_1,\ldots,x_p)
            (a_1,\ldots,a_r,\omega(Y))\notag\\
&=\sum_{i=1}^n(-1)^{n-i+r}
 \rho(y_1,\ldots,\widehat y_i,\ldots,y_n)
 \rho_r(x_1,\ldots,x_p,y_i)(a_1,\ldots,a_r);
\label{eq:NAr0}
\end{align}
and, for $1\leq r\leq n-1$, $1\leq s\leq n$,
\begin{align}
&\rho_{r+1}(x_1,\ldots,x_p)
 (a_1,\ldots,a_r,\rho_s(y_1,\ldots,y_q)(b_1,\ldots,b_s))
 \notag\\
&=\sum_{i=1}^{q}(-1)^{q-i+r}
 \rho_{s+1}(y_1,\ldots,\widehat y_i,\ldots,y_q)
 \bigl(\rho_r(x_1,\ldots,x_p,y_i)(a_1,\ldots,a_r),
       b_1,\ldots,b_s\bigr)\notag\\
&\quad+\sum_{j=1}^{s}
 \rho_s(y_1,\ldots,y_q)
 \bigl(b_1,\ldots,
 \rho_{r+1}(x_1,\ldots,x_p)(a_1,\ldots,a_r,b_j),
 \ldots,b_s\bigr).
\label{eq:NArs}
\end{align}
In \eqref{eq:NA1} and \eqref{eq:NArs}, the sum involving
$\rho_{s+1}$ is absent when $s=n$.
\end{theorem}

This is the full crossed-product criterion of
\cite{ZhangExtending}, written with all mixed components and with the
kernel an ideal; see also the extension constructions in
\cite{AfiBasdouri,SongMakhloufTang}.
We retain its component identities for the subsequent calculations.

\begin{definition}\label{def:cocycle}
A family $\Omega$ satisfying
\eqref{eq:NA0}--\eqref{eq:NArs} is a \emph{non-abelian $2$-cocycle}.
Write $Z^2_{\nab}(\g,V)$ for their set, and
$\g\oplus_{\Omega}V$ for the corresponding $n$-Lie algebra.
The index $2$ is the extension-theoretic degree; it does not count
the number of arguments of $\omega$.
\end{definition}

\begin{remark}\label{rem:curvature}
Taking $s=n$ in \eqref{eq:NA1} proves
$\rho(X)\in\Der(V)$. Taking $s=1$ gives
\begin{align}\label{eq:curvature}
&[\rho(X),\rho(y_1,\ldots,y_{n-1})]a
 -\sum_{i=1}^{n-1}
 \rho(y_1,\ldots,[X,y_i],\ldots,y_{n-1})a\notag\\
&=\sum_{i=1}^{n-1}(-1)^{n-1-i}
 \rho_2(y_1,\ldots,\widehat y_i,\ldots,y_{n-1})
             (\omega(X,y_i),a).
\end{align}
Thus $\rho$ need not be a representation of $\g$.
For $n=2$, $\rho_2$ is the bracket of $V$, and
\eqref{eq:curvature} becomes
$[\rho(x),\rho(y)]-\rho([x,y])=\ad_{V}\omega(x,y)$.
For general $n$, \eqref{eq:NA0} remains linear in $\omega$ when
$\rho$ is fixed, but the full cocycle system is nonlinear.
\end{remark}

\begin{remark}\label{rem:full-crossed}
The family above is the full crossed-product datum of
\cite{ZhangExtending}, with $V$ as ideal and $\g$ as quotient:
its component with $k$ quotient entries is
$\gamma_k=\rho_{n-k}$. Retaining only $\rho_1$ and
$\rho_{n-1}$ gives the sparse crossed products of that reference
when $n\geq3$. Such a restriction is additional for $n\geq4$ and
is not imposed in the general theory.
\end{remark}

\section{Changes of section and non-abelian cohomology}
\label{sec:cohomology}

\subsection{Exact change-of-section formulas}
For $f\in\Hom(\g,V)$ put
\begin{equation}\label{eq:shear}
T_f(x+a)=x+f(x)+a,\qquad
T_f^{-1}=T_{-f},\qquad T_fT_k=T_{f+k}.
\end{equation}
Define $\Omega^f$ by
\begin{equation}\label{eq:gauge-bracket}
[z_1,\ldots,z_n]_{\Omega^f}
 =T_{-f}[T_fz_1,\ldots,T_fz_n]_{\Omega}.
\end{equation}
Since $T_f$ is the identity on $V$ and induces the identity on
$\g$, this operation again has the form
\eqref{eq:bracket-pure}--\eqref{eq:bracket-mixed}.

\begin{proposition}\label{prop:gauge}
The datum of the section $s+f$ is $\Omega^f$. For $m=n-r$ and
$1\leq r\leq n-1$, its components are
\begin{align}
\rho_r^f(x_1,\ldots,x_m)(a_1,\ldots,a_r)
 &=[x_1+f(x_1),\ldots,x_m+f(x_m),a_1,\ldots,a_r]_{\Omega},
 \label{eq:gauge-rho}\\
\omega^f(x_1,\ldots,x_n)
 &=[x_1+f(x_1),\ldots,x_n+f(x_n)]_{\Omega}
   -[x_1,\ldots,x_n]_{\g}
   -f([x_1,\ldots,x_n]_{\g}).
 \label{eq:gauge-omega}
\end{align}
Equivalently, their differences are the following finite sums:
\begin{align}
&(\rho_r^f-\rho_r)(x_1,\ldots,x_m)(a_1,\ldots,a_r)\notag\\
&=\sum_{i=1}^{m}
 [x_1,\ldots,x_{i-1},f(x_i),
   x_{i+1}+f(x_{i+1}),\ldots,x_m+f(x_m),
   a_1,\ldots,a_r]_{\Omega},
 \label{eq:gauge-rho-telescope}\\
&(\omega^f-\omega)(x_1,\ldots,x_n)\notag\\
&=\sum_{i=1}^{n}
 [x_1,\ldots,x_{i-1},f(x_i),
   x_{i+1}+f(x_{i+1}),\ldots,x_n+f(x_n)]_{\Omega}
   -f([x_1,\ldots,x_n]_{\g}).
 \label{eq:gauge-omega-telescope}
\end{align}
Every bracket in these sums lies in $V$ and is evaluated by
\eqref{eq:bracket-mixed}. In particular, the formulas include all
terms of every degree in $f$.
\end{proposition}

\begin{proof}
For a mixed bracket the quotient component is zero. Hence $T_{-f}$
acts trivially on its value, which gives \eqref{eq:gauge-rho}.
For a pure quotient bracket the quotient component is
$[x_1,\ldots,x_n]_{\g}$; applying $T_{-f}$ subtracts
$f([x_1,\ldots,x_n]_{\g})$. Its remaining kernel component is
\eqref{eq:gauge-omega}.

For clarity, multilinearity gives the telescoping identity
\begin{align*}
&[v_1+w_1,\ldots,v_m+w_m,c_1,\ldots,c_{n-m}]
 -[v_1,\ldots,v_m,c_1,\ldots,c_{n-m}]\\
&=\sum_{i=1}^{m}
 [v_1,\ldots,v_{i-1},w_i,v_{i+1}+w_{i+1},
  \ldots,v_m+w_m,c_1,\ldots,c_{n-m}].
\end{align*}
Consecutive terms cancel when the sum is written as differences
between brackets with successively fewer shifted entries.
Taking $v_i=x_i$, $w_i=f(x_i)$ proves
\eqref{eq:gauge-rho-telescope} and
\eqref{eq:gauge-omega-telescope}. Finally, replacing $x_i$ in
these calculations by $s(x_i)$ gives exactly the defining maps
of $s+f$ in \eqref{eq:omega-section}--\eqref{eq:rho-section}.
\end{proof}

For the highest mixed component the formula reads
\begin{equation}\label{eq:gauge-top}
\rho_{n-1}^f(x)(a_1,\ldots,a_{n-1})
 =\rho_{n-1}(x)(a_1,\ldots,a_{n-1})
  +[f(x),a_1,\ldots,a_{n-1}]_{V}.
\end{equation}
For $n\geq3$, the next component is
\begin{align}\label{eq:gauge-next}
&\rho_{n-2}^f(x,y)(a_1,\ldots,a_{n-2})
 -\rho_{n-2}(x,y)(a_1,\ldots,a_{n-2})\notag\\
&=\rho_{n-1}(x)(f(y),a_1,\ldots,a_{n-2})
 -\rho_{n-1}(y)(f(x),a_1,\ldots,a_{n-2})\notag\\
&\quad+[f(x),f(y),a_1,\ldots,a_{n-2}]_{V}.
\end{align}
When $n=2$ these formulas reduce to
\begin{align}
\rho^f(x)a&=\rho(x)a+[f(x),a]_{V},\notag\\
\omega^f(x,y)
 &=\omega(x,y)+\rho(x)f(y)-\rho(y)f(x)
   +[f(x),f(y)]_{V}-f([x,y]_{\g}).
 \label{eq:binary-gauge}
\end{align}
When $n=3$, with $\vartheta=\rho_2$ and $f_i=f(x_i)$, one obtains
\begin{align}
\vartheta^f(x)(a,b)&=\vartheta(x)(a,b)+[f(x),a,b]_{V},\notag\\
\rho^f(x,y)a&=\rho(x,y)a
 +\vartheta(x)(f(y),a)-\vartheta(y)(f(x),a)+[f(x),f(y),a]_{V},
 \label{eq:ternary-gauge-rho}\\
\omega^f(x_1,x_2,x_3)
={}&\omega(x_1,x_2,x_3)
 +\rho(x_1,x_2)f_3-\rho(x_1,x_3)f_2+\rho(x_2,x_3)f_1
 \notag\\
&+\vartheta(x_1)(f_2,f_3)-\vartheta(x_2)(f_1,f_3)
 +\vartheta(x_3)(f_1,f_2)\notag\\
&+[f_1,f_2,f_3]_{V}-f([x_1,x_2,x_3]_{\g}).
 \label{eq:ternary-gauge-omega}
\end{align}
In particular, the change of $\omega$ is not generally a linear
coboundary.

\begin{lemma}\label{lem:gauge-action}
For $\Omega\in Z^2_{\nab}(\g,V)$ and $f,k\in\Hom(\g,V)$,
\begin{equation}\label{eq:gauge-action}
\Omega^f\in Z^2_{\nab}(\g,V),\qquad
(\Omega^f)^k=\Omega^{f+k},\qquad \Omega^0=\Omega.
\end{equation}
Moreover, $T_f:\g\oplus_{\Omega^f}V\to\g\oplus_{\Omega}V$
is an equivalence of extensions.
\end{lemma}

\begin{proof}
For arbitrary $u_i,v_j\in E$, \eqref{eq:gauge-bracket} gives
\begin{align*}
&[u_1,\ldots,u_{n-1},[v_1,\ldots,v_n]_{\Omega^f}]_{\Omega^f}\\
&=T_{-f}[T_fu_1,\ldots,T_fu_{n-1},
                [T_fv_1,\ldots,T_fv_n]_{\Omega}]_{\Omega}\\
&=\sum_{i=1}^n T_{-f}
 [T_fv_1,\ldots,[T_fu_1,\ldots,T_fu_{n-1},T_fv_i]_{\Omega},
                   \ldots,T_fv_n]_{\Omega}\\
&=\sum_{i=1}^n
 [v_1,\ldots,[u_1,\ldots,u_{n-1},v_i]_{\Omega^f},
                   \ldots,v_n]_{\Omega^f}.
\end{align*}
Thus $\Omega^f$ is a cocycle by Theorem~\ref{thm:cocycle-criterion}.
Using \eqref{eq:shear} twice,
\begin{align*}
[z_1,\ldots,z_n]_{(\Omega^f)^k}
&=T_{-k}T_{-f}[T_fT_kz_1,\ldots,T_fT_kz_n]_{\Omega}\\
&=T_{-(f+k)}[T_{f+k}z_1,\ldots,T_{f+k}z_n]_{\Omega}.
\end{align*}
This is the bracket of $\Omega^{f+k}$. Equation
\eqref{eq:gauge-bracket} also gives
$T_f[z_1,\ldots,z_n]_{\Omega^f}
 =[T_fz_1,\ldots,T_fz_n]_{\Omega}$.
Together with $T_f|_{V}=\id$ and $pT_f=p$, this proves the last
assertion.
\end{proof}

\subsection{Classification}
\begin{definition}\label{def:Hnab}
Two non-abelian cocycles $\Omega,\Omega'$ are equivalent if
$\Omega'=\Omega^f$ for some $f\in\Hom(\g,V)$. Put
\[
H^2_{\nab}(\g,V)=Z^2_{\nab}(\g,V)/{\sim}.
\]
This is a set of equivalence classes.
\end{definition}

\begin{theorem}\label{thm:classification}
The assignment $\mathcal E\mapsto[\Omega_s]$ gives a bijection
\begin{equation}\label{eq:classification}
\Ext_{\nab}(\g,V)\ \cong\ H^2_{\nab}(\g,V).
\end{equation}
Its inverse sends $[\Omega]$ to the extension with middle algebra
$\g\oplus_{\Omega}V$, injection $a\mapsto0+a$, and projection
$x+a\mapsto x$.
\end{theorem}

This is the established classification of non-abelian extensions
by crossed-product data modulo changes of section; see
\cite{ZhangExtending,AfiBasdouri}, and
\cite{SongMakhloufTang,Zhang3Lie} for the ternary case.
The equivalences in these coordinates are precisely the shears
$T_f$ from \eqref{eq:gauge-bracket}.

\begin{corollary}\label{cor:extension-split}
The extension $\mathcal E$ has an $n$-Lie algebra section if and
only if $\omega^f=0$ for some $f\in\Hom(\g,V)$.
Equivalently,
\begin{equation}\label{eq:section-split-equation}
f([x_1,\ldots,x_n]_{\g})
 =\omega(x_1,\ldots,x_n)
  +\sum_{i=1}^n
 [x_1,\ldots,x_{i-1},f(x_i),
  x_{i+1}+f(x_{i+1}),\ldots,x_n+f(x_n)]_{\Omega}.
\end{equation}
\end{corollary}

\begin{proof}
Every linear section is $s+f$. Directly,
\[
[(s+f)x_1,\ldots,(s+f)x_n]-(s+f)[x_1,\ldots,x_n]_{\g}
 =\omega^f(x_1,\ldots,x_n).
\]
Its vanishing is exactly preservation of the bracket.
Substitution of \eqref{eq:gauge-omega-telescope} gives
\eqref{eq:section-split-equation}.
\end{proof}

\section{Centers, outer actions, and the third cohomology obstruction}
\label{sec:centers-obstruction}

\subsection{Inner and outer derivations of the kernel}
For an $n$-Lie algebra $V$, put
\begin{align}
Z(V)&=\{a\in V:[a,u_2,\ldots,u_n]_V=0
                    \text{ for all }u_2,\ldots,u_n\in V\},\notag\\
\Inn(V)&=\operatorname{span}\{\ad_V(u_1,\ldots,u_{n-1}):u_i\in V\},
\qquad
\Out(V)=\Der(V)/\Inn(V).
\label{eq:inner-outer}
\end{align}
Both derivation spaces in the quotient are ordinary Lie algebras
with commutator bracket. In particular, $\Out(V)$ is a quotient,
not an equality between $\Der(V)$ and $\Inn(V)$.

\begin{proposition}\label{prop:outer-center}
The space $\Inn(V)$ is an ideal of $\Der(V)$.
The center $Z(V)$ is invariant under $\Der(V)$, and the action on
$Z(V)$ factors through $\Out(V)$.
\end{proposition}

\begin{proof}
For $D\in\Der(V)$ and $v\in V$,
\begin{align*}
[D,\ad_V(u_1,\ldots,u_{n-1})]v
&=D[u_1,\ldots,u_{n-1},v]_V
                       -[u_1,\ldots,u_{n-1},Dv]_V\\
&=\sum_{i=1}^{n-1}
 [u_1,\ldots,Du_i,\ldots,u_{n-1},v]_V.
\end{align*}
This is a sum of inner derivations.
For $a\in Z(V)$,
\[
0=D[a,u_2,\ldots,u_n]_V
 =[Da,u_2,\ldots,u_n]_V
   +\sum_{j=2}^n[a,u_2,\ldots,Du_j,\ldots,u_n]_V
 =[Da,u_2,\ldots,u_n]_V.
\]
Thus $Da\in Z(V)$, whereas every inner derivation kills $a$.
The asserted factorization follows.
\end{proof}

In an extension, every contraction of the mixed maps gives a
derivation of $V$:
\begin{align}\label{eq:mixed-contractions}
v\longmapsto
\rho_{k+1}(x_1,\ldots,x_{n-1-k})(u_1,\ldots,u_k,v)
&=[s(x_1),\ldots,s(x_{n-1-k}),u_1,\ldots,u_k,v],
\notag\\[-2pt]
&\hspace{22mm}0\leq k\leq n-1.
\end{align}
The assertion follows by applying the Filippov identity to an
inner bracket with all entries in $V$.
Only the case $k=n-1$ is necessarily inner in $V$.

Define two Lie algebras of operators on $V$:
\begin{align}
\mathfrak D_E(V)
 &=\operatorname{span}\{\ad_E(z_1,\ldots,z_{n-1})|_V:z_i\in E\},
\notag\\
\mathfrak I_E(V)
 &=\operatorname{span}\{\ad_E(z_1,\ldots,z_{n-1})|_V:
                                \text{some }z_i\in V\}.
\label{eq:extension-outer-operators}
\end{align}
For $n=2$, $\mathfrak I_E(V)=\Inn(V)$.
For $n>2$, this equality need not hold.

\begin{proposition}\label{prop:relative-outer-action}
The space $\mathfrak I_E(V)$ is an ideal of $\mathfrak D_E(V)$.
The map
\begin{equation}\label{eq:relative-outer-action}
\overline\rho_E:\bigwedge^{n-1}\g\longrightarrow
\mathfrak D_E(V)/\mathfrak I_E(V),\qquad
X\longmapsto\rho(X)+\mathfrak I_E(V),
\end{equation}
is independent of the section and preserves the fundamental
Leibniz bracket:
\[
[\overline\rho_E(X),\overline\rho_E(Y)]
 =\overline\rho_E([X,Y]_{\mathrm F}),\qquad
[X,Y]_{\mathrm F}
 =\sum_{i=1}^{n-1}
 (y_1,\ldots,[X,y_i]_\g,\ldots,y_{n-1}).
\]
Its coefficient space is
\begin{equation}\label{eq:annihilator-mixed-inner}
A_\rho=\{a\in V:Da=0,\, \forall D\in\mathfrak I_E(V)\}.
\end{equation}
\end{proposition}

\begin{proof}
For $U=(z_1,\ldots,z_{n-1})$ and
$W=(w_1,\ldots,w_{n-1})$, the Filippov identity gives
\[
[\ad_E(U),\ad_E(W)]
 =\sum_{i=1}^{n-1}
 \ad_E(w_1,\ldots,[U,w_i],\ldots,w_{n-1}).
\]
Restriction to $V$ is legitimate because $V$ is an ideal.
If one $w_i$ belongs to $V$, each summand retains a kernel entry,
including the summand in which that entry is replaced.
This proves the ideal assertion.
Replacing $s(x_i)$ by $s(x_i)+f(x_i)$ in $\rho(X)$ adds only
operators with a kernel entry. Hence its coset is unchanged.
Every curvature term on the right of \eqref{eq:curvature} also
has a kernel entry among its $n-1$ operator arguments, so the
curvature is zero in the quotient.
Finally, the operators in $\mathfrak I_E(V)$ kill $a$ exactly
when $[a,b,z_3,\ldots,z_n]=0$ for all indicated entries.
\end{proof}

Thus the quotient in \eqref{eq:relative-outer-action} acts on
$A_\rho$. If
\begin{equation}\label{eq:inner-mixed-hypothesis}
\mathfrak I_E(V)=\Inn(V),
\end{equation}
then $A_\rho=Z(V)$ and \eqref{eq:relative-outer-action} can be
viewed as a section-independent map into $\Out(V)$.
Even then, for $n>2$ the other mixed maps must still be specified.
The weaker condition
\begin{equation}\label{eq:center-compatible}
[Z(V),V,E,\ldots,E]=0
\end{equation}
already suffices for $A_\rho=Z(V)$ and for a canonical
$\g$-representation on the entire center.

\begin{example}[Why $\Out(V)$ alone does not define the coupling]
\label{ex:outer-section-failure}
Let $n\geq3$, let $S$ be a centerless non-abelian $n$-Lie algebra,
and let $N$ have basis $t_1,\ldots,t_{n-2},p,q$ with
$[t_1,\ldots,t_{n-2},p,q]=q$.
This is an $n$-Lie algebra: write its bracket as
$[v_1,\ldots,v_n]=\det(v_1,\ldots,v_n)q$.
Every inner operator is $D=q\otimes\ell$, and
$Dq=\ell(q)q=(\tr D)q$; hence
\[
D[v_1,\ldots,v_n]
=\det(v_1,\ldots,v_n)Dq
=\sum_i[v_1,\ldots,Dv_i,\ldots,v_n].
\]
Adjoin a central element $y$ and set
\[
E=S\oplus N\oplus\F y,\qquad
V=S\oplus\operatorname{span}\{p,q\},\qquad
\g=E/V.
\]
The kernel is an ideal, $Z(V)=\operatorname{span}\{p,q\}$,
and $\g$ is abelian with basis
$\bar t_1,\ldots,\bar t_{n-2},\bar y$.
For the evident section, $\rho_1=0$, whereas
\[
\rho_2(\bar t_1,\ldots,\bar t_{n-2})(p,q)=q,\qquad A_\rho=0.
\]
Change the section only at $\bar y$, taking
$s'(\bar y)=y+p$. Then
\[
\rho_1^{s'}(\bar t_1,\ldots,\bar t_{n-2},\bar y)q=q.
\]
Every inner derivation of $V$ kills $Z(V)$.
Thus this change of $\rho_1$ is nonzero even modulo $\Inn(V)$.
The original zero outer map and the new nonzero outer map arise
from the same extension. The coefficient module and relative
outer quotient in Propositions~\ref{prop:relative-outer-action}
and~\ref{prop:central-representation} avoid this dependence.
\end{example}

\subsection{The relative central coefficient space}
Define
\begin{align}
A_\rho
 &=\{a\in V:[a,b,z_3,\ldots,z_n]_{E}=0
                    \text{ for all }b\in V,\ z_j\in E\}\notag\\
 &=\{a\in V:\rho_r(x_1,\ldots,x_{n-r})
             (a,b_2,\ldots,b_r)=0,\quad 2\leq r\leq n\}.
\label{eq:relative-center}
\end{align}
For $n=2$ the first condition means $[a,V]=0$.
Thus $A_\rho\subseteq Z(V)$, but in higher arity it can be
strictly smaller than $Z(V)$.
The first description shows that $A_\rho$ does not depend on
the section. Here and below $Z(V)=\{a\in V:[a,V,\ldots,V]=0\}$.

\begin{proposition}\label{prop:central-representation}
The map
\[
\rho_A(x_1,\ldots,x_{n-1})
 =\rho(x_1,\ldots,x_{n-1})|_{A_\rho}
\]
is a representation of $\g$ on $A_\rho$, independent of the
section.
\end{proposition}

\begin{proof}
Let $a\in A_\rho$, $b\in V$, and $z_3,\ldots,z_n\in E$.
With $S=(s(x_1),\ldots,s(x_{n-1}))$, the Filippov identity gives
\begin{align*}
0=[S,[a,b,z_3,\ldots,z_n]]
={}&[[S,a],b,z_3,\ldots,z_n]
 +[a,[S,b],z_3,\ldots,z_n]\\
&+\sum_{j=3}^n[a,b,z_3,\ldots,[S,z_j],\ldots,z_n].
\end{align*}
The last terms vanish by \eqref{eq:relative-center}, since
$[S,b]\in V$. Therefore $[S,a]\in A_\rho$.

The curvature formula \eqref{eq:curvature}, evaluated on $a$, has
zero right-hand side: every term contains $a$ and the kernel
element $\omega(X,y_i)$. Hence
\[
[\rho_A(X),\rho_A(y_1,\ldots,y_{n-1})]a
 =\sum_{i=1}^{n-1}
   \rho_A(y_1,\ldots,[X,y_i],\ldots,y_{n-1})a.
\]
Equation \eqref{eq:NAr0} with $r=1$ gives
\begin{align*}
&-\rho(x_1,\ldots,x_{n-2},[y_1,\ldots,y_n]_{\g})a
 +\rho_2(x_1,\ldots,x_{n-2})(a,\omega(y_1,\ldots,y_n))\\
&=\sum_{i=1}^n(-1)^{n-i+1}
 \rho(y_1,\ldots,\widehat y_i,\ldots,y_n)
 \rho(x_1,\ldots,x_{n-2},y_i)a.
\end{align*}
The $\rho_2$-term is zero. Multiplication by $-1$ yields the
second representation identity.
Finally, in
$[s(x_1)+f(x_1),\ldots,s(x_{n-1})+f(x_{n-1}),a]$,
every additional term contains both $a$ and a kernel entry
$f(x_i)$, and is zero. This proves section independence.
\end{proof}

We use extension degrees one and two. For a linear map
$f:\g\to V$ set
\begin{equation}\label{eq:relative-d1}
(d_\rho f)(x_1,\ldots,x_n)
 =\sum_{i=1}^n(-1)^{n-i}
 \rho(x_1,\ldots,\widehat x_i,\ldots,x_n)f(x_i)
 -f([x_1,\ldots,x_n]_{\g}).
\end{equation}
For an alternating $V$-valued $n$-linear map $\theta$, set
\begin{align}
(d_\rho\theta)(X;Y)
={}&\rho(X)\theta(Y)+\theta(X,[y_1,\ldots,y_n]_{\g})\notag\\
&-\sum_{i=1}^n\theta(y_1,\ldots,[X,y_i],\ldots,y_n)\notag\\
&-\sum_{i=1}^n(-1)^{n-i}
 \rho(y_1,\ldots,\widehat y_i,\ldots,y_n)\theta(X,y_i).
\label{eq:relative-d2}
\end{align}
On $A_\rho$-valued cochains these are the usual alternating
extension differentials for $\rho_A$; in particular $d_\rho^2f=0$
for $f:\g\to A_\rho$. The standard cohomology construction is
recalled from \cite{DaletskiiTakhtajan,Takhtajan,CasasLodayPirashvili}.
For arbitrary $V$-valued maps, \eqref{eq:relative-d1}--\eqref{eq:relative-d2}
are formulas only; no representation of $\g$ on all of $V$
is being assumed.
Write
\[
Z_\rho^2=\{\theta\in\Hom(\bigwedge^n\g,A_\rho):d_\rho\theta=0\}.
\]

\subsection{The representation class of the highest mixed map}
\label{sec:highest-representation}
The highest mixed operation has an additional intrinsic
interpretation, valid for full data in every arity.
For $U=(u_1,\ldots,u_{n-1})$ put $\ad_V(U)v=[U,v]_V$, and define
\begin{equation}\label{eq:chi-theta}
\chi_s(U)x=(-1)^{n-1}\rho_{n-1}(x)(U).
\end{equation}
Let $Z_{\mathrm{mod}}^1(V;\g)$ be the space of alternating maps
$q:\bigwedge^{n-1}V\to\Hom(\g,V)$ satisfying
\begin{align}
\ad_V(U)q(W)-\ad_V(W)q(U)
 &=\sum_i q(w_1,\ldots,[U,w_i]_V,\ldots,w_{n-1}),
\label{eq:theta-module1}\\
q(u_1,\ldots,u_{n-2},[v_1,\ldots,v_n]_V)
 &=\sum_i(-1)^{n-i}\ad_V(v_1,\ldots,\widehat v_i,\ldots,v_n)
                   q(u_1,\ldots,u_{n-2},v_i).
\label{eq:theta-module2}
\end{align}
The equalities are in $\Hom(\g,V)$.
For $f:\g\to V$ define
\[
(\partial_V f)(U)x=[U,f(x)]_V,\qquad
H_{\mathrm{mod}}^1(V;\g)
 =Z_{\mathrm{mod}}^1(V;\g)/\partial_V\Hom(\g,V).
\]
\begin{proposition}\label{prop:theta-module-class}
The maps
\[
\rho_V^s(U)(s(x)+v)
 =[U,v]_V+\chi_s(U)x
\]
form a representation of $V$ on $E$.
The sequence of $V$-representations
\begin{equation}\label{eq:theta-module-extension}
0\longrightarrow V_{\ad}\longrightarrow E_{\rho_V^s}
\longrightarrow\g_{\mathrm{triv}}\longrightarrow0
\end{equation}
has extension cocycle $\chi_s\in Z_{\mathrm{mod}}^1(V;\g)$.
For every section change,
\begin{equation}\label{eq:theta-module-boundary}
\chi_{s+f}=\chi_s+\partial_V f.
\end{equation}
The class $[\chi_s]$ is intrinsic, and it is zero
precisely when \eqref{eq:theta-module-extension} has a
$V$-equivariant linear section.
\end{proposition}

\begin{proof}
Directly in $E$,
\[
[u_1,\ldots,u_{n-1},s(x)+v]
 =[U,v]_V+(-1)^{n-1}[s(x),u_1,\ldots,u_{n-1}].
\]
Thus $\rho_V^s$ is the restriction of the adjoint
representation of $E$ to $V$. Its restriction to $V$ is
$\ad_V$, and its induced action on $\g$ is zero.
To see its cocycle equations explicitly, evaluate the first
representation identity at $s(x)$:
\begin{align*}
[\rho_V^s(U),\rho_V^s(W)]s(x)
 &=\ad_V(U)\chi_s(W)x-\ad_V(W)\chi_s(U)x,\\
\sum_i\rho_V^s(w_1,\ldots,[U,w_i]_V,\ldots,w_{n-1})s(x)
 &=\sum_i\chi_s(w_1,\ldots,[U,w_i]_V,\ldots,w_{n-1})x.
\end{align*}
The second representation identity evaluated at $s(x)$ gives
\[
\chi_s(u_1,\ldots,u_{n-2},[v_1,\ldots,v_n]_V)x
 =\sum_i(-1)^{n-i}\ad_V(v_1,\ldots,\widehat v_i,\ldots,v_n)
                       \chi_s(u_1,\ldots,u_{n-2},v_i)x.
\]
These are \eqref{eq:theta-module1}--\eqref{eq:theta-module2}.
Conversely, these two equations for $q$, together with the
adjoint identities on $V$, verify both representation identities
for $v+x\mapsto\ad_V(U)v+q(U)x$.
Thus the displayed cocycle space describes exactly these
extensions of representations.

For $s'=s+f$, direct evaluation gives
\[
[U,s'(x)]=[U,s(x)]+[U,f(x)]_V,
\]
which proves \eqref{eq:theta-module-boundary}.
Taking $q=0$ and changing the section also proves that every
$\partial_V f$ satisfies the two cocycle identities.
Finally, $s+f$ is $V$-equivariant exactly when
$[U,s(x)+f(x)]=0$ for all $U,x$, that is,
$\chi_s+\partial_V f=0$.
This proves the last assertion and the classification modulo
the displayed boundaries.
\end{proof}
An equivariant section here need not be an algebra section.
The class concerns the underlying representation extension;
the other mixed maps and the pure component remain part of the
algebra extension datum.

Put
$\mathcal L=\bigwedge^{n-1}\g$ and
\[
\widetilde C^p(\g,A_\rho)
 =\Hom(\mathcal L^{\otimes p}\otimes\g,A_\rho),\qquad p\geq0.
\]
For $X_j=(x_{j1},\ldots,x_{j,n-1})$, set
\[
[X_j,X_k]_{\mathrm F}
 =\sum_{\ell=1}^{n-1}
 (x_{k1},\ldots,[X_j,x_{k\ell}],\ldots,x_{k,n-1}).
\]
If $X_p=(u_1,\ldots,u_{n-1})$, the standard differential on
$F\in\widetilde C^{p-1}$ is
\begin{align}
(\widetilde dF)(X_1,\ldots,X_p,z)
={}&\sum_{j<k}(-1)^j
 F(X_1,\ldots,\widehat X_j,\ldots,X_{k-1},
        [X_j,X_k]_{\mathrm F},X_{k+1},\ldots,X_p,z)\notag\\
&+\sum_j(-1)^j F(X_1,\ldots,\widehat X_j,\ldots,X_p,[X_j,z])
 \notag\\
&+\sum_j(-1)^{j+1}\rho_A(X_j)
                 F(X_1,\ldots,\widehat X_j,\ldots,X_p,z)\notag\\
&+\sum_{i=1}^{n-1}(-1)^{n+p-i+1}
 \rho_A(u_1,\ldots,\widehat u_i,\ldots,u_{n-1},z)
 F(X_1,\ldots,X_{p-1},u_i).
\label{eq:full-relative-differential}
\end{align}
The identity $\widetilde d^{\,2}=0$ is the established
cochain identity in
\cite{DaletskiiTakhtajan,Takhtajan}.

\subsection{Extension degrees and alternating primitives}
The preceding formulas define the coefficient complex in the
degrees needed here. We state the convention explicitly because
an arbitrary cochain in
$\Hom(\bigwedge^{n-1}\g\otimes\g,A_\rho)$ need not define an
alternating $n$-ary bracket.
For $n\geq3$ and a representation $\sigma$ on a vector space $A$,
put
\begin{align}
C_{\ext}^1(\g,A)&=\Hom(\g,A),&
C_{\ext}^2(\g,A)&=\Hom(\bigwedge^n\g,A),\notag\\
C_{\ext}^3(\g,A)&=\Hom(\bigwedge^{n-1}\g\otimes\bigwedge^n\g,A),
&
C_{\ext}^q(\g,A)&=\widetilde C^{\,q-1}(\g,A)\quad(q\geq4).
\label{eq:extension-complex}
\end{align}
The maps in the first three spaces are regarded as fundamental
cochains by keeping the displayed entries in order.
Use $d_\sigma$ from \eqref{eq:relative-d1}--\eqref{eq:relative-d2}
in the first two degrees, and $\widetilde d$ thereafter.
Here \eqref{eq:full-relative-differential} is used with $\rho_A$
replaced by $\sigma$. Its square is zero.
The first differential has alternating output, and the second is
alternating in the outer $n-1$ and inner $n$ entries.
Thus \eqref{eq:extension-complex} is a cochain complex.
For $n=2$ we use the usual alternating Chevalley--Eilenberg
complex instead; the same displayed differential restricts to it.
Write $H_{\ext}^q(\g,A;\sigma)$ for these cohomology spaces.
In particular,
\begin{align}
H_{\ext}^2(\g,A;\sigma)
 &=H_{\ab}^2(\g,A;\sigma),\notag\\
H_{\ext}^3(\g,A;\sigma)
 &=\frac{\{\kappa\in C_{\ext}^3(\g,A):\widetilde d\kappa=0\}}
 {d_\sigma\Hom(\bigwedge^n\g,A)}.
\label{eq:H3-definition}
\end{align}
For $n=2$ the numerator in the second line consists of alternating
Chevalley--Eilenberg $3$-cocycles.
The degrees $2,3$ used here correspond to fundamental-cochain
degrees $1,2$ in the standard $n$-Lie convention. In particular,
the extension degree is shifted by one from that used
in~\cite{TanXu,Xu}.

There is a natural map from $H_{\ext}^3$ to
$\widetilde H^2$ for $n\geq3$. Vanishing after this map is not,
without an additional argument, enough to remove an alternating
bracket defect: its primitive in $\widetilde C^1$ might not be
alternating in all $n$ entries. Definition~\eqref{eq:H3-definition}
retains exactly the required primitives.

\subsection{A finite identity for the Filippov defect}
We need an identity which does not assume the Filippov identity.
Let $W$ carry any alternating $n$-linear operation. For ordered
$(n-1)$-tuples $X,Y$ define
\[
D_Xz=[X,z],\qquad
X\circ Y=\sum_{i=1}^{n-1}
(y_1,\ldots,[X,y_i],\ldots,y_{n-1}).
\]
Operations with a sum of tuples are interpreted linearly.
Set
\begin{align}
J(X;Y,z)
 &=[X,[Y,z]]-[Y,[X,z]]
       -\sum_{i=1}^{n-1}
 [y_1,\ldots,[X,y_i],\ldots,y_{n-1},z],\notag\\
Q(X,Y)&=[D_X,D_Y]-D_{X\circ Y}.
\label{eq:universal-FI-defect}
\end{align}
Thus $Q(X,Y)z=J(X;Y,z)$.

\begin{lemma}[Finite Bianchi identity]\label{lem:finite-bianchi}
For $Z=(z_1,\ldots,z_{n-1})$ and $w\in W$,
\begin{align}
0={}&-J(X\circ Y;Z,w)-J(Y;X\circ Z,w)+J(X;Y\circ Z,w)
\notag\\
&-J(Y;Z,[X,w])+J(X;Z,[Y,w])-J(X;Y,[Z,w])
\notag\\
&+[X,J(Y;Z,w)]-[Y,J(X;Z,w)]+[Z,J(X;Y,w)]
\notag\\
&+\sum_{i=1}^{n-1}
[z_1,\ldots,z_{i-1},J(X;Y,z_i),z_{i+1},\ldots,z_{n-1},w].
\label{eq:finite-bianchi}
\end{align}
The identity is valid for every alternating operation.
\end{lemma}

\begin{proof}
Expanding the nested insertions in $X\circ(Y\circ Z)$, those
involving two different slots of $Z$ cancel with the corresponding
terms of $Y\circ(X\circ Z)$. In slot $i$ the remaining term is
\[
[X,[Y,z_i]]-[Y,[X,z_i]]
  -\sum_{j=1}^{n-1}
  [y_1,\ldots,[X,y_j],\ldots,y_{n-1},z_i]
=Q(X,Y)z_i.
\]
Consequently,
\begin{equation}\label{eq:tuple-defect}
X\circ(Y\circ Z)-Y\circ(X\circ Z)-(X\circ Y)\circ Z
 =\sum_{i=1}^{n-1}
 (z_1,\ldots,Q(X,Y)z_i,\ldots,z_{n-1}).
\end{equation}
The right side of \eqref{eq:finite-bianchi}, before application
to $w$, is the operator
\begin{align*}
&[D_X,Q(Y,Z)]-[D_Y,Q(X,Z)]+[D_Z,Q(X,Y)]\\
&\quad-Q(X\circ Y,Z)-Q(Y,X\circ Z)+Q(X,Y\circ Z)\\
&\quad+D_{\,X\circ(Y\circ Z)-Y\circ(X\circ Z)-(X\circ Y)\circ Z}.
\end{align*}
Substitute $Q(U,T)=[D_U,D_T]-D_{U\circ T}$.
The terms involving one composed tuple cancel in pairs:
\begin{align*}
&-[D_X,D_{Y\circ Z}]+[D_X,D_{Y\circ Z}]=0,\\
&[D_Y,D_{X\circ Z}]-[D_Y,D_{X\circ Z}]=0,\\
&-[D_Z,D_{X\circ Y}]-[D_{X\circ Y},D_Z]=0,\\
&D_{(X\circ Y)\circ Z}+D_{Y\circ(X\circ Z)}
 -D_{X\circ(Y\circ Z)}
 +D_{X\circ(Y\circ Z)-Y\circ(X\circ Z)-(X\circ Y)\circ Z}=0.
\end{align*}
The remaining operator is
\[
[D_X,[D_Y,D_Z]]-[D_Y,[D_X,D_Z]]+[D_Z,[D_X,D_Y]]=0.
\]
For example its terms $D_XD_YD_Z$ and $-D_XD_YD_Z$
cancel, and the other five orders cancel in the same pairs.
Applying this equality to $w$ proves the formula.
\end{proof}

\subsection{The obstruction to completing mixed data}
\begin{definition}\label{def:pre-extension}
A mixed family $\rho_1,\ldots,\rho_{n-1}$ is
\emph{admissible} if there exists an alternating
$\omega_0:\bigwedge^n\g\to V$ satisfying
\eqref{eq:NA1}, \eqref{eq:NAr0}, and \eqref{eq:NArs}.
The resulting operation on $\g\oplus V$ is called a
\emph{pre-extension}. The pure equation \eqref{eq:NA0}
is not required.
\end{definition}
This is an existence condition for the mixed equations, not an
automatic property of an arbitrary map into $\Out(V)$.
In a pre-extension the Filippov defect vanishes if any of its
$2n-1$ arguments belongs to $V$.
The proof of Proposition~\ref{prop:central-representation} uses
only such identities, so $A_\rho$ and its representation $\rho_A$
are already defined for a pre-extension.
The same observation applies to
Proposition~\ref{prop:relative-outer-action}.

\begin{theorem}[Third cohomology obstruction]\label{thm:third-obstruction}
Let $\rho_1,\ldots,\rho_{n-1}$ be admissible, and choose
$\omega_0$ as in Definition~\ref{def:pre-extension}. Then
\begin{align}
\kappa_{\rho,\omega_0}(X;y_1,\ldots,y_n)
={}&\rho(X)\omega_0(y_1,\ldots,y_n)
       +\omega_0(X,[y_1,\ldots,y_n]_\g)\notag\\
&-\sum_{i=1}^n\omega_0(y_1,\ldots,[X,y_i]_\g,\ldots,y_n)
\notag\\
&-\sum_{i=1}^n(-1)^{n-i}
 \rho(y_1,\ldots,\widehat y_i,\ldots,y_n)\omega_0(X,y_i)
\label{eq:third-cocycle}
\end{align}
takes values in $A_\rho$ and is a $3$-cocycle in
\eqref{eq:H3-definition}. Its class
\begin{equation}\label{eq:third-obstruction}
\mathfrak o(\rho_1,\ldots,\rho_{n-1})
=[\kappa_{\rho,\omega_0}]
\in H_{\ext}^3(\g,A_\rho;\rho_A)
\end{equation}
is independent of $\omega_0$.
There exists an $n$-Lie extension with these mixed maps if and
only if this class is zero.
\end{theorem}

\begin{proof}
Apply Lemma~\ref{lem:finite-bianchi} to the pre-extension.
On pure section inputs its defect is exactly
\[
J(s(X);s(y_1),\ldots,s(y_n))
=\kappa_{\rho,\omega_0}(X;y_1,\ldots,y_n).
\]
The quotient component is zero by the Filippov identity in $\g$.
To prove centrality, take the first outer tuple in
\eqref{eq:finite-bianchi} to be
$U=(u_1,\ldots,u_{n-1})$ with at least one $u_i\in V$;
take the other entries to be section elements.
Every defect containing $U$ is zero.
Every summand of $U\circ s(Y)$ and $U\circ s(Z)$ has a kernel
entry because $V$ is an ideal.
Also $[U,s(w)]\in V$. Thus all terms except one vanish:
\[
0=[U,J(s(Y);s(Z),s(w))]
 =[U,\kappa_{\rho,\omega_0}(Y;Z,w)].
\]
Allowing arbitrary other entries of $U$, and using
\eqref{eq:relative-center}, proves that $\kappa$ takes values in
$A_\rho$.

Now make all entries in \eqref{eq:finite-bianchi} section
elements. A kernel component inside any argument of $J$ gives
zero. An outer bracket on a value of $J$ is the action $\rho_A$.
Consequently \eqref{eq:finite-bianchi} becomes
\begin{align*}
0={}&-\kappa([X,Y]_{\mathrm F};Z,w)
     -\kappa(Y;[X,Z]_{\mathrm F},w)
     +\kappa(X;[Y,Z]_{\mathrm F},w)\\
&-\kappa(Y;Z,[X,w]_\g)
  +\kappa(X;Z,[Y,w]_\g)-\kappa(X;Y,[Z,w]_\g)\\
&+\rho_A(X)\kappa(Y;Z,w)-\rho_A(Y)\kappa(X;Z,w)
  +\rho_A(Z)\kappa(X;Y,w)\\
&+\sum_{i=1}^{n-1}(-1)^{n-i}
 \rho_A(z_1,\ldots,\widehat z_i,\ldots,z_{n-1},w)
                  \kappa(X;Y,z_i).
\end{align*}
This is $\widetilde d\kappa=0$.
The defect is alternating in its outer $n-1$ entries and its
inner $n$ entries. For $n=2$ it is the alternating Jacobi
defect. Hence it lies in the required cochain space.

Let $\omega_1$ be another solution of the mixed equations and
put $\eta=\omega_1-\omega_0$.
Subtracting \eqref{eq:NAr0} for the two choices gives
\[
\rho_{r+1}(x_1,\ldots,x_{n-1-r})
 (a_1,\ldots,a_r,\eta(y_1,\ldots,y_n))=0,
\qquad1\leq r\leq n-1.
\]
Therefore $\eta$ is $A_\rho$-valued.
Conversely any alternating $A_\rho$-valued $\eta$ preserves all
mixed equations: its additional terms in \eqref{eq:NA1} and
\eqref{eq:NAr0} vanish by the definition of $A_\rho$, and
\eqref{eq:NArs} does not contain $\omega_0$.
The possible choices are consequently
\[
\omega_0+\Hom(\bigwedge^n\g,A_\rho),\qquad
\kappa_{\rho,\omega_0+\eta}
 =\kappa_{\rho,\omega_0}+d_{\rho_A}\eta.
\]
This proves independence of the class.
If $[\kappa]=0$, choose an alternating $\eta$ with
$d_{\rho_A}\eta=\kappa$. The datum with pure component
$\omega_0-\eta$ satisfies the mixed equations and has zero pure
defect. Theorem~\ref{thm:cocycle-criterion} makes it an extension.
Conversely, a completing pure component has zero defect, so the
same difference formula makes $[\kappa]=0$.
\end{proof}

\begin{proposition}[Section invariance]\label{prop:third-gauge}
The class \eqref{eq:third-obstruction} depends only on the orbit
of the mixed family under changes of section.
Under the natural identifications, its representative is unchanged
when the pre-extension is transported by $T_f$.
\end{proposition}

\begin{proof}
Formula \eqref{eq:gauge-bracket} applies to any alternating
operation. Substitution in the two nested brackets defining $J$
gives
\[
J^{\,f}(z_1,\ldots,z_{n-1};w_1,\ldots,w_n)
 =T_{-f}J(T_fz_1,\ldots,T_fz_{n-1};T_fw_1,\ldots,T_fw_n).
\]
For a mixed input this is zero. For pure quotient inputs every
term containing $f(z_i)$ or $f(w_j)$ is zero, and $T_{-f}$
fixes the value in $V$. Hence
\[
\kappa_{\rho^f,\omega_0^f}=\kappa_{\rho,\omega_0}.
\]
The condition defining $A_\rho$ concerns brackets with at least
two kernel entries and is preserved by $T_f$.
For $a\in A_\rho$, every additional term of
$[s(x_1)+f(x_1),\ldots,s(x_{n-1})+f(x_{n-1}),a]$
contains $a$ and another kernel entry. It vanishes, proving
$\rho_A^f=\rho_A$. The coefficient complexes agree, as claimed.
\end{proof}

We call such an orbit a \emph{mixed coupling}. This terminology
records precisely the equivalence just defined; it does not
replace it by a homomorphism into $\Out(V)$.
Once the mixed equations admit a solution, third cohomology is
the remaining existence obstruction. When it vanishes,
Theorem~\ref{thm:affine-classification} below identifies the
extension fiber with an affine space over $\mathcal H_\rho^2$.

\begin{corollary}\label{cor:center-obstruction}
If \eqref{eq:center-compatible} holds for a pre-extension, then
$A_\rho=Z(V)$ and
\[
\mathfrak o(\rho_1,\ldots,\rho_{n-1})
\in H_{\ext}^3(\g,Z(V);\rho_A).
\]
If this class vanishes, the equivalence classes with fixed mixed
maps form a torsor for $H_{\ext}^2(\g,Z(V);\rho_A)$.
If $Z(V)=0$, every admissible family has a unique completing
pure component. Consequently, each admissible coupling
has exactly one extension class.
\end{corollary}

\begin{proof}
Under \eqref{eq:center-compatible},
$\Hom(\g,Z(V))=\Hom(\g,A_\rho)$.
Every such map preserves the mixed brackets.
Conversely, the highest mixed equation forces the image of every
mixed-preserving correction into $Z(V)$; see
\eqref{eq:gauge-top}. Thus
$\mathcal F_\rho=\Hom(\g,A_\rho)$, and the second-cohomology
quotient of Theorem~\ref{thm:affine-classification} is the ordinary
alternating extension quotient.
If $Z(V)=0$, then $A_\rho=0$, so the defect in
Theorem~\ref{thm:third-obstruction} is zero and the completing
pure component is unique. Transport under section changes gives
the assertion for the coupling.
\end{proof}

For $n=2$, a mixed coupling is the usual outer action
$\g\to\Out(V)$, the mixed equation is
$
[\rho(x),\rho(y)]-\rho([x,y])=\ad_V\omega_0(x,y)
$, and \eqref{eq:third-cocycle} is its Chevalley--Eilenberg
$3$-cocycle. Thus the result recovers the classical
$H^3(\g,Z(V))$ obstruction and $H^2(\g,Z(V))$ torsor
of~\cite{AlekseevskyMichorRuppert}.
For $n>2$, mixed realizability and the full coupling are
essential additional data.
In particular, $\Der(V)=\Inn(V)$ means $\Out(V)=0$; it does not
remove the mixed equations from the hypotheses of our theorem.

\subsection{Naturality and the two obstruction stages}
Let $c=(\beta,\alpha)\in\Aut(V)\times\Aut(\g)$.
Transport of a pre-extension is defined by the same formulas
\eqref{eq:twist-omega}--\eqref{eq:twist-rho} as transport of an
extension. It gives
\begin{align}
A_{{}^c\rho}&=\beta A_\rho,\notag\\
({}^c\kappa)(X;y_1,\ldots,y_n)
 &=\beta\kappa(\alpha^{-1}X;
                       \alpha^{-1}y_1,\ldots,\alpha^{-1}y_n),\notag\\
\kappa_{{}^c\rho,{}^c\omega_0}&={}^c\kappa_{\rho,\omega_0}.
\label{eq:obstruction-naturality}
\end{align}
Indeed, applying $\beta$ to \eqref{eq:third-cocycle} and using
bracket preservation by $\alpha$ gives each term of the last
identity. The same substitution commutes with every term of
\eqref{eq:full-relative-differential}, so it induces the
corresponding isomorphism in third cohomology.
If $c$ preserves the mixed coupling, then
$ {}^c\rho=\rho^f$ for some $f$, and
Proposition~\ref{prop:third-gauge} gives
\begin{equation}\label{eq:obstruction-stabilizer}
c\cdot\mathfrak o(\rho)=\mathfrak o(\rho).
\end{equation}
Thus mixed symmetries preserve the existence obstruction.
When $\mathfrak o(\rho)=0$, one chooses a completing $\omega$.
The Wells obstructions in Sections~\ref{sec:derivation-wells}
and~\ref{sec:Wells} then decide whether a mixed symmetry
preserves this particular extension class.
The third-cohomology obstruction concerns existence of an
extension; the second-cohomology Wells obstruction concerns a
lift for an extension already chosen.

\begin{proposition}\label{prop:infinitesimal-H3}
Let $P$ preserve $V$ in a pre-extension, induce
$d=(D_V,D_\g)\in\Der(V)\oplus\Der(\g)$, and satisfy the
derivation identity on every mixed bracket. Write
$P(s(x))=s(D_\g x)+\lambda(x)$ and set
\[
\Theta=L_d\omega_0-d_\rho\lambda.
\]
Then $\Theta$ takes values in $A_\rho$ and
\begin{equation}\label{eq:infinitesimal-H3}
L_d\kappa_{\rho,\omega_0}=d_{\rho_A}\Theta.
\end{equation}
Here $L_d$ applies $D_V$ to the value and subtracts insertion
of $D_\g$ in each of the $2n-1$ quotient arguments.
In particular, the induced action of $d$ annihilates the
third-cohomology obstruction class.
\end{proposition}

\begin{proof}
Put $\eta(z_1,\ldots,z_n)
=P[z_1,\ldots,z_n]-\sum_i[z_1,\ldots,Pz_i,\ldots,z_n]$.
This map is zero on mixed tuples and its pure value is
$\Theta$. For $U=(u_1,\ldots,u_{n-1})$ and
$Y=(y_1,\ldots,y_n)$, applying $P$ to each nested bracket
in the definition of $J$ gives
\begin{align}
&(L_PJ)(U;Y)\notag\\
&=\eta(U,[Y])+[U,\eta(Y)]
-\sum_i\eta(y_1,\ldots,[U,y_i],\ldots,y_n)
-\sum_i[y_1,\ldots,\eta(U,y_i),\ldots,y_n].
\label{eq:defect-covariance}
\end{align}
Here $L_PJ$ means $PJ$ minus insertion of $P$ in each input.
For instance,
\[
P[U,[Y]]
=\sum_j[u_1,\ldots,Pu_j,\ldots,u_{n-1},[Y]]
 +[U,P[Y]]+\eta(U,[Y]),
\]
and
$P[Y]=\sum_i[y_1,\ldots,Py_i,\ldots,y_n]+\eta(Y)$.
The analogous expansion in each right-hand nested bracket
cancels precisely the input-insertion terms of $L_PJ$,
leaving \eqref{eq:defect-covariance}.
No Filippov identity for pure inputs has been used.

If $U$ has a kernel entry, the left side is zero since $P$
preserves $V$. On the right only $[U,\Theta(pY)]$ can remain.
Its vanishing proves $\Theta(pY)\in A_\rho$.
For pure section entries the left side is $L_d\kappa$:
every $\lambda$-insertion is a mixed defect and vanishes.
The right side is exactly $d_{\rho_A}\Theta$, proving
\eqref{eq:infinitesimal-H3}.
The mixed identities for $P$ also imply
$D_V(A_\rho)\subseteq A_\rho$ and
$[D_V,\rho_A(X)]=\rho_A(D_\g X)$, by the calculation of
Lemma~\ref{lem:compatible-module}.
Thus $L_d$ commutes with the coefficient differential and
acts on third cohomology. Equation \eqref{eq:infinitesimal-H3}
makes its value on $[\kappa]$ zero.
\end{proof}

\section{Section invariance and the non-abelian good criterion}
\label{sec:nonabelian-good}
Fix an actual extension and a linear section $s$.
The following finite calculations identify the allowed
corrections and then characterize goodness.
\subsection{Section changes preserving the mixed brackets}
For $f:\g\to V$ and $m=n-r$ define
\begin{equation}\label{eq:partial-f}
(\partial_r f)(x_1,\ldots,x_m)(a_1,\ldots,a_r)
 =\sum_{i=1}^{m}(-1)^{m-i}
 \rho_{r+1}(x_1,\ldots,\widehat x_i,\ldots,x_m)
       (f(x_i),a_1,\ldots,a_r),
\end{equation}
and put
\begin{equation}\label{eq:F-rho}
\mathcal F_\rho
 =\{f\in\Hom(\g,V):\partial_r f=0,\quad1\leq r\leq n-1\}.
\end{equation}
The equation with $r=n-1$ implies $f(\g)\subseteq Z(V)$.
The stronger requirement $f(\g)\subseteq A_\rho$ is sufficient
but is not necessary in general.
Let $\delta_f=i f p$, so that
$\delta_f(s(x)+a)=f(x)$ and $\delta_f^2=0$.

\begin{lemma}\label{lem:partial-shears}
The following conditions are equivalent:
\begin{enumerate}
\item $f\in\mathcal F_\rho$;
\item $\delta_f=ifp$ satisfies the derivation identity on every
      bracket with at least one argument in $V$;
\item for the section $s'=s+f$, one has $\rho_r^{s'}=\rho_r^s$
      for all $1\leq r\leq n-1$.
\end{enumerate}
For every such $f$,
\begin{equation}\label{eq:partial-gauge-affine}
\rho_r^f=\rho_r,\qquad \omega^f=\omega+d_\rho f,
\qquad d_\rho f\in Z_\rho^2.
\end{equation}
\end{lemma}

\begin{proof}
Write $\delta=\delta_f$. On a mixed bracket, $\delta$ annihilates
both its value and every kernel argument. Its remaining
derivation terms are
\begin{align*}
&\sum_{i=1}^{m}
[s(x_1),\ldots,f(x_i),\ldots,s(x_m),a_1,\ldots,a_r]\\
&\quad=\sum_{i=1}^{m}(-1)^{m-i}
\rho_{r+1}(x_1,\ldots,\widehat x_i,\ldots,x_m)
             (f(x_i),a_1,\ldots,a_r)=\partial_rf.
\end{align*}
This proves the equivalence of the first two conditions.

We give a finite multilinear calculation for the section change.
For fixed $z_1,\ldots,z_n$, let $S_k$ be the sum of the terms in
$[z_1+\delta z_1,\ldots,z_n+\delta z_n]$ with exactly $k$
entries replaced by $\delta z_i$, in their original slots.
Thus
\begin{align*}
S_0&=[z_1,\ldots,z_n],\\
S_1&=\sum_i[z_1,\ldots,\delta z_i,\ldots,z_n],\\
S_2&=\sum_{i<j}
 [z_1,\ldots,\delta z_i,\ldots,\delta z_j,\ldots,z_n],\\
[z_1+\delta z_1,\ldots,z_n+\delta z_n]&=S_0+S_1+\cdots+S_n.
\end{align*}
Suppose the mixed derivation identity holds.
Each term of $S_k$ with $k\geq1$ is mixed and has value in
$V$. Applying $\delta$ to that term gives zero on the left.
On the right, an already replaced slot contributes
$\delta^2z_i=0$. Every term with $k+1$ replaced slots occurs
exactly $k+1$ times when these identities are added. Consequently,
\[
0=\delta S_k=(k+1)S_{k+1}\quad(1\leq k<n),
\qquad S_2=\cdots=S_n=0.
\]
For an originally mixed tuple, the same calculation at $k=0$
also gives $S_1=\delta S_0=0$. Hence
\[
[z_1+\delta z_1,\ldots,z_n+\delta z_n]=[z_1,\ldots,z_n]
\]
on every mixed tuple, proving the third condition.

Conversely, assume the third condition. Then $T_f=\id+\delta$
preserves every mixed bracket and fixes its value in $V$.
For an originally mixed tuple, every summand of every $S_k$
is again mixed. Expanding the preservation identity for each
such summand and adding, one obtains
\begin{equation}\label{eq:finite-insertion-triangle}
0=\sum_{j=k+1}^{n}\binom jk S_j,
\qquad 0\leq k<n.
\end{equation}
Indeed, a term with $j$ replaced slots is obtained once for
each choice of the $k$ slots already replaced; hence its
coefficient is $\binom jk$.
For $k=n-1$, \eqref{eq:finite-insertion-triangle} gives $nS_n=0$.
Descending successively gives
$S_n=S_{n-1}=\cdots=S_1=0$, since the characteristic is zero.
Thus the mixed derivation identity holds.

Finally take $z_i=s(x_i)$ and use the calculation
$S_2=\cdots=S_n=0$ already proved for $f\in\mathcal F_\rho$.
For $s'=s+f$ we obtain directly
\begin{align*}
\omega^{f}(x_1,\ldots,x_n)
 &=[s(x_1)+f(x_1),\ldots,s(x_n)+f(x_n)]
                -(s+f)[x_1,\ldots,x_n]_\g\\
 &=\omega(x_1,\ldots,x_n)
   +\sum_i[s(x_1),\ldots,f(x_i),\ldots,s(x_n)]
   -f([x_1,\ldots,x_n]_\g)\\
 &=\omega(x_1,\ldots,x_n)
   +\sum_i(-1)^{n-i}\rho(x_1,\ldots,\widehat x_i,\ldots,x_n)f(x_i)
   -f([x_1,\ldots,x_n]_\g)\\
 &=\omega(x_1,\ldots,x_n)+(d_\rho f)(x_1,\ldots,x_n).
\end{align*}
The two section data are cocycles with the same mixed maps.
Proposition~\ref{prop:fixed-mixed-differences} therefore gives
$d_\rho f\in Z_\rho^2$.
\end{proof}

\subsection{The exact scope of section independence}
\begin{definition}\label{def:relative-good}
Fix a linear subspace $\mathcal S\subseteq\Hom(\g,V)$.
The section $s$ is \emph{good relative to $\mathcal S$} if
\[
\rho_r^{s+f}=\rho_r^s
\quad(f\in\mathcal S,\ 1\leq r\leq n-1),
\]
and $\rho=\rho_1^s$ is a representation of $\g$ on $V$.
It is \emph{centrally good} if
$\mathcal S=\Hom(\g,Z(V))$.
An extension is \emph{absolutely good} if the same conditions
hold for $\mathcal S=\Hom(\g,V)$.
\end{definition}
Relative goodness concerns the specified family of sections
$s+\mathcal S$. It does not assert invariance under changes
outside that family. For abelian $V$, absolute and central goodness agree.

\begin{proposition}[Non-abelian good criterion]
\label{prop:good-criterion}
Fix $s$, write $\rho_r=\rho_r^s$, and put $\omega=\omega_s$.
The section is good relative to $\mathcal S$ precisely when,
for every $f\in\mathcal S$ and all arguments,
\begin{align}
0&=\sum_{i=1}^{n-r}(-1)^{n-r-i}
\rho_{r+1}(x_1,\ldots,\widehat x_i,\ldots,x_{n-r})
       (f(x_i),a_1,\ldots,a_r),
\quad 1\leq r\leq n-1,
\label{eq:good-independence}\\
0&=\rho_2(x_1,\ldots,x_{n-2})
       (\omega(y_1,\ldots,y_n),a).
\label{eq:good-omega}
\end{align}
For $n=2$, \eqref{eq:good-omega} means
$[\omega(y_1,y_2),a]_V=0$.
Under these conditions, the direct section-change formula is
\begin{equation}\label{eq:good-direct-change}
\omega_{s+f}=\omega_s+d_\rho f,\qquad
\rho_r^{s+f}=\rho_r^s,\qquad f\in\mathcal S.
\end{equation}
\end{proposition}

\begin{proof}
Fix $f\in\mathcal S$ and set $s'=s+f$.
On a mixed tuple, multilinearity gives
\begin{align*}
&\rho_r^{s'}(x_1,\ldots,x_{n-r})(a_1,\ldots,a_r)\\
&=[s(x_1)+f(x_1),\ldots,s(x_{n-r})+f(x_{n-r}),
                                                    a_1,\ldots,a_r].
\end{align*}
The sum of terms with exactly one changed section entry is
\[
\sum_{i=1}^{n-r}(-1)^{n-r-i}
\rho_{r+1}(x_1,\ldots,\widehat x_i,\ldots,x_{n-r})
             (f(x_i),a_1,\ldots,a_r).
\]
If all these sums vanish, the map $\delta_f=fp$ satisfies
the derivation identity on all mixed brackets. The finite
insertion calculation in Lemma~\ref{lem:partial-shears}
then makes every sum with two or more changed entries vanish.
Therefore $\rho_r^{s'}=\rho_r^s$.
Conversely, if all mixed maps agree, expansion of their
preservation identities on tuples with successively more
entries $f(x_i)$ gives the triangular equations
\eqref{eq:finite-insertion-triangle}. Descending through
those equations makes the one-entry sum zero.
This proves that \eqref{eq:good-independence} is exactly
section independence along $\mathcal S$.
For the highest mixed map this calculation is simply
\begin{equation}\label{eq:highest-good-change}
\rho_{n-1}^{s+f}(x)(a_1,\ldots,a_{n-1})
-\rho_{n-1}^s(x)(a_1,\ldots,a_{n-1})
=[f(x),a_1,\ldots,a_{n-1}]_V.
\end{equation}
In particular every allowed $f$ must take values in $Z(V)$.

It remains to verify precisely when $\rho$ is a representation.
Let $X=(x_1,\ldots,x_{n-1})$ and
$Y=(y_1,\ldots,y_{n-1})$.
Apply the Filippov identity to
$s(x_1),\ldots,s(x_{n-1})$ and
$s(y_1),\ldots,s(y_{n-1}),a$.
Expanding
$[s(X),s(y_i)]=s([X,y_i]_\g)+\omega(X,y_i)$ gives
\begin{align*}
\rho(X)\rho(Y)a
={}&\sum_{i=1}^{n-1}
 \rho(y_1,\ldots,[X,y_i]_\g,\ldots,y_{n-1})a
 +\rho(Y)\rho(X)a\\
&+\sum_{i=1}^{n-1}(-1)^{n-1-i}
 \rho_2(y_1,\ldots,\widehat y_i,\ldots,y_{n-1})
                   (\omega(X,y_i),a).
\end{align*}
Consequently the first representation defect is
\begin{align}
&\left([\rho(X),\rho(Y)]
 -\sum_i\rho(y_1,\ldots,[X,y_i]_\g,\ldots,y_{n-1})\right)a\notag\\
&\quad=\sum_{i=1}^{n-1}(-1)^{n-1-i}
\rho_2(y_1,\ldots,\widehat y_i,\ldots,y_{n-1})
                   (\omega(X,y_i),a).
\label{eq:good-first-defect}
\end{align}
For the second defect use outer entries
$s(x_1),\ldots,s(x_{n-2}),a$ and inner entries
$s(y_1),\ldots,s(y_n)$.
The left side is
\[
-\rho(x_1,\ldots,x_{n-2},[y_1,\ldots,y_n]_\g)a
+\rho_2(x_1,\ldots,x_{n-2})(a,\omega(y_1,\ldots,y_n)).
\]
Its $i$th right-hand term is
\[
(-1)^{n-i+1}
\rho(y_1,\ldots,\widehat y_i,\ldots,y_n)
\rho(x_1,\ldots,x_{n-2},y_i)a.
\]
Moving $\omega(Y)$ past $a$ and multiplying by $-1$ yields
\begin{align}
&\left(\rho(x_1,\ldots,x_{n-2},[y_1,\ldots,y_n]_\g)
 -\sum_i(-1)^{n-i}\rho(y_1,\ldots,\widehat y_i,\ldots,y_n)
                    \rho(x_1,\ldots,x_{n-2},y_i)\right)a\notag\\
&\quad=-\rho_2(x_1,\ldots,x_{n-2})(\omega(y_1,\ldots,y_n),a).
\label{eq:good-second-defect}
\end{align}
Condition \eqref{eq:good-omega} makes both defects zero.
Conversely, the second representation identity makes the
right side of \eqref{eq:good-second-defect} zero, proving
\eqref{eq:good-omega}.

Finally, the same finite expansion on pure section inputs gives
\begin{align*}
\omega_{s'}(x_1,\ldots,x_n)-\omega_s(x_1,\ldots,x_n)
 &=\sum_i[s(x_1),\ldots,f(x_i),\ldots,s(x_n)]
                                      -f([x_1,\ldots,x_n]_\g)\\
 &=\sum_i(-1)^{n-i}
\rho(x_1,\ldots,\widehat x_i,\ldots,x_n)f(x_i)
                                      -f([x_1,\ldots,x_n]_\g).
\end{align*}
This is $d_\rho f$, proving \eqref{eq:good-direct-change}.
\end{proof}

\begin{corollary}\label{cor:absolute-good-abelian}
If $\g\ne0$ and the extension is absolutely good, then $V$ is
abelian. For abelian $V$, the highest equation in
\eqref{eq:good-independence} is automatic, so the criterion
reduces to the equations with $1\leq r\leq n-2$ together
with \eqref{eq:good-omega}.
\end{corollary}

\begin{proof}
Choose $0\ne x\in \g$. Given any $a_0\in V$, choose a linear
$f:\g\to V$ with $f(x)=a_0$. Equation
\eqref{eq:highest-good-change} and absolute goodness give
\[
[a_0,a_1,\ldots,a_{n-1}]_V=0
\]
for arbitrary $a_1,\ldots,a_{n-1}$. Thus the kernel bracket
is zero. The converse simplification follows by substituting
that zero bracket in the same equation.
\end{proof}

The hypothesis $\g\ne0$ is necessary: if $\g=0$, there is only
one section and the section-independence condition has no
content. In particular, extending the abelian criterion by
omitting the highest mixed equation would miss a genuine
non-abelian obstruction.

For example, a direct product $V\oplus \g$ with $V$ non-abelian
and $\g\ne0$ has $\rho_r=0$ for every mixed $r$ and $\omega=0$.
It satisfies both equations of the abelian criterion when
the range is truncated to $r\leq n-2$.
But \eqref{eq:highest-good-change} is nonzero for a suitable
$f$, so the extension is not absolutely good.

\section{Cohomology classes and the Lie algebra of compatible pairs}
\label{sec:compatible-derivations}

Fix an extension with mixed family $\rho_1,\ldots,\rho_{n-1}$.
The coefficient module $A_\rho$ and the space $\mathcal F_\rho$
are those of Sections~\ref{sec:centers-obstruction} and
\ref{sec:nonabelian-good}.

\subsection{The affine space of extensions with fixed mixed maps}
\begin{proposition}[Differences with fixed mixed brackets]
\label{prop:fixed-mixed-differences}
With the mixed family fixed, $(\omega+\theta,\rho_1,\ldots,\rho_{n-1})$
is a non-abelian cocycle precisely when $\theta\in Z_\rho^2$.
\end{proposition}

\begin{proof}
Suppose both data are cocycles. Subtracting their equations
\eqref{eq:NA0} gives $d_\rho\theta=0$.
Subtracting \eqref{eq:NAr0} gives, for $1\leq r\leq n-1$,
\[
\rho_{r+1}(x_1,\ldots,x_{n-1-r})
       (a_1,\ldots,a_r,\theta(y_1,\ldots,y_n))=0.
\]
These are exactly the conditions that every value of $\theta$
belongs to $A_\rho$.

Conversely, let $\theta\in Z_\rho^2$.
Equation \eqref{eq:NA0} changes by $d_\rho\theta=0$.
The additional terms in \eqref{eq:NA1} are
\[
\sum_{i=1}^{n-s}(-1)^{n-s-i}
\rho_{s+1}(y_1,\ldots,\widehat y_i,\ldots,y_{n-s})
        (\theta(X,y_i),b_1,\ldots,b_s)=0,
\]
and those in \eqref{eq:NAr0} are the preceding displayed
expressions, again zero by the definition of $A_\rho$.
Equation \eqref{eq:NArs} does not contain $\omega$.
Thus every cocycle identity is preserved.
\end{proof}

We therefore obtain the vector spaces
\begin{equation}\label{eq:relative-cohomology}
B_\rho^2=d_\rho\mathcal F_\rho,\qquad
\mathcal H_\rho^2=Z_\rho^2/B_\rho^2,\qquad
Z_\Omega=\{f\in\mathcal F_\rho:d_\rho f=0\}.
\end{equation}
Since $\Hom(\g,A_\rho)\subseteq\mathcal F_\rho$, there is a
canonical surjection
\begin{equation}\label{eq:ordinary-to-relative}
H_{\ab}^2(\g,A_\rho;\rho_A)\longrightarrow\mathcal H_\rho^2.
\end{equation}
Here $H_{\ab}^2(\g,A_\rho;\rho_A)
=Z_\rho^2/d_\rho\Hom(\g,A_\rho)$ denotes the fully alternating
extension part of the usual coefficient cohomology.
The larger boundary space in \eqref{eq:relative-cohomology}
is needed because an equivalence preserving the mixed maps can
have its correction in $Z(V)$ without taking values in $A_\rho$.
More precisely, there is a short exact sequence of vector spaces
\begin{equation}\label{eq:H2-comparison-exact}
0\longrightarrow
\frac{d_\rho\mathcal F_\rho}{d_\rho\Hom(\g,A_\rho)}
\longrightarrow H_{\ext}^2(\g,A_\rho;\rho_A)
\longrightarrow\mathcal H_\rho^2\longrightarrow0.
\end{equation}
The middle class maps to zero precisely when its representative
is $d_\rho f$ for some $f\in\mathcal F_\rho$, which proves
exactness. This records the extra identifications between
extensions allowed by kernel-central section changes.

\begin{theorem}\label{thm:affine-classification}
Fix one cocycle $\Omega_0=(\omega_0,\rho_1,\ldots,\rho_{n-1})$.
The map
\[
[\theta]\longmapsto[\omega_0+\theta,\rho_1,\ldots,\rho_{n-1}]
\]
is a bijection from $\mathcal H_\rho^2$ to the equivalence classes
of extensions represented by these fixed mixed maps.
Changing $\omega_0$ translates this bijection.
\end{theorem}

\begin{proof}
Proposition~\ref{prop:fixed-mixed-differences} says that all such
cocycles are uniquely $\omega_0+\theta$, $\theta\in Z_\rho^2$.
By the extension classification, two of them are equivalent
exactly when related by a gauge map $f$. Equality of their mixed
maps gives $f\in\mathcal F_\rho$ by
Lemma~\ref{lem:partial-shears}, and their pure components then
differ by $d_\rho f$. Conversely, every such boundary is realized
by the shear $T_f$. Thus the equivalence relation is precisely
the quotient in \eqref{eq:relative-cohomology}.
If $\omega_1=\omega_0+\theta_1$, the coordinate of
$\omega_0+\theta$ relative to $\omega_1$ is
$[\theta]-[\theta_1]$, proving the last assertion.
\end{proof}

\subsection{Compatible pairs and their correction maps}
Let $d=(D_V,D_\g)\in\Der(V)\oplus\Der(\g)$.
For every mixed map define
\begin{align}
(L_d\rho_r)(x_1,\ldots,x_m)(a_1,\ldots,a_r)
={}&D_V\rho_r(x_1,\ldots,x_m)(a_1,\ldots,a_r)\notag\\
&-\sum_{i=1}^m
 \rho_r(x_1,\ldots,D_\g x_i,\ldots,x_m)(a_1,\ldots,a_r)\notag\\
&-\sum_{j=1}^r
 \rho_r(x_1,\ldots,x_m)(a_1,\ldots,D_V a_j,\ldots,a_r).
\label{eq:mixed-Ld}
\end{align}

\begin{definition}\label{def:full-compatible}
A pair $d$ is \emph{compatible with the mixed brackets} if there
is $\lambda:\g\to V$ such that
\begin{equation}\label{eq:corrected-compatibility}
L_d\rho_r=\partial_r\lambda,\qquad1\leq r\leq n-1.
\end{equation}
The set of such pairs is denoted by $\mathfrak G_\rho$.
The pairs satisfying $L_d\rho_r=0$ for all $r$ form the
subspace $G_\rho^0$, called the strictly compatible pairs.
\end{definition}

For a chosen correction, put
\begin{equation}\label{eq:partial-lift-P}
P_{d,\lambda}(s(x)+a)=s(D_\g x)+\lambda(x)+D_V a.
\end{equation}
Equation \eqref{eq:corrected-compatibility} says exactly that
$P_{d,\lambda}$ satisfies the derivation identity on mixed
brackets. Indeed, the left-hand side of that identity is
$D_V\rho_r(x_1,\ldots,x_m)(a_1,\ldots,a_r)$ and its
right-hand side is
\begin{align*}
&\sum_{i=1}^m\rho_r(x_1,\ldots,D_\g x_i,\ldots,x_m)(a_1,\ldots,a_r)\\
&\quad+\sum_{j=1}^r\rho_r(x_1,\ldots,x_m)(a_1,\ldots,D_V a_j,\ldots,a_r)
       +(\partial_r\lambda)(x_1,\ldots,x_m)(a_1,\ldots,a_r).
\end{align*}
On all-kernel inputs the assertion is $D_V\in\Der(V)$.

\begin{theorem}\label{thm:compatible-Lie-algebra}
The spaces $\mathfrak G_\rho$ and $G_\rho^0$ are Lie subalgebras
of $\Der(V)\oplus\Der(\g)$.
If $\lambda,\mu$ are corrections for
$d=(D_V,D_\g)$ and $e=(E_V,E_\g)$, respectively, then
\begin{equation}\label{eq:compatible-bracket-correction}
\nu=D_V\mu-\mu D_\g-E_V\lambda+\lambda E_\g
\end{equation}
is a correction for $[d,e]$.
All corrections for a fixed compatible pair form
$\lambda+\mathcal F_\rho$.
\end{theorem}

\begin{proof}
Linearity of \eqref{eq:corrected-compatibility} proves closure
under linear combinations. Let $P=P_{d,\lambda}$ and
$Q=P_{e,\mu}$, and take a mixed tuple $z_1,\ldots,z_n$.
Both operators preserve $V$, so all brackets below remain
mixed. Applying their derivation identities twice gives
\begin{align*}
PQ[z_1,\ldots,z_n]
 &=\sum_i[z_1,\ldots,PQz_i,\ldots,z_n]\\
 &\quad+\sum_{i\ne j}
 [z_1,\ldots,Qz_i,\ldots,Pz_j,\ldots,z_n],\\
QP[z_1,\ldots,z_n]
 &=\sum_i[z_1,\ldots,QPz_i,\ldots,z_n]\\
 &\quad+\sum_{i\ne j}
 [z_1,\ldots,Pz_i,\ldots,Qz_j,\ldots,z_n].
\end{align*}
The double sums cancel after exchanging $i,j$.
Thus $[P,Q]$ also satisfies the derivation identity on mixed
brackets. On a section element,
\begin{align*}
PQ(s(x))&=s(D_\g E_\g x)+\lambda(E_\g x)+D_V\mu(x),\\
QP(s(x))&=s(E_\g D_\g x)+\mu(D_\g x)+E_V\lambda(x).
\end{align*}
Consequently $[P,Q]=P_{[d,e],\nu}$ with $\nu$ as in
\eqref{eq:compatible-bracket-correction}.
This proves closure of $\mathfrak G_\rho$.
If $\lambda=\mu=0$, then $\nu=0$, proving closure of $G_\rho^0$.
Finally, subtracting the mixed equations for two corrections
gives $\partial_r(\lambda-\lambda')=0$ for every $r$, which
is exactly $\lambda-\lambda'\in\mathcal F_\rho$.
\end{proof}

\subsection{The action on cochains and cohomology}
For $d\in\mathfrak G_\rho$ define
\begin{align}
L_df&=D_V f-fD_\g,\label{eq:Ld-f}\\
(L_d\theta)(x_1,\ldots,x_n)
 &=D_V\theta(x_1,\ldots,x_n)
   -\sum_{i=1}^n\theta(x_1,\ldots,D_\g x_i,\ldots,x_n).
\label{eq:Ld-theta}
\end{align}

\begin{lemma}\label{lem:compatible-module}
Every $d\in\mathfrak G_\rho$ preserves $A_\rho$ and induces a
compatible derivation pair $(D_V|_{A_\rho},D_\g)$ for $\rho_A$.
Moreover $L_d\mathcal F_\rho\subseteq\mathcal F_\rho$.
\end{lemma}

\begin{proof}
Choose a correction $\lambda$ and use $P=P_{d,\lambda}$.
For $a\in A_\rho$, $b\in V$ and $z_j\in E$, the mixed
derivation identity gives
\begin{align*}
0=P[a,b,z_3,\ldots,z_n]
={}&[D_V a,b,z_3,\ldots,z_n]+[a,D_V b,z_3,\ldots,z_n]\\
&+\sum_{j=3}^n[a,b,z_3,\ldots,Pz_j,\ldots,z_n].
\end{align*}
All terms except the first vanish by
\eqref{eq:relative-center}. Hence $D_V a\in A_\rho$.
On $[s(x_1),\ldots,s(x_{n-1}),a]$, the correction terms are
\[
\sum_i[s(x_1),\ldots,\lambda(x_i),\ldots,s(x_{n-1}),a]=0,
\]
since they contain both $a\in A_\rho$ and a kernel element.
The remaining terms give
\begin{equation}\label{eq:restricted-compatible}
[D_V|_{A_\rho},\rho_A(X)]
 =\sum_{i=1}^{n-1}\rho_A(x_1,\ldots,D_\g x_i,\ldots,x_{n-1}).
\end{equation}
If $f\in\mathcal F_\rho$, both $P$ and $\delta_f$ satisfy the
mixed derivation identity. Their commutator does as well by
the calculation in Theorem~\ref{thm:compatible-Lie-algebra}.
Direct evaluation gives
\[
[P,\delta_f](s(x)+a)=D_V f(x)-f(D_\g x)=\delta_{L_df}(s(x)+a).
\]
Lemma~\ref{lem:partial-shears} now proves $L_df\in\mathcal F_\rho$.
\end{proof}

\begin{lemma}[Cocycles and boundaries]\label{lem:relative-cochain-action}
For $d\in\mathfrak G_\rho$,
\[
L_d Z_\rho^2\subseteq Z_\rho^2,\qquad
L_d(d_\rho f)=d_\rho(L_df)\quad(f\in\mathcal F_\rho).
\]
\end{lemma}

\begin{proof}
Write $D=D_\g$. For $\theta\in Z_\rho^2$, direct expansion of
\eqref{eq:relative-d2} gives
\begin{align*}
&d_\rho L_d\theta(X;Y)\\
&=\rho(X)D_V\theta(Y)
 -\sum_k\rho(X)\theta(y_1,\ldots,Dy_k,\ldots,y_n)\\
&\quad+D_V\theta(X,[y_1,\ldots,y_n]_{\g})
 -\sum_j\theta(x_1,\ldots,Dx_j,\ldots,x_{n-1},[y_1,\ldots,y_n]_{\g})
 -\theta(X,D[y_1,\ldots,y_n]_{\g})\\
&\quad-\sum_iD_V\theta(y_1,\ldots,[X,y_i],\ldots,y_n)\\
&\quad+\sum_i\sum_{k\ne i}
 \theta(y_1,\ldots,Dy_k,\ldots,[X,y_i],\ldots,y_n)
 +\sum_i\theta(y_1,\ldots,D[X,y_i],\ldots,y_n)\\
&\quad-\sum_i(-1)^{n-i}
 \rho(y_1,\ldots,\widehat y_i,\ldots,y_n)D_V\theta(X,y_i)\\
&\quad+\sum_i\sum_j(-1)^{n-i}
 \rho(y_1,\ldots,\widehat y_i,\ldots,y_n)
 \theta(x_1,\ldots,Dx_j,\ldots,x_{n-1},y_i)\\
&\quad+\sum_i(-1)^{n-i}
 \rho(y_1,\ldots,\widehat y_i,\ldots,y_n)\theta(X,Dy_i).
\end{align*}
All values of $\theta$ lie in $A_\rho$. We can therefore use
\eqref{eq:restricted-compatible} to move $D_V$ to the left:
\begin{align*}
\rho(X)D_V\theta(Y)
 &=D_V\rho(X)\theta(Y)
   -\sum_j\rho(x_1,\ldots,Dx_j,\ldots,x_{n-1})\theta(Y),\\
D[y_1,\ldots,y_n]_{\g}&=\sum_k[y_1,\ldots,Dy_k,\ldots,y_n]_{\g},\\
D[X,y_i]&=\sum_j[x_1,\ldots,Dx_j,\ldots,x_{n-1},y_i]
                  +[X,Dy_i].
\end{align*}
The same commutation formula applies to each action omitting $y_i$.
Collecting the terms in which $D$ occupies the same argument yields
\begin{align*}
d_\rho L_d\theta(X;Y)
={}&D_V(d_\rho\theta)(X;Y)\\
&-\sum_j(d_\rho\theta)(x_1,\ldots,Dx_j,\ldots,x_{n-1};Y)\\
&-\sum_k(d_\rho\theta)(X;y_1,\ldots,Dy_k,\ldots,y_n)=0.
\end{align*}
Alternation is preserved, proving the cocycle assertion.

For a boundary we must account for the fact that $f$ need not
take values in $A_\rho$. Expanding \eqref{eq:relative-d1}
and cancelling the bracket terms using $D\in\Der(\g)$ gives
\begin{align*}
&L_d(d_\rho f)(x_1,\ldots,x_n)-d_\rho(L_df)(x_1,\ldots,x_n)\\
&=\sum_i(-1)^{n-i}
 (L_d\rho_1)(x_1,\ldots,\widehat x_i,\ldots,x_n)f(x_i)\\
&=\sum_i(-1)^{n-i}
 (\partial_1\lambda)(x_1,\ldots,\widehat x_i,\ldots,x_n)f(x_i)\\
&=\sum_j\sum_{i\ne j}
 [s(x_1),\ldots,f(x_i),\ldots,\lambda(x_j),\ldots,s(x_n)].
\end{align*}
In the last line the replacements occupy their original slots.
For fixed $j$, apply the mixed derivation identity for $\delta_f$
to the bracket with $\lambda(x_j)$ in slot $j$. Its value and
its $j$th argument are annihilated by $\delta_f$, so the inner
sum is zero. Thus $L_d(d_\rho f)=d_\rho(L_df)$.
The right-hand side is a permitted boundary because
$L_df\in\mathcal F_\rho$.
\end{proof}

\begin{theorem}\label{thm:Phi-relative}
The map
\begin{equation}\label{eq:Phi-relative}
\Phi_\rho:\mathfrak G_\rho\longrightarrow\End(\mathcal H_\rho^2),
\qquad
\Phi_\rho(d)[\theta]=[L_d\theta],
\end{equation}
is a Lie algebra representation.
\end{theorem}

\begin{proof}
Lemma~\ref{lem:relative-cochain-action} shows that the map is
well defined on classes. Let $d=(D_V,D)$, $e=(E_V,E)$.
On $x_1,\ldots,x_n$, expansion of the first composition gives
\begin{align*}
L_dL_e\theta
={}&D_V E_V\theta(x_1,\ldots,x_n)
 -\sum_iD_V\theta(x_1,\ldots,Ex_i,\ldots,x_n)\\
&-\sum_iE_V\theta(x_1,\ldots,Dx_i,\ldots,x_n)
 +\sum_i\theta(x_1,\ldots,EDx_i,\ldots,x_n)\\
&+\sum_{i\ne j}\theta(x_1,\ldots,Dx_i,\ldots,Ex_j,\ldots,x_n).
\end{align*}
Interchange $d,e$ to obtain $L_eL_d\theta$.
The distinct-slot sums cancel after exchanging their indices,
and the remaining difference is
\[
[D_V,E_V]\theta(x_1,\ldots,x_n)
 -\sum_i\theta(x_1,\ldots,[D,E]x_i,\ldots,x_n)
 =L_{[d,e]}\theta(x_1,\ldots,x_n).
\]
Consequently $[\Phi_\rho(d),\Phi_\rho(e)]=\Phi_\rho([d,e])$.
\end{proof}

\subsection{The action in all cohomological degrees}
Use the complex \eqref{eq:full-relative-differential}.
\begin{proposition}\label{prop:relative-full-action}
The operators
\begin{align*}
L_dF(X_1,\ldots,X_p,z)
={}&D_V F(X_1,\ldots,X_p,z)
 -\sum_{j=1}^p F(X_1,\ldots,D_\g X_j,\ldots,X_p,z)\\
&-F(X_1,\ldots,X_p,D_\g z),\\
D_\g X_j
={}&\sum_{\ell=1}^{n-1}
(x_{j1},\ldots,D_\g x_{j\ell},\ldots,x_{j,n-1})
\end{align*}
commute with $\widetilde d$ and induce a Lie algebra action on
every $\widetilde H^p(\g,A_\rho)$.
They preserve the alternating subspaces in
\eqref{eq:extension-complex} and hence act on
$H_{\ext}^2(\g,A_\rho)$ and $H_{\ext}^3(\g,A_\rho)$ as well.
In extension degree, this action descends further through
\eqref{eq:ordinary-to-relative} to $\Phi_\rho$.
\end{proposition}

\begin{proof}
The identities needed for the four lines of
\eqref{eq:full-relative-differential} are
\begin{align*}
D_\g[X,Y]_{\mathrm F}
 &=[D_\g X,Y]_{\mathrm F}+[X,D_\g Y]_{\mathrm F},\\
D_\g[X,z]&=[D_\g X,z]+[X,D_\g z],\\
D_V\rho_A(X)-\rho_A(X)D_V&=\rho_A(D_\g X),\\
[D_V,\rho_A(u_1,\ldots,\widehat u_i,\ldots,u_{n-1},z)]
 &=\sum_{j\ne i}\rho_A(u_1,\ldots,D_\g u_j,\ldots,
                         \widehat u_i,\ldots,u_{n-1},z)\\
 &\quad+\rho_A(u_1,\ldots,\widehat u_i,\ldots,u_{n-1},D_\g z).
\end{align*}
The first two follow by applying the derivation rule to each
bracket; the last two are \eqref{eq:restricted-compatible}.
In the first differential line, replacing
$D_\g[X_j,X_k]_{\mathrm F}$ by its two terms accounts exactly
for inserting $D_\g$ into $X_j$ and $X_k$ before the differential.
The second line is treated by the displayed rule for $[X_j,z]$.
For the third line, the commutator with the coefficient action
supplies the insertion in $X_j$. In the last line it supplies
the insertions in all $u_j$ with $j\ne i$ and in $z$; the
remaining insertion in $u_i$ occurs inside $F$.
Thus each argument receiving $D_\g$ appears once, with its
unchanged displayed sign, and
$\widetilde dL_dF=L_d\widetilde dF$.

In the commutator of two operators, insertions in distinct
individual slots cancel, whereas insertions in the same slot
give $[D_\g,E_\g]$. The output term gives $[D_V,E_V]$.
Hence $[L_d,L_e]=L_{[d,e]}$ in every degree.
On alternating $n$-cochains this is the action already proved;
Lemma~\ref{lem:relative-cochain-action} also preserves the
additional boundaries $d_\rho\mathcal F_\rho$.
\end{proof}

\section{Extensibility of derivations and the derivation Wells sequence}
\label{sec:derivation-wells}

We now combine the mixed compatibility equations with the pure
extension equation. The resulting obstruction takes values in
$\mathcal H_\rho^2$, on which the full compatible Lie algebra
$\mathfrak G_\rho$ acts by Theorem~\ref{thm:Phi-relative}.

\subsection{The cocycle attached to a compatible pair}
For a linear operator $P:E\to E$ preserving $V$, whose restriction
and quotient are derivations, define its bracket defect by
\begin{equation}\label{eq:general-defect}
\eta_P(z_1,\ldots,z_n)
 =P[z_1,\ldots,z_n]
  -\sum_{i=1}^n[z_1,\ldots,Pz_i,\ldots,z_n].
\end{equation}
It is $V$-valued and vanishes on all-kernel inputs.
The same calculation supplies the full infinitesimal obstruction space.

\begin{lemma}\label{lem:defect-cocycle}
For arbitrary $u_j,v_i\in E$, the defect satisfies
\begin{align}
&\eta_P(u_1,\ldots,u_{n-1},[v_1,\ldots,v_n])
 +[u_1,\ldots,u_{n-1},\eta_P(v_1,\ldots,v_n)]\notag\\
&=\sum_{i=1}^n
 \eta_P(v_1,\ldots,[u_1,\ldots,u_{n-1},v_i],\ldots,v_n)\notag\\
&\quad+\sum_{i=1}^n
 [v_1,\ldots,\eta_P(u_1,\ldots,u_{n-1},v_i),\ldots,v_n].
\label{eq:linear-FI}
\end{align}
If $P$ satisfies the derivation identity on all mixed brackets,
then
\[
\theta_P(x_1,\ldots,x_n)=\eta_P(s(x_1),\ldots,s(x_n))
\]
belongs to $Z_\rho^2$.
\end{lemma}

\begin{proof}
Write $U=(u_1,\ldots,u_{n-1})$ and  $Y=(v_1,\ldots,v_n)$.
For this calculation only, $U_j^P$ means $U$ with its $j$th
entry replaced by $Pu_j$. Expanding the left side of the
Filippov identity after applying $P$ gives
\begin{align*}
P[U,[Y]]
={}&\sum_j[U_j^P,[Y]]
 +\sum_i[U,[v_1,\ldots,Pv_i,\ldots,v_n]]\\
&+[U,\eta_P(Y)]+\eta_P(U,[Y]).
\end{align*}
Expanding the right side gives
\begin{align*}
P\sum_i[v_1,\ldots,[U,v_i],\ldots,v_n]
={}&\sum_i\sum_{k\ne i}
 [v_1,\ldots,Pv_k,\ldots,[U,v_i],\ldots,v_n]\\
&+\sum_i\sum_j[v_1,\ldots,[U_j^P,v_i],\ldots,v_n]\\
&+\sum_i[v_1,\ldots,[U,Pv_i],\ldots,v_n]\\
&+\sum_i[v_1,\ldots,\eta_P(U,v_i),\ldots,v_n]\\
&+\sum_i\eta_P(v_1,\ldots,[U,v_i],\ldots,v_n).
\end{align*}
The first sum in the first expansion equals the second double
sum in the second expansion by the Filippov identity for $U_j^P,Y$.
The second sum in the first expansion equals the first double
sum and the third sum in the second expansion, by applying the
same identity with $Pv_i$ in the $i$th inner slot.
The remaining terms give \eqref{eq:linear-FI}.

Now suppose $\eta_P$ vanishes on mixed tuples. Put
\[
U=(s(x_1),\ldots,s(x_{n-1-r}),a_1,\ldots,a_r),
\qquad Y=(s(y_1),\ldots,s(y_n)),\quad r\geq1.
\]
Every term of \eqref{eq:linear-FI} except $[U,\theta_P(Y)]$
vanishes: an argument of each defect is in $V$, and
$[U,s(y_i)]\in V$. Hence
\[
\rho_{r+1}(x_1,\ldots,x_{n-1-r})
 (a_1,\ldots,a_r,\theta_P(Y))=0,\qquad1\leq r\leq n-1.
\]
Thus $\theta_P(Y)\in A_\rho$.
On pure section inputs \eqref{eq:linear-FI} becomes
\begin{align*}
&\theta_P(X,[Y]_\g)+\rho(X)\theta_P(Y)\\
&=\sum_i\theta_P(y_1,\ldots,[X,y_i],\ldots,y_n)
 +\sum_i(-1)^{n-i}\rho(y_1,\ldots,\widehat y_i,\ldots,y_n)
                  \theta_P(X,y_i).
\end{align*}
The extra kernel terms inside the defects vanish because the
defects vanish on mixed tuples. The displayed identity is
$d_\rho\theta_P=0$.
\end{proof}

For $d=(D_V,D_\g)\in\mathfrak G_\rho$, choose a mixed
correction $\lambda$ as in \eqref{eq:corrected-compatibility}.
Define
\begin{align}
\Theta^\omega_{d,\lambda}(x_1,\ldots,x_n)
={}&D_V\omega(x_1,\ldots,x_n)
 -\sum_i\omega(x_1,\ldots,D_\g x_i,\ldots,x_n)\notag\\
&-\sum_i(-1)^{n-i}
 \rho(x_1,\ldots,\widehat x_i,\ldots,x_n)\lambda(x_i)
 +\lambda([x_1,\ldots,x_n]_\g)\notag\\
={}&(L_d\omega-d_\rho\lambda)(x_1,\ldots,x_n).
\label{eq:relative-obstruction}
\end{align}

\begin{proposition}\label{prop:relative-obstruction}
One has $\Theta^\omega_{d,\lambda}\in Z_\rho^2$, and its class
is independent of the chosen mixed correction. The map
\begin{equation}\label{eq:relative-derivation-Wells}
\mathcal W_{\mathcal E}:\mathfrak G_\rho\longrightarrow\mathcal H_\rho^2,
\qquad
\mathcal W_{\mathcal E}(d)=[\Theta^\omega_{d,\lambda}]
\end{equation}
is linear.
\end{proposition}

\begin{proof}
Use $P=P_{d,\lambda}$ from \eqref{eq:partial-lift-P}.
The mixed equations make $\eta_P$ zero on mixed tuples.
On pure section inputs,
\begin{align*}
&P[s(x_1),\ldots,s(x_n)]\\
&\quad=s(D_\g[x_1,\ldots,x_n]_\g)
   +\lambda([x_1,\ldots,x_n]_\g)+D_V\omega(x_1,\ldots,x_n),\\
&\sum_i[s(x_1),\ldots,Ps(x_i),\ldots,s(x_n)]\\
&\quad=s\left(\sum_i[x_1,\ldots,D_\g x_i,\ldots,x_n]_\g\right)
       +\sum_i\omega(x_1,\ldots,D_\g x_i,\ldots,x_n)\\
&\qquad+\sum_i(-1)^{n-i}
 \rho(x_1,\ldots,\widehat x_i,\ldots,x_n)\lambda(x_i).
\end{align*}
The section terms cancel since $D_\g$ is a derivation.
Thus $\theta_P=\Theta^\omega_{d,\lambda}$, and
Lemma~\ref{lem:defect-cocycle} proves the cocycle assertion.

If $\lambda'$ is another correction, then
$f=\lambda'-\lambda\in\mathcal F_\rho$ and
\[
\Theta^\omega_{d,\lambda'}-\Theta^\omega_{d,\lambda}=-d_\rho f.
\]
The class is therefore independent of $\lambda$.
If $\mu$ corrects $e$, then $a\lambda+b\mu$ corrects $ad+be$ and
\[
\Theta^\omega_{ad+be,a\lambda+b\mu}
 =a\Theta^\omega_{d,\lambda}+b\Theta^\omega_{e,\mu}.
\]
This proves linearity without choosing corrections globally.
\end{proof}

\subsection{Extensibility and all derivation lifts}
\begin{theorem}\label{thm:relative-extensibility}
A pair $d\in\Der(V)\oplus\Der(\g)$ is induced by a derivation
of $E$ preserving $V$ precisely when
\[
d\in\mathfrak G_\rho,\qquad \mathcal W_{\mathcal E}(d)=0.
\]
If $\lambda$ is any mixed correction, the lifts are exactly
\begin{equation}\label{eq:relative-all-derivations}
D(s(x)+a)=s(D_\g x)+\lambda(x)+f(x)+D_V a,
\qquad
f\in\mathcal F_\rho,\quad d_\rho f=\Theta^\omega_{d,\lambda}.
\end{equation}
Whenever one lift exists, all lifts form an affine space over
$Z_\Omega$.
\end{theorem}

\begin{proof}
Any linear lift of the prescribed pair has the unique form
$P_{d,\ell}$ for some $\ell:\g\to V$.
For $m=n-r$, its derivation identity on a mixed tuple reads
\begin{align*}
D_V\rho_r(x_1,\ldots,x_m)(a_1,\ldots,a_r)
={}&\sum_{i=1}^m
\rho_r(x_1,\ldots,D_\g x_i,\ldots,x_m)(a_1,\ldots,a_r)\\
&+\sum_{j=1}^r
\rho_r(x_1,\ldots,x_m)(a_1,\ldots,D_V a_j,\ldots,a_r)\\
&+\sum_{i=1}^m(-1)^{m-i}
\rho_{r+1}(x_1,\ldots,\widehat x_i,\ldots,x_m)
                  (\ell(x_i),a_1,\ldots,a_r).
\end{align*}
Consequently $\ell$ is a mixed correction and
$d\in\mathfrak G_\rho$.
The pure calculation in Proposition~\ref{prop:relative-obstruction}
shows that the remaining derivation identity is
\[
L_d\omega=d_\rho\ell.
\]
Fixing one mixed correction $\lambda$, every other one is
$\ell=\lambda+f$, $f\in\mathcal F_\rho$. The pure equation is
then $d_\rho f=L_d\omega-d_\rho\lambda$.
It has a solution exactly when the class in
\eqref{eq:relative-derivation-Wells} is zero.

Conversely, a solution in \eqref{eq:relative-all-derivations}
satisfies the displayed mixed identity, the pure identity, and
the all-kernel identity for $D_V$. Alternation and multilinearity
give the derivation rule on every $z_i=s(x_i)+a_i$.
If $f_0$ is one solution, all others satisfy
\[
\partial_r(f-f_0)=0,\qquad d_\rho(f-f_0)=0,
\]
and hence are exactly $f_0+Z_\Omega$.
\end{proof}

\subsection{The Lie cocycle identity and exactness}
\begin{theorem}\label{thm:relative-Lie-cocycle}
The derivation Wells map is a Lie algebra $1$-cocycle for
$\Phi_\rho$:
\begin{equation}\label{eq:relative-Lie-cocycle}
\mathcal W_{\mathcal E}([d,e])
 =\Phi_\rho(d)\mathcal W_{\mathcal E}(e)
  -\Phi_\rho(e)\mathcal W_{\mathcal E}(d).
\end{equation}
\end{theorem}

\begin{proof}
Choose mixed corrections $\lambda,\mu$ and write
$P=P_{d,\lambda}$, $Q=P_{e,\mu}$.
To see the defect of their commutator, expand on arbitrary
$z_1,\ldots,z_n$:
\begin{align*}
PQ[z_1,\ldots,z_n]
={}&\sum_i[z_1,\ldots,PQz_i,\ldots,z_n]\\
&+\sum_{i\ne j}[z_1,\ldots,Qz_i,\ldots,Pz_j,\ldots,z_n]\\
&+\sum_i\eta_P(z_1,\ldots,Qz_i,\ldots,z_n)+P\eta_Q(z_1,\ldots,z_n).
\end{align*}
Subtract the expansion with $P,Q$ interchanged.
The distinct-slot sums cancel, leaving
\begin{align*}
\eta_{[P,Q]}(z_1,\ldots,z_n)
={}&P\eta_Q(z_1,\ldots,z_n)
  -\sum_i\eta_Q(z_1,\ldots,Pz_i,\ldots,z_n)\\
&-Q\eta_P(z_1,\ldots,z_n)
  +\sum_i\eta_P(z_1,\ldots,Qz_i,\ldots,z_n).
\end{align*}
The defects vanish on mixed tuples and take values in $A_\rho$.
Thus, on section inputs, terms involving $\lambda(x_i)$ or
$\mu(x_i)$ inside a defect vanish. The remaining equality is
\begin{equation}\label{eq:Theta-bracket}
\Theta^\omega_{[d,e],\nu}
 =L_d\Theta^\omega_{e,\mu}-L_e\Theta^\omega_{d,\lambda},
\qquad
\nu=D_V\mu-\mu D_\g-E_V\lambda+\lambda E_\g.
\end{equation}
Theorem~\ref{thm:compatible-Lie-algebra} ensures that $\nu$ is
a valid correction. Passing to classes proves
\eqref{eq:relative-Lie-cocycle}.
\end{proof}

Put
\[
\Der_V(E)=\{D\in\Der(E):D(V)\subseteq V\},
\qquad \tau_*(D)=(D|_V,\overline D).
\]

\begin{theorem}[Derivation Wells sequence]\label{thm:relative-derivation-exact}
There is an exact sequence of vector spaces
\begin{equation}\label{eq:relative-derivation-exact}
0\longrightarrow Z_\Omega
\xrightarrow{\,f\mapsto\delta_f\,}\Der_V(E)
\xrightarrow{\,\tau_*\,}\mathfrak G_\rho
\xrightarrow{\,\mathcal W_{\mathcal E}\,}\mathcal H_\rho^2.
\end{equation}
The first two arrows are Lie algebra homomorphisms, with
$Z_\Omega$ given its zero bracket. The last arrow is the
$1$-cocycle in \eqref{eq:relative-Lie-cocycle}.
In particular, with $\mathfrak k_{\mathcal E}=\ker\mathcal W_{\mathcal E}$,
\begin{equation}\label{eq:derivation-short}
0\longrightarrow Z_\Omega\longrightarrow\Der_V(E)
\xrightarrow{\tau_*}\mathfrak k_{\mathcal E}\longrightarrow0
\end{equation}
is a short exact sequence of Lie algebras with abelian kernel.
\end{theorem}

\begin{proof}
For $f\in Z_\Omega$, the mixed identities are $\partial_rf=0$,
and the pure identity is $d_\rho f=0$. Hence $\delta_f$ is
a derivation. The map $f\mapsto\delta_f$ is injective because
$\delta_fs(x)=f(x)$.
If $\tau_*(D)=0$, then $D|_V=0$ and $pD=0$.
The map $f(x)=D(s(x))$ is independent of the representative
of $x$, and $D=\delta_f$. Its derivation identities imply
$f\in Z_\Omega$. This proves exactness at $\Der_V(E)$.
Theorem~\ref{thm:relative-extensibility} proves
$\im\tau_*=\ker\mathcal W_{\mathcal E}$.

For lifts $D=P_{d,\lambda}$ and $D'=P_{e,\mu}$, direct evaluation gives
\begin{align*}
[D,D'](s(x)+a)
={}&s([D_\g,E_\g]x)+[D_V,E_V]a\\
&+D_V\mu(x)-\mu(D_\g x)-E_V\lambda(x)+\lambda(E_\g x).
\end{align*}
Thus $\tau_*[D,D']=[\tau_*D,\tau_*D']$.
Moreover
\[
[\delta_f,\delta_k]=0,\qquad
[D,\delta_f]=\delta_{L_df}.
\]
The kernel is an abelian ideal. Finally,
\eqref{eq:relative-Lie-cocycle} makes $\ker\mathcal W_{\mathcal E}$
closed under brackets.
\end{proof}

For completeness, define
\[
\mathfrak G_{\rho,V}=\{D_V:(D_V,0)\in\mathfrak G_\rho\},
\qquad
\mathfrak G_{\rho,\g}=\{D_\g:(0,D_\g)\in\mathfrak G_\rho\}.
\]
Restriction of the same calculation gives
\begin{align*}
0\longrightarrow Z_\Omega\longrightarrow
 \{D\in\Der_V(E):\overline D=0\}
 &\longrightarrow\mathfrak G_{\rho,V}
 \xrightarrow{\,D_V\mapsto\mathcal W_{\mathcal E}(D_V,0)\,}
 \mathcal H_\rho^2,\\
0\longrightarrow Z_\Omega\longrightarrow
 \{D\in\Der_V(E):D|_V=0\}
 &\longrightarrow\mathfrak G_{\rho,\g}
 \xrightarrow{\,D_\g\mapsto\mathcal W_{\mathcal E}(0,D_\g)\,}
 \mathcal H_\rho^2.
\end{align*}
No surjectivity onto the final cohomology space is asserted.

\subsection{Independence of the section}
\begin{proposition}\label{prop:relative-section}
The spaces $A_\rho,\mathcal F_\rho,Z_\Omega,\mathcal H_\rho^2$,
the Lie algebra $\mathfrak G_\rho$, its action $\Phi_\rho$, and
the map $\mathcal W_{\mathcal E}$ are independent of the section
under their natural identifications.
If $s'=s+k$, a correction for the same operator changes by
\begin{equation}\label{eq:derivation-section}
\lambda'=\lambda+D_V k-kD_\g.
\end{equation}
The strictly compatible subalgebra $G_\rho^0$ need not be
section independent.
\end{proposition}

\begin{proof}
The space $A_\rho$ is given intrinsically by
$[A_\rho,V,E,\ldots,E]=0$.
Its representation is unchanged by
Proposition~\ref{prop:central-representation}, so $Z_\rho^2$
is unchanged. The map $\delta_f=ifp$ is also intrinsic.
By Lemma~\ref{lem:partial-shears}, membership of $f$ in
$\mathcal F_\rho$ says that $\delta_f$ satisfies the derivation
identity on all mixed brackets; it is therefore unchanged.

For $f\in\mathcal F_\rho$ the defect $\eta_{\delta_f}$ vanishes
on mixed tuples, and its value on pure section inputs is
$-d_\rho f$. Consequently
\begin{align*}
-(d_{\rho'}f)(x_1,\ldots,x_n)
 &=\eta_{\delta_f}(s(x_1)+k(x_1),\ldots,s(x_n)+k(x_n))\\
 &=\eta_{\delta_f}(s(x_1),\ldots,s(x_n))
 =-(d_\rho f)(x_1,\ldots,x_n).
\end{align*}
Thus the boundary spaces and their kernels agree exactly.
This identifies $\mathcal H_{\rho'}^2=\mathcal H_\rho^2$.

Compatibility of $d$ means the existence of an operator on $E$
inducing $d$ and satisfying the mixed derivation identities.
This intrinsic description identifies the two compatible Lie
algebras. For the same operator $P$,
\begin{align*}
P(s'(x)+a)
 &=s(D_\g x)+\lambda(x)+D_V k(x)+D_V a\\
 &=s'(D_\g x)+(\lambda+D_V k-kD_\g)(x)+D_V a,
\end{align*}
which proves \eqref{eq:derivation-section}.
Since $\eta_P$ vanishes on mixed tuples,
\begin{align*}
\Theta^{\omega'}_{d,\lambda'}(x_1,\ldots,x_n)
 &=\eta_P(s'(x_1),\ldots,s'(x_n))\\
 &=\eta_P(s(x_1),\ldots,s(x_n))
 =\Theta^\omega_{d,\lambda}(x_1,\ldots,x_n).
\end{align*}
Hence the Wells maps agree already at the cocycle level.
The formula for $\Phi_\rho$ uses only $D_V|_{A_\rho},D_\g$,
so its action is unchanged. By contrast, $\lambda=0$ need not
imply $\lambda'=0$, explaining the last assertion.
\end{proof}

\subsection{Extensibility in every extension with fixed mixed brackets}
Fix $\Omega_0=(\omega_0,\rho_1,\ldots,\rho_{n-1})$ and denote by
$\mathcal E_\theta$ the extension represented by
$(\omega_0+\theta,\rho_1,\ldots,\rho_{n-1})$, where
$\theta\in Z_\rho^2$.
This is the affine family classified in
Theorem~\ref{thm:affine-classification}.

\begin{theorem}\label{thm:universal-extensibility}
For every $d\in\mathfrak G_\rho$,
\begin{equation}\label{eq:affine-Wells}
\mathcal W_{\mathcal E_\theta}(d)
 =\mathcal W_{\mathcal E_0}(d)+\Phi_\rho(d)[\theta].
\end{equation}
The pair $d$ is extensible in every extension with these mixed
brackets precisely when
\begin{equation}\label{eq:universal-extensibility}
\mathcal W_{\mathcal E_0}(d)=0,\qquad
\Phi_\rho(d)=0\ \text{on }\mathcal H_\rho^2.
\end{equation}
In particular, if $\mathcal H_\rho^2=0$, every compatible pair
is extensible in every extension in the family.
\end{theorem}

\begin{proof}
The mixed correction equation is the same for every $\theta$.
Using the same $\lambda$ in both obstructions gives
\begin{align*}
\Theta^{\omega_0+\theta}_{d,\lambda}
 &=L_d(\omega_0+\theta)-d_\rho\lambda\\
 &=\Theta^{\omega_0}_{d,\lambda}+L_d\theta.
\end{align*}
Passing to classes proves \eqref{eq:affine-Wells}.
If $d$ extends in every member, choosing $\theta=0$ gives the
first equality in \eqref{eq:universal-extensibility}.
For arbitrary $\theta$, subtracting that equality from
$\mathcal W_{\mathcal E_\theta}(d)=0$ gives
$\Phi_\rho(d)[\theta]=0$. Every class has such a representative,
so $\Phi_\rho(d)=0$.
Conversely, the two conditions make the right side of
\eqref{eq:affine-Wells} vanish for every $\theta$;
Theorem~\ref{thm:relative-extensibility} supplies the lifts.
\end{proof}

In an abelian semidirect product the origin $\omega_0=0$ is
available and its obstruction is zero. A fixed non-abelian
mixed family need not admit that origin. Even when it does,
a compatible pair may require a nonzero mixed correction.
Thus the first condition in \eqref{eq:universal-extensibility}
cannot be discarded in the general statement. This is the
additional term required in the non-abelian extension of
the universal extensibility criterion of~\cite{Xu}.

\subsection{Splitting of the derivation sequence}
If $\omega=0$, every strictly compatible pair has the lift
$s(x)+a\mapsto s(D_\g x)+D_V a$.
The bracket calculation above therefore gives
\begin{equation}\label{eq:strict-derivation-split}
\tau_*^{-1}(G_\rho^0)\simeq Z_\Omega\rtimes G_\rho^0,
\qquad
[(f,d),(k,e)]=(L_dk-L_ef,[d,e]).
\end{equation}
For the full image $\mathfrak k_{\mathcal E}$, splitting has
a further obstruction, which can also be written elementwise.

\begin{proposition}\label{prop:derivation-factor}
Choose a linear family of actual lifts
$\sigma(d)=P_{d,\lambda_d}$ for $d\in\mathfrak k_{\mathcal E}$.
Then
\begin{equation}\label{eq:derivation-factor}
\kappa_*(d,e)
 =L_d\lambda_e-L_e\lambda_d-\lambda_{[d,e]}
\end{equation}
takes values in $Z_\Omega$, is alternating, and satisfies
\begin{align}
0={}&L_d\kappa_*(e,f)-L_e\kappa_*(d,f)+L_f\kappa_*(d,e)\notag\\
&-\kappa_*([d,e],f)+\kappa_*([d,f],e)-\kappa_*([e,f],d).
\label{eq:derivation-factor-cocycle}
\end{align}
The short exact sequence \eqref{eq:derivation-short} splits
as Lie algebras precisely when there is a linear
$b:\mathfrak k_{\mathcal E}\to Z_\Omega$ such that
\[
\kappa_*(d,e)+L_db(e)-L_eb(d)-b([d,e])=0.
\]
\end{proposition}

\begin{proof}
A linear section of the surjective vector-space map $\tau_*$
exists. The commutator formula gives
\[
[\sigma(d),\sigma(e)]-\sigma([d,e])=\delta_{\kappa_*(d,e)}.
\]
Its left side is a derivation inducing zero on kernel and
quotient, hence $\kappa_*(d,e)\in Z_\Omega$.
Alternation follows immediately.
Using $[\sigma(d),\delta_z]=\delta_{L_dz}$ in the Jacobi identity,
we obtain
\begin{align*}
0
&=\sum_{\mathrm{cyc}}
 [\sigma(d),[\sigma(e),\sigma(f)]]\\
&=\sum_{\mathrm{cyc}}
 \left\{\sigma([d,[e,f]])
       +\delta_{\kappa_*(d,[e,f])}
       +\delta_{L_d\kappa_*(e,f)}\right\}.
\end{align*}
The $\sigma$-terms vanish by the Jacobi identity in
$\mathfrak k_{\mathcal E}$. Injectivity of $z\mapsto\delta_z$
gives \eqref{eq:derivation-factor-cocycle}.
Replacing $\sigma(d)$ by $\sigma(d)+\delta_{b(d)}$ changes
the factor to
\[
\kappa_*'(d,e)=\kappa_*(d,e)+L_db(e)-L_eb(d)-b([d,e]),
\]
because $[\delta_{b(d)},\delta_{b(e)}]=0$.
The new section is a Lie homomorphism exactly when this
factor vanishes, proving the assertion.
\end{proof}

\subsection{The binary specialization}
For $n=2$ the mixed datum is $\rho:\g\to\Der(V)$ and
\[
A_\rho=Z(V),\qquad
\mathcal F_\rho=\Hom(\g,Z(V)),\qquad
\mathcal H_\rho^2=H^2(\g,Z(V)).
\]
The full compatibility equation is
\[
[D_V,\rho(x)]-\rho(D_\g x)=\ad_V(\lambda(x)).
\]
Its obstruction is
\begin{align*}
\Theta^\omega_{d,\lambda}(x,y)
={}&D_V\omega(x,y)-\omega(D_\g x,y)-\omega(x,D_\g y)\\
&-\rho(x)\lambda(y)+\rho(y)\lambda(x)+\lambda([x,y]).
\end{align*}
Thus \eqref{eq:relative-derivation-exact} is the familiar
central-coefficient derivation obstruction sequence for
non-abelian Lie extensions; the higher-arity formulas retain
all the additional mixed compatibility equations.
For the binary extension framework see
\cite{InassaridzeKhmaladzeLadra,Fregier}.

\subsection{Comparison with the full infinitesimal obstruction space}
\label{sec:derivations}
The preceding sequence first solves the mixed equations and then
measures the remaining pure obstruction. There is also a single
linear obstruction which retains every component.
We describe its precise relation to \eqref{eq:relative-derivation-exact}.

Use section coordinates $E=\g\oplus_\Omega V$ and put
$C^1=\Hom(\g,V)$. Let $C_\Omega^2$ consist of alternating maps
$\eta:\bigwedge^nE\to V$ vanishing on $\bigwedge^nV$.
Let $Z_\Omega^2$ be the solutions of \eqref{eq:linear-FI} in
$C_\Omega^2$. Define
\begin{align}
(d_\Omega f)(z_1,\ldots,z_n)
 &=\sum_i[z_1,\ldots,\delta_fz_i,\ldots,z_n]_\Omega
   -\delta_f[z_1,\ldots,z_n]_\Omega.
\label{eq:dOmega}
\end{align}
Its pure and mixed components are
\begin{align}
(d_\Omega f)_0&=d_\rho f,\label{eq:dOmega0}\\
(d_\Omega f)_r&=\partial_rf,\qquad1\leq r\leq n-1.
\label{eq:dOmegar}
\end{align}
Since $d_\Omega f=-\eta_{\delta_f}$,
Lemma~\ref{lem:defect-cocycle} gives
$d_\Omega C^1\subseteq Z_\Omega^2$. Thus
\begin{equation}\label{eq:tangent-H}
H_\Omega^2=Z_\Omega^2/d_\Omega C^1
\end{equation}
is a vector space. It concerns first-order changes of all the
extension data; it is not an additive structure on
$H_{\nab}^2(\g,V)$.

For $d=(D_V,D_\g)\in\Der(V)\oplus\Der(\g)$ define
$\widehat d(x+a)=D_\g x+D_V a$ and
\begin{equation}\label{eq:Ob-def}
\Ob_d(z_1,\ldots,z_n)
 =\widehat d[z_1,\ldots,z_n]_\Omega
 -\sum_i[z_1,\ldots,\widehat d z_i,\ldots,z_n]_\Omega.
\end{equation}
Its components are
\begin{align}
(\Ob_d)_0&=L_d\omega,\label{eq:Ob0}\\
(\Ob_d)_r&=L_d\rho_r,\qquad1\leq r\leq n-1.
\label{eq:Obr}
\end{align}

\begin{theorem}\label{thm:derivation-obstruction}
For every derivation pair $d$, the defect $\Ob_d$ belongs to
$Z_\Omega^2$. The pair is extensible precisely when
$[\Ob_d]=0$ in $H_\Omega^2$, and its lifts are
\begin{equation}\label{eq:derivation-lift}
D(s(x)+a)=s(D_\g x)+\lambda(x)+D_V a,
\qquad d_\Omega\lambda=\Ob_d.
\end{equation}
\end{theorem}

\begin{proof}
The first assertion is Lemma~\ref{lem:defect-cocycle} with
$P=\widehat d$. For a linear lift $D=\widehat d+\delta_\lambda$,
the defect is
\[
\eta_D=\eta_{\widehat d}+\eta_{\delta_\lambda}
       =\Ob_d-d_\Omega\lambda.
\]
It vanishes precisely when the displayed lifting equation
holds. Its mixed components are
$\partial_r\lambda=L_d\rho_r$ and its pure component is
$d_\rho\lambda=L_d\omega$, exactly the two stages of
Theorem~\ref{thm:relative-extensibility}.
\end{proof}

\begin{corollary}\label{thm:derivation-exact}
There is an exact sequence of vector spaces
\begin{equation}\label{eq:derivation-exact}
0\longrightarrow Z_\Omega\longrightarrow\Der_V(E)
\xrightarrow{\tau_*}\Der(V)\oplus\Der(\g)
\xrightarrow{\,\mathcal W_\Omega\,}H_\Omega^2,
\qquad \mathcal W_\Omega(d)=[\Ob_d].
\end{equation}
\end{corollary}
Exactness at the first two terms was proved in
Theorem~\ref{thm:relative-derivation-exact}; exactness at the
pair space is the preceding lifting criterion.
We do not assign the full pair space an action on $H_\Omega^2$.
The action needed for the Lie cocycle statement is the action
of $\mathfrak G_\rho$ on $\mathcal H_\rho^2$ constructed above.

\begin{proposition}\label{prop:relative-tangent-inclusion}
Extend $\theta\in Z_\rho^2$ to a cochain $j\theta$ on $E$
by requiring it to vanish whenever an argument lies in $V$:
\[
(j\theta)(x_1+a_1,\ldots,x_n+a_n)=\theta(x_1,\ldots,x_n).
\]
This induces an injection
$j:\mathcal H_\rho^2\hookrightarrow H_\Omega^2$.
For $d\in\mathfrak G_\rho$,
\begin{equation}\label{eq:relative-tangent-comparison}
\mathcal W_\Omega(d)=j\,\mathcal W_{\mathcal E}(d).
\end{equation}
\end{proposition}

\begin{proof}
On pure inputs, the linearized equation for $j\theta$ is
$d_\rho\theta=0$. If at least one input lies in $V$, every
potentially nonzero term is a bracket of a value of $\theta$
with a kernel argument. It vanishes because $\theta$ takes
values in $A_\rho$. Thus $j\theta\in Z_\Omega^2$.
For $f\in\mathcal F_\rho$, equations
\eqref{eq:dOmega0}--\eqref{eq:dOmegar} give
$d_\Omega f=j(d_\rho f)$, proving well-definedness.
If $j\theta=d_\Omega f$ for an arbitrary $f$, comparison of
the mixed components gives $\partial_rf=0$ for all $r$.
Hence $f\in\mathcal F_\rho$ and $\theta=d_\rho f$.
This proves injectivity.

If $\lambda$ is a mixed correction for $d$, the mixed
components of $\Ob_d-d_\Omega\lambda$ are zero, while its
pure component is $\Theta^\omega_{d,\lambda}$.
Therefore
\[
\Ob_d-d_\Omega\lambda=j\Theta^\omega_{d,\lambda}.
\]
Passing to quotient spaces proves \eqref{eq:relative-tangent-comparison}.
\end{proof}

The full space is also independent of section by direct
transport of the elementwise equations.
For $\Omega'=\Omega^k$, put
\begin{equation}\label{eq:eta-transport}
(P_k\eta)(z_1,\ldots,z_n)
 =\eta(T_kz_1,\ldots,T_kz_n).
\end{equation}
Write $U=(u_1,\ldots,u_{n-1})$, $Y=(v_1,\ldots,v_n)$.
Since $T_k[U,v]_{\Omega'}=[T_kU,T_kv]_\Omega$ and
$T_k$ is the identity on $V$, the four terms in the
linearized identity transform as follows:
\begin{align*}
(P_k\eta)(U,[Y]_{\Omega'})
 &=\eta(T_kU,[T_kY]_\Omega),\\
[U,(P_k\eta)(Y)]_{\Omega'}
 &=[T_kU,\eta(T_kY)]_\Omega,\\
(P_k\eta)(v_1,\ldots,[U,v_i]_{\Omega'},\ldots,v_n)
 &=\eta(T_kv_1,\ldots,[T_kU,T_kv_i]_\Omega,\ldots,T_kv_n),\\
[v_1,\ldots,(P_k\eta)(U,v_i),\ldots,v_n]_{\Omega'}
 &=[T_kv_1,\ldots,\eta(T_kU,T_kv_i),\ldots,T_kv_n]_\Omega.
\end{align*}
Their signed sum is zero by \eqref{eq:linear-FI}.
Thus $P_kZ_\Omega^2=Z_{\Omega'}^2$, with inverse $P_{-k}$.
Also $\delta_fT_k=T_k\delta_f=\delta_f$, so substitution
in \eqref{eq:dOmega} gives
\[
P_k(d_\Omega f)=d_{\Omega'}f.
\]
Consequently $P_k$ descends to an isomorphism of the quotient
spaces. Direct evaluation further gives
\[
T_{-k}\widehat dT_k=\widehat d+\delta_{D_V k-kD_\g},
\qquad
P_k\Ob_d^\Omega
 =\Ob_d^{\Omega'}-d_{\Omega'}(D_V k-kD_\g).
\]
The obstruction classes therefore agree.

The relation with the automorphism problem can be stated
entirely in components. The pure and mixed equations for a
derivation lift are
\begin{equation}\label{eq:Wells-linearization}
L_d\omega=d_\rho\lambda,\qquad
L_d\rho_r=\partial_r\lambda\quad(1\leq r\leq n-1).
\end{equation}
They follow by placing the correction $\lambda(x_i)$ in
one quotient slot at a time, as in the proof of
Theorem~\ref{thm:relative-extensibility}.
The automorphism problem, studied next, compares the
corresponding finite section changes. Both calculations
retain every mixed input type.

\section{Inducibility of automorphism pairs}\label{sec:inducibility}

Fix \eqref{eq:extension} and a section $s$ with datum $\Omega$.
Set
\begin{equation}\label{eq:autgroups}
\Aut_{V}(E)=\{\gamma\in\Aut(E):\gamma(V)=V\},\qquad
G=\Aut(V)\times\Aut(\g).
\end{equation}
Equality $\gamma(V)=V$ is part of the definition, also in infinite
dimension. For $\gamma\in\Aut_{V}(E)$ let
\begin{equation}\label{eq:tau}
\tau(\gamma)=(\gamma|_{V},\bar\gamma),
\qquad \bar\gamma(x)=p\gamma s(x).
\end{equation}
If $s'=s+f$, then
$p\gamma s'(x)=p\gamma s(x)+p\gamma f(x)=p\gamma s(x)$.
Also $p\gamma=\bar\gamma p$, so $\bar\gamma$ is an automorphism
with inverse induced by $\gamma^{-1}$ and
\[
\overline{\gamma\delta}=\bar\gamma\,\bar\delta.
\]
Thus $\tau$ is a section-independent group homomorphism.
A pair $(\beta,\alpha)\in G$ is \emph{inducible} when it lies in
$\im\tau$.

\begin{theorem}[Elementwise lifting criterion]\label{thm:inducibility}
The pair $(\beta,\alpha)\in G$ is inducible if and only if there is
$\lambda\in\Hom(\g,V)$ satisfying, for $1\leq r\leq n-1$
and $m=n-r$,
\begin{align}
&\beta\rho_r(x_1,\ldots,x_m)(a_1,\ldots,a_r)
 -\rho_r(\alpha x_1,\ldots,\alpha x_m)
            (\beta a_1,\ldots,\beta a_r)\notag\\
&=\sum_{i=1}^{m}
 [\alpha x_1,\ldots,\alpha x_{i-1},\lambda(x_i),
   \alpha x_{i+1}+\lambda(x_{i+1}),\ldots,
   \alpha x_m+\lambda(x_m),\beta a_1,\ldots,\beta a_r]_{\Omega},
 \label{eq:lift-rho}\\
&\beta\omega(x_1,\ldots,x_n)
 -\omega(\alpha x_1,\ldots,\alpha x_n)\notag\\
&=\sum_{i=1}^{n}
 [\alpha x_1,\ldots,\alpha x_{i-1},\lambda(x_i),
  \alpha x_{i+1}+\lambda(x_{i+1}),\ldots,
  \alpha x_n+\lambda(x_n)]_{\Omega}\notag\\
&\quad-\lambda([x_1,\ldots,x_n]_{\g}).
 \label{eq:lift-omega}
\end{align}
The corresponding lift and its inverse are
\begin{align}
\gamma(s(x)+a)&=s(\alpha x)+\lambda(x)+\beta a,
 \label{eq:lift-map}\\
\gamma^{-1}(s(x)+a)
 &=s(\alpha^{-1}x)+
   \beta^{-1}\bigl(a-\lambda(\alpha^{-1}x)\bigr).
 \label{eq:lift-inverse}
\end{align}
Every lift of $(\beta,\alpha)$ has this form.
\end{theorem}

\begin{proof}
Suppose $\tau(\gamma)=(\beta,\alpha)$. Then
\[
p(\gamma s(x)-s(\alpha x))=\alpha x-\alpha x=0,
\]
so $\lambda(x)=\gamma s(x)-s(\alpha x)$ is a linear map into $V$.
For a mixed bracket, preservation by $\gamma$ gives
\begin{align*}
&\beta\rho_r(x_1,\ldots,x_m)(a_1,\ldots,a_r)\\
&=\gamma[s(x_1),\ldots,s(x_m),a_1,\ldots,a_r]\\
&=[s(\alpha x_1)+\lambda(x_1),\ldots,
   s(\alpha x_m)+\lambda(x_m),\beta a_1,\ldots,\beta a_r].
\end{align*}
Subtracting the bracket with no $\lambda$-entries and applying the
telescoping expansion proves \eqref{eq:lift-rho}.
For pure quotient inputs,
\begin{align*}
\gamma[s(x_1),\ldots,s(x_n)]
 &=s(\alpha[x_1,\ldots,x_n]_{\g})
   +\lambda([x_1,\ldots,x_n]_{\g})
   +\beta\omega(x_1,\ldots,x_n),\\
[\gamma s(x_1),\ldots,\gamma s(x_n)]
 &=[s(\alpha x_1)+\lambda(x_1),\ldots,
    s(\alpha x_n)+\lambda(x_n)].
\end{align*}
Their quotient components agree because $\alpha\in\Aut(\g)$.
Subtracting $s(\alpha[x_1,\ldots,x_n]_{\g})$ from the equality of
these expressions gives \eqref{eq:lift-omega}.

Conversely, define $\gamma$ by \eqref{eq:lift-map}. Formula
\eqref{eq:lift-inverse} gives a two-sided inverse, as
\begin{align*}
\gamma\bigl(s(\alpha^{-1}x)
 +\beta^{-1}(a-\lambda(\alpha^{-1}x))\bigr)
 &=s(x)+\lambda(\alpha^{-1}x)
   +a-\lambda(\alpha^{-1}x)\\
 &=s(x)+a.
\end{align*}
Equation \eqref{eq:lift-rho}, with its unshifted term restored,
says that $\gamma$ preserves every mixed bracket.
Equation \eqref{eq:lift-omega} says that it preserves a bracket of
$n$ section elements. A bracket of $n$ kernel elements is preserved
because $\beta\in\Aut(V)$.
Expanding each argument $s(x_i)+a_i$ by multilinearity now proves
\[
\gamma[s(x_1)+a_1,\ldots,s(x_n)+a_n]
 =[\gamma(s(x_1)+a_1),\ldots,\gamma(s(x_n)+a_n)].
\]
Finally $\gamma|_{V}=\beta$ and $p\gamma s=\alpha$, as required.
The necessity argument already determines $\lambda$ uniquely for a
given $\gamma$.
\end{proof}

For $r=n-1$, the first lifting equation is particularly transparent:
\begin{align}\label{eq:lift-top}
&\beta\rho_{n-1}(x)(a_1,\ldots,a_{n-1})
 -\rho_{n-1}(\alpha x)(\beta a_1,\ldots,\beta a_{n-1})\notag\\
&=[\lambda(x),\beta a_1,\ldots,\beta a_{n-1}]_{V}.
\end{align}
For $n=2$, \eqref{eq:lift-rho}--\eqref{eq:lift-omega} become
\begin{align}
\beta\rho(x)a-\rho(\alpha x)\beta a
 &=[\lambda(x),\beta a]_{V},\notag\\
\beta\omega(x,y)-\omega(\alpha x,\alpha y)
 &=\rho(\alpha x)\lambda(y)-\rho(\alpha y)\lambda(x)
   +[\lambda(x),\lambda(y)]_{V}-\lambda([x,y]_{\g}).
 \label{eq:binary-lift}
\end{align}
These are the non-abelian Lie lifting equations.

\begin{remark}\label{rem:necessary-center}
If both $\beta$ and $\alpha$ are the identity, \eqref{eq:lift-top}
implies $\lambda(\g)\subseteq Z(V)$. The other mixed equations
still have to be checked. For $n\geq3$, being central inside $V$
does not imply that a kernel element has zero brackets with
section elements and fewer kernel entries.
\end{remark}

\section{The Wells map \texorpdfstring{$\varpi$}{varpi}}\label{sec:Wells}

\subsection{Transporting the complete cocycle}
For $c=(\beta,\alpha)\in G$, define
\begin{align}
{}^c\omega(x_1,\ldots,x_n)
 &=\beta\omega(\alpha^{-1}x_1,\ldots,\alpha^{-1}x_n),
 \label{eq:twist-omega}\\
{}^c\rho_r(x_1,\ldots,x_{n-r})(a_1,\ldots,a_r)
 &=\beta\rho_r(\alpha^{-1}x_1,\ldots,\alpha^{-1}x_{n-r})
                 (\beta^{-1}a_1,\ldots,\beta^{-1}a_r).
 \label{eq:twist-rho}
\end{align}
Write $ {}^c\Omega=({}^c\omega,{}^c\rho_1,\ldots,{}^c\rho_{n-1})$
and
\begin{equation}\label{eq:G-action-f}
{}^cf=\beta f\alpha^{-1},\qquad
S_c(x+a)=\alpha x+\beta a.
\end{equation}
The kernel bracket is unchanged, since
$\beta[\beta^{-1}a_1,\ldots,\beta^{-1}a_n]_{V}
 =[a_1,\ldots,a_n]_{V}$.

\begin{proposition}\label{prop:transport}
Equations \eqref{eq:twist-omega}--\eqref{eq:twist-rho}
define a left action of $G$ on $Z^2_{\nab}(\g,V)$.
It satisfies
\begin{equation}\label{eq:transport-gauge}
{}^c(\Omega^f)=({}^c\Omega)^{{}^cf}.
\end{equation}
Consequently,
\begin{equation}\label{eq:cohom-action}
{}^c[\Omega]=[{}^c\Omega]
\end{equation}
is a well-defined left action on $H^2_{\nab}(\g,V)$.
\end{proposition}

\begin{proof}
On a pure quotient bracket one has
\begin{align*}
S_c[S_c^{-1}x_1,\ldots,S_c^{-1}x_n]_{\Omega}
 &=[x_1,\ldots,x_n]_{\g}
   +\beta\omega(\alpha^{-1}x_1,\ldots,\alpha^{-1}x_n).
\end{align*}
On a mixed bracket the same calculation is
\begin{align*}
&S_c[\alpha^{-1}x_1,\ldots,\alpha^{-1}x_{n-r},
                   \beta^{-1}a_1,\ldots,\beta^{-1}a_r]_{\Omega}\\
&=\beta\rho_r(\alpha^{-1}x_1,\ldots,\alpha^{-1}x_{n-r})
                   (\beta^{-1}a_1,\ldots,\beta^{-1}a_r).
\end{align*}
It follows that, on all elements,
\begin{equation}\label{eq:transport-bracket}
[z_1,\ldots,z_n]_{{}^c\Omega}
 =S_c[S_c^{-1}z_1,\ldots,S_c^{-1}z_n]_{\Omega}.
\end{equation}
For arbitrary $u_i,v_j$, the left-hand side of the Filippov identity
for this bracket is
\[
S_c[S_c^{-1}u_1,\ldots,S_c^{-1}u_{n-1},
       [S_c^{-1}v_1,\ldots,S_c^{-1}v_n]_{\Omega}]_{\Omega}.
\]
Applying \eqref{eq:FI} inside $S_c$ gives the sum
\[
\sum_{i=1}^n S_c
[S_c^{-1}v_1,\ldots,
 [S_c^{-1}u_1,\ldots,S_c^{-1}u_{n-1},S_c^{-1}v_i]_{\Omega},
 \ldots,S_c^{-1}v_n]_{\Omega}.
\]
By \eqref{eq:transport-bracket} this is its required right-hand side.
Thus $ {}^c\Omega$ is again a cocycle.
For example, the curvature equation explicitly transports to
\begin{align*}
&[{}^c\rho(X),{}^c\rho(y_1,\ldots,y_{n-1})]a
 -\sum_i{}^c\rho(y_1,\ldots,[X,y_i],\ldots,y_{n-1})a\\
&=\beta\left(
 [\rho(\alpha^{-1}X),\rho(\alpha^{-1}y_1,\ldots,\alpha^{-1}y_{n-1})]
 \beta^{-1}a\right.\\
&\hspace{17mm}\left.
 -\sum_i\rho(\alpha^{-1}y_1,\ldots,
       [\alpha^{-1}X,\alpha^{-1}y_i],\ldots,\alpha^{-1}y_{n-1})
       \beta^{-1}a\right)\\
&=\sum_i(-1)^{n-1-i}\,
 {}^c\rho_2(y_1,\ldots,\widehat y_i,\ldots,y_{n-1})
              ({}^c\omega(X,y_i),a).
\end{align*}
Here $\alpha^{-1}X$ denotes the displayed list of $n-1$ inverse
images, with no change in its order.

Since $S_cS_d=S_{cd}$, equation \eqref{eq:transport-bracket}
gives $ {}^c({}^d\Omega)={}^{cd}\Omega$ and $ {}^1\Omega=\Omega$.
For the interaction with a shear, compute
\begin{align*}
S_cT_fS_c^{-1}(x+a)
 &=S_c(\alpha^{-1}x+f(\alpha^{-1}x)+\beta^{-1}a)\\
 &=x+\beta f(\alpha^{-1}x)+a=T_{{}^cf}(x+a).
\end{align*}
Substituting this equality into
\eqref{eq:gauge-bracket} and \eqref{eq:transport-bracket} proves
\eqref{eq:transport-gauge}. Hence equivalent cocycles have
equivalent transforms, proving \eqref{eq:cohom-action}.
\end{proof}

\subsection{Definition and the obstruction criterion}
The extension $\mathcal E$ determines the distinguished point
\[
0_{\mathcal E}=[\Omega]\quad\text{in }H^2_{\nab}(\g,V).
\]
This symbol denotes a base point, not a zero vector.

\begin{definition}\label{def:Wells}
The \emph{Wells map} of $\mathcal E$ is the map of pointed sets
\begin{equation}\label{eq:Wells}
\varpi_{\mathcal E}:G\longrightarrow
 \bigl(H^2_{\nab}(\g,V),0_{\mathcal E}\bigr),
\qquad
\varpi_{\mathcal E}(c)=[{}^c\Omega].
\end{equation}
We abbreviate it to $\varpi$. Its kernel means the fiber
\[
\ker\varpi=\{c\in G:\varpi(c)=0_{\mathcal E}\}.
\]
\end{definition}

The value \eqref{eq:Wells} measures whether the transported extension
has the same class as the original one. In general it must not be
written as a difference of two elements of $H^2_{\nab}(\g,V)$:
there is no specified group structure on this set. Even the
componentwise difference of two cocycles need not satisfy the
nonlinear identities \eqref{eq:NA0}--\eqref{eq:NArs}.

\begin{theorem}[Wells obstruction]\label{thm:Wells-obstruction}
For $c=(\beta,\alpha)\in G$, the following three conditions are
equivalent:
\begin{enumerate}
\item $c$ is induced by an automorphism in $\Aut_{V}(E)$;
\item $\varpi(c)=0_{\mathcal E}$;
\item there is $h\in\Hom(\g,V)$ such that
      $ {}^c\Omega=\Omega^h$.
\end{enumerate}
For each solution $h$, the corresponding automorphism is
\begin{equation}\label{eq:normalized-lift}
\gamma_{c,h}(s(x)+a)=s(\alpha x)+h(\alpha x)+\beta a.
\end{equation}
\end{theorem}

\begin{proof}
Let $\gamma$ induce $c$ and use the map
$\lambda=\gamma s-s\alpha$ of
Theorem~\ref{thm:inducibility}. Put $h=\lambda\alpha^{-1}$.
In \eqref{eq:lift-rho}, replace each $x_i$ by $\alpha^{-1}x_i$
and each $a_j$ by $\beta^{-1}a_j$. Restoring the unshifted
term gives
\begin{align*}
{}^c\rho_r(x_1,\ldots,x_m)(a_1,\ldots,a_r)
 &=[x_1+h(x_1),\ldots,x_m+h(x_m),a_1,\ldots,a_r]_{\Omega}\\
 &=\rho_r^h(x_1,\ldots,x_m)(a_1,\ldots,a_r).
\end{align*}
The same substitutions in \eqref{eq:lift-omega} give
\begin{align*}
{}^c\omega(x_1,\ldots,x_n)
 &=[x_1+h(x_1),\ldots,x_n+h(x_n)]_{\Omega}
   -[x_1,\ldots,x_n]_{\g}\\
 &\quad-h([x_1,\ldots,x_n]_{\g})
 =\omega^h(x_1,\ldots,x_n).
\end{align*}
Thus $ {}^c\Omega=\Omega^h$.

If this equality holds, $S_c:E_{\Omega}\to E_{{}^c\Omega}$
preserves brackets by \eqref{eq:transport-bracket}, and
$T_h:E_{\Omega^h}\to E_{\Omega}$ preserves brackets by
\eqref{eq:gauge-bracket}. Their composite $T_hS_c$ therefore
defines an automorphism. In elements,
\begin{align*}
T_hS_c(x+a)&=\alpha x+h(\alpha x)+\beta a,\\
(T_hS_c)^{-1}(x+a)&=
 \alpha^{-1}x+\beta^{-1}(a-h(x)).
\end{align*}
Transporting this formula by $F_s$ gives
\eqref{eq:normalized-lift} and proves inducibility.
Finally, Definition~\ref{def:Hnab} says exactly that
$[{}^c\Omega]=[\Omega]$ when $ {}^c\Omega=\Omega^h$ for some $h$.
This proves all three assertions.
\end{proof}

The theorem gives an explicit procedure: form
\eqref{eq:twist-omega}--\eqref{eq:twist-rho}, solve
\eqref{eq:gauge-rho}--\eqref{eq:gauge-omega} with $f=h$,
and insert the solution into \eqref{eq:normalized-lift}.
The equations for the mixed maps are part of the obstruction, even
when $\omega=0$.

\subsection{Independence of the section}
\begin{proposition}\label{prop:Wells-section}
Both $0_{\mathcal E}$ and $\varpi_{\mathcal E}$ are independent
of the section. If $s'=s+f$ and $h$ is a solution for $c$ with
respect to $s$, then the solution defining the same lift with
respect to $s'$ is
\begin{equation}\label{eq:lift-section-change}
h'=h+{}^cf-f.
\end{equation}
\end{proposition}

\begin{proof}
The new datum is $\Omega'=\Omega^f$. Consequently,
\[
[\Omega']=[\Omega],\qquad
[{}^c\Omega']=[({}^c\Omega)^{{}^cf}]=[{}^c\Omega],
\]
by \eqref{eq:transport-gauge}. This proves independence of the
base point and of every value of $\varpi$.
More explicitly, if $ {}^c\Omega=\Omega^h$, then
\begin{align*}
{}^c\Omega'
 &=({}^c\Omega)^{{}^cf}
   =(\Omega^h)^{{}^cf}
   =\Omega^{h+{}^cf}\\
 &=(\Omega^f)^{h+{}^cf-f}=(\Omega')^{h'}.
\end{align*}
At the level of elements,
\begin{align*}
\gamma_{c,h}(s'(x)+a)
 &=\gamma_{c,h}(s(x)+f(x)+a)\\
 &=s(\alpha x)+h(\alpha x)+\beta f(x)+\beta a\\
 &=s'(\alpha x)
   +\bigl(h+\beta f\alpha^{-1}-f\bigr)(\alpha x)+\beta a.
\end{align*}
This proves \eqref{eq:lift-section-change}.
\end{proof}

If $\Phi:E\to E'$ is an equivalence of extensions, then
$\gamma\mapsto\Phi\gamma\Phi^{-1}$ identifies their automorphism
groups over $V$ and commutes with $\tau$. Choosing $s'=\Phi s$
makes the two data equal, since
\[
[\Phi s(x_1),\ldots,\Phi s(x_{n-r}),a_1,\ldots,a_r]
 =\Phi[s(x_1),\ldots,s(x_{n-r}),a_1,\ldots,a_r],
\]
and $\Phi|_{V}=\id$. Their Wells maps therefore agree as well.

\subsection{Orbit identities and fibers}
\begin{proposition}\label{prop:Wells-orbit}
For $c,d\in G$,
\begin{equation}\label{eq:Wells-orbit-law}
\varpi(cd)={}^c\varpi(d).
\end{equation}
The set $K_{\mathcal E}=\ker\varpi$ is the stabilizer of
$[\Omega]$, hence a subgroup. Each fiber is a left coset:
\begin{equation}\label{eq:Wells-fibers}
\varpi(c)=\varpi(d)\ \text{ if and only if }\ c\in dK_{\mathcal E}.
\end{equation}
Thus $G/K_{\mathcal E}$ is in bijection with the orbit
$G\cdot[\Omega]$.
\end{proposition}

\begin{proof}
The action law gives
\[
\varpi(cd)=[{}^{cd}\Omega]=[{}^c({}^d\Omega)]
          ={}^c[{}^d\Omega].
\]
If $c,d$ fix $[\Omega]$, then
$ {}^{cd}[\Omega]={}^c({}^d[\Omega])=[\Omega]$.
Applying $c^{-1}$ to $ {}^c[\Omega]=[\Omega]$ shows that
$c^{-1}$ also fixes $[\Omega]$.
If $\varpi(c)=\varpi(d)$, application of $d^{-1}$ gives
$ {}^{d^{-1}c}[\Omega]=[\Omega]$, so $c=dk$ for
$k\in K_{\mathcal E}$. Conversely, for such $k$,
$\varpi(dk)={}^d({}^k[\Omega])={}^d[\Omega]=\varpi(d)$.
This proves the fiber description and the stated bijection.
\end{proof}

Equation \eqref{eq:Wells-orbit-law} is an orbit identity. The
$G$-action need not fix the chosen base point; it does so exactly
on $K_{\mathcal E}$. We therefore do not interpret
\eqref{eq:Wells-orbit-law} as an additive crossed-homomorphism
identity on a non-abelian cohomology group. An additive map on
the mixed-compatible subgroup is constructed below; its abelian
specialization appears in Section~\ref{sec:abelian}.

\subsection{All lifts and their multiplication}
Define
\begin{equation}\label{eq:Aomega}
\mathcal A_{\Omega}
 =\{(c,h)\in G\times\Hom(\g,V):{}^c\Omega=\Omega^h\}.
\end{equation}

\begin{theorem}\label{thm:all-lifts}
The bijection $(c,h)\mapsto\gamma_{c,h}$ identifies
$\mathcal A_{\Omega}$ with $\Aut_{V}(E)$. The multiplication
and inverse are
\begin{align}
(c,h)(d,k)&=(cd,h+{}^ck),\label{eq:lifts-product}\\
(c,h)^{-1}&=(c^{-1},-{}^{c^{-1}}h).
\label{eq:lifts-inverse}
\end{align}
If $c$ has one solution $h_0$, all its solutions are
\begin{equation}\label{eq:lifts-torsor}
h_0+Z_{\Omega},\qquad
Z_{\Omega}=\{f\in\Hom(\g,V):\Omega^f=\Omega\}.
\end{equation}
Conjugation on this kernel is
\begin{equation}\label{eq:conjugation}
\gamma_{c,h}\,T_f\,\gamma_{c,h}^{-1}=T_{{}^cf},
\qquad f\in Z_{\Omega}.
\end{equation}
Here the shears are viewed on $E$ using $F_s$.
\end{theorem}

\begin{proof}
The bijection follows from Theorem~\ref{thm:Wells-obstruction}
and uniqueness of $\lambda=\gamma s-s\alpha$.
Let $c=(\beta,\alpha)$ and $d=(\beta',\alpha')$. On $s(x)+a$,
\begin{align*}
\gamma_{c,h}\gamma_{d,k}(s(x)+a)
 &=\gamma_{c,h}(s(\alpha'x)+k(\alpha'x)+\beta'a)\\
 &=s(\alpha\alpha'x)+h(\alpha\alpha'x)
   +\beta k(\alpha'x)+\beta\beta'a\\
 &=s(\alpha\alpha'x)
   +(h+\beta k\alpha^{-1})(\alpha\alpha'x)+\beta\beta'a.
\end{align*}
This is \eqref{eq:lifts-product}. At the level of cocycles,
\begin{align*}
{}^{cd}\Omega
 &={}^c(\Omega^k)
   =({}^c\Omega)^{{}^ck}
   =(\Omega^h)^{{}^ck}
   =\Omega^{h+{}^ck},
\end{align*}
which also proves closure.
Putting $k=-{}^{c^{-1}}h$ gives
$h+{}^ck=0$ and proves \eqref{eq:lifts-inverse}.

If $h$ and $h_0$ are solutions for $c$, then
$\Omega^h=\Omega^{h_0}$. Applying the gauge change $-h_0$
to both sides gives $\Omega^{h-h_0}=\Omega$.
Conversely, $\Omega^f=\Omega$ implies
$\Omega^{h_0+f}=(\Omega^f)^{h_0}=\Omega^{h_0}$.
This proves \eqref{eq:lifts-torsor}.
Finally, the product and inverse formulas give
\[
(c,h)(1,f)(c,h)^{-1}
 =(1,h+{}^cf- h)=(1,{}^cf),
\]
proving \eqref{eq:conjugation}.
\end{proof}

\subsection{An additive Wells map for compatible non-abelian data}
The pointed map $\varpi$ applies to all pairs in
$G=\Aut(V)\times\Aut(\g)$. The relative cohomology constructed
in Section~\ref{sec:compatible-derivations} gives a more precise
additive map after the mixed lifting equations have been solved.
Define
\begin{equation}\label{eq:full-compatible-group}
\mathcal C_\rho
 =\{c=(\beta,\alpha)\in G:
       {}^c\rho_r=\rho_r^b\ (1\leq r\leq n-1)
       \text{ for some }b:\g\to V\}.
\end{equation}
Equivalently, the invertible map
\[
\Gamma_{c,b}(s(x)+a)=s(\alpha x)+b(\alpha x)+\beta a
\]
preserves every mixed bracket. This follows by substituting
section and kernel elements, exactly as in the mixed part of
the lifting equations.

\begin{proposition}\label{prop:relative-group-action}
The set $\mathcal C_\rho$ is a subgroup of $G$.
It acts on $\mathcal H_\rho^2$ by
\[
c[\theta]=[{}^c\theta],\qquad
{}^c\theta(x_1,\ldots,x_n)
 =\beta\theta(\alpha^{-1}x_1,\ldots,\alpha^{-1}x_n).
\]
For $f\in\mathcal F_\rho$,
\begin{equation}\label{eq:relative-group-boundary}
{}^cf=\beta f\alpha^{-1}\in\mathcal F_\rho,\qquad
d_\rho({}^cf)={}^c(d_\rho f).
\end{equation}
\end{proposition}

\begin{proof}
Invertible maps preserving $V$ and all mixed brackets are
closed under composition and inverse: their images and inverse
images of mixed tuples are again mixed. Direct evaluation gives
\[
\Gamma_{c,b}\Gamma_{d,u}=\Gamma_{cd,b+{}^cu},
\qquad
\Gamma_{c,b}^{-1}=\Gamma_{c^{-1},-{}^{c^{-1}}b}.
\]
Hence their induced pairs form the subgroup $\mathcal C_\rho$.

Preservation of mixed brackets gives
$\beta A_\rho=A_\rho$. On a bracket with its last entry
$a\in A_\rho$, all terms containing $b(x_i)$ vanish, so
\[
\beta\rho_A(\alpha^{-1}x_1,\ldots,\alpha^{-1}x_{n-1})
 =\rho_A(x_1,\ldots,x_{n-1})\beta.
\]
Substitution into the four terms of \eqref{eq:relative-d2}
yields
\[
(d_\rho\,{}^c\theta)(X;Y)
 =\beta(d_\rho\theta)(\alpha^{-1}X;\alpha^{-1}Y).
\]
Thus the action preserves $Z_\rho^2$.
Conjugating $\delta_f$ by a mixed-bracket-preserving map gives
\[
\Gamma_{c,b}\delta_f\Gamma_{c,b}^{-1}
 =\delta_{\beta f\alpha^{-1}}.
\]
The conjugate still satisfies the derivation identity on mixed
brackets. Lemma~\ref{lem:partial-shears} gives
$ {}^cf\in\mathcal F_\rho$.

To verify the boundary formula without assuming $f(\g)\subseteq A_\rho$,
write $ {}^c\Omega=\Omega^b+\theta_c$, where
$\theta_c={}^c\omega-\omega^b$. The two cocycles have the same
mixed maps, so Proposition~\ref{prop:fixed-mixed-differences}
and section independence imply $\theta_c\in Z_\rho^2$.
Adding a pure cochain commutes with every gauge change:
\[
(\Omega+\theta)^u=\Omega^u+\theta,
\]
because the input projections to $\g$ are unchanged and
$T_{-u}$ is the identity on the output in $V$.
Gauge covariance and Lemma~\ref{lem:partial-shears} now give
\begin{align*}
{}^c(\Omega^f)&=\Omega^b+\theta_c+{}^c(d_\rho f),\\
({}^c\Omega)^{{}^cf}
 &=(\Omega^b+\theta_c)^{{}^cf}
   =\Omega^b+\theta_c+d_\rho({}^cf).
\end{align*}
The left sides agree by \eqref{eq:transport-gauge}; the boundary
formula follows by comparing pure components.
The action therefore descends to $\mathcal H_\rho^2$.
\end{proof}

\begin{theorem}\label{thm:relative-group-Wells}
For $c\in\mathcal C_\rho$, choose any mixed correction $b$.
Then
\begin{equation}\label{eq:relative-group-Wells}
\varpi_\rho(c)=[{}^c\omega-\omega^b]\in\mathcal H_\rho^2
\end{equation}
is independent of $b$ and of the section. It satisfies
\begin{equation}\label{eq:relative-group-crossed}
\varpi_\rho(cd)=\varpi_\rho(c)+c\varpi_\rho(d).
\end{equation}
There is an exact sequence
\begin{equation}\label{eq:relative-group-exact}
1\longrightarrow Z_\Omega\longrightarrow\Aut_V(E)
\xrightarrow{\tau}\mathcal C_\rho
\xrightarrow{\varpi_\rho}\mathcal H_\rho^2.
\end{equation}
For a chosen $b$, the pair $c$ lifts precisely when
\[
{}^c\omega-\omega^b=d_\rho f
\quad\text{for some }f\in\mathcal F_\rho;
\]
in that case all lifts have normalized correction $b+f+Z_\Omega$.
\end{theorem}

\begin{proof}
The cocycle assertion was proved above.
If $b'$ is another mixed correction, the section-independent
description of $\mathcal F_\rho$ gives $b'-b\in\mathcal F_\rho$.
The affine gauge formula, applied at $\Omega^b$, gives
\[
\omega^{b'}-\omega^b=d_\rho(b'-b).
\]
Thus \eqref{eq:relative-group-Wells} is well defined.

Choose corrections $b_c,b_d$, and write
$ {}^c\Omega=\Omega^{b_c}+\theta_c$ and
$ {}^d\Omega=\Omega^{b_d}+\theta_d$. Gauge covariance gives
\begin{align*}
{}^{cd}\Omega
 &={}^c(\Omega^{b_d}+\theta_d)\\
 &=({}^c\Omega)^{{}^ct_d}+{}^c\theta_d\\
 &=(\Omega^{b_c}+\theta_c)^{{}^ct_d}+{}^c\theta_d\\
 &=\Omega^{b_c+{}^ct_d}+\theta_c+{}^c\theta_d.
\end{align*}
The sum $b_c+{}^ct_d$ is a mixed correction for $cd$.
Taking classes proves \eqref{eq:relative-group-crossed}.

The condition $\varpi_\rho(c)=0$ means
$\theta_c=d_\rho f$ for some $f\in\mathcal F_\rho$.
Then
\[
{}^c\Omega=\Omega^{b_c}+\theta_c
          =\Omega^{b_c+f},
\]
so Theorem~\ref{thm:Wells-obstruction} supplies a lift.
Conversely, any lift has a mixed correction $h$, so
$h-b_c\in\mathcal F_\rho$ and
$\theta_c=d_\rho(h-b_c)$. Its class is zero.
The differences between solutions are $Z_\Omega$, proving
the last assertion and exactness at $\mathcal C_\rho$.
The kernel of $\tau$ consists of the shears
$s(x)+a\mapsto s(x)+f(x)+a$ with $\Omega^f=\Omega$.
By Lemma~\ref{lem:partial-shears} these are exactly $f\in Z_\Omega$,
which proves the remaining exactness assertions.

Finally, replace $\Omega$ by $\Omega'=\Omega^k$.
The same partial automorphism has correction
$b'_c=b_c+{}^ck-k$. Indeed,
\[
{}^c\Omega'=({}^c\Omega)^{{}^ck}
 =\Omega^{b_c+{}^ck}+\theta_c
 =(\Omega')^{b'_c}+\theta_c.
\]
The same cocycle $\theta_c$ represents the new obstruction,
and Proposition~\ref{prop:relative-section} identifies the
cohomology spaces. This proves section independence.
\end{proof}

The pointed and additive maps detect the same inducible pairs:
$\varpi(c)=[\Omega]$ precisely when
$c\in\mathcal C_\rho$ and $\varpi_\rho(c)=0$.
There is no subtraction of arbitrary non-abelian cohomology
classes in this assertion; subtraction is performed only
after the mixed data have been aligned.
The corresponding derivation equations are
$L_d\rho_r=\partial_r\lambda$ and
$L_d\omega=d_\rho\lambda$.
Their derivation in Section~\ref{sec:derivation-wells} is a
direct comparison of brackets, with a single correction
inserted in each quotient slot.
In both the group and Lie algebra settings, solving the
mixed equations leaves the pure cohomological obstruction.

\section{Wells exact sequences and splitting}\label{sec:exact}

\subsection{The kernel as a vector space}
For $f\in\Hom(\g,V)$ define a linear map on $E$ by
\begin{equation}\label{eq:delta-f}
\delta_f(s(x)+a)=f(x),\qquad
\delta_f=i f p,\qquad \delta_f^2=0.
\end{equation}
In these coordinates $T_f=\id+\delta_f$.

\begin{proposition}\label{prop:kernel-linear}
The following conditions on $f$ are equivalent:
\begin{enumerate}
\item $f\in Z_{\Omega}$;
\item $\id+\delta_f\in\Aut_{V}(E)$;
\item $\delta_f\in\Der(E)$;
\item $f$ satisfies the following equations, with
      $1\leq r\leq n-1$ in the second:
\begin{align}
f([x_1,\ldots,x_n]_{\g})
 &=\sum_{i=1}^n(-1)^{n-i}
 \rho(x_1,\ldots,\widehat x_i,\ldots,x_n)f(x_i),
 \label{eq:kernel-pure}\\
0&=\sum_{i=1}^{n-r}(-1)^{n-r-i}
 \rho_{r+1}(x_1,\ldots,\widehat x_i,\ldots,x_{n-r})
       (f(x_i),a_1,\ldots,a_r).
 \label{eq:kernel-mixed}
\end{align}
\end{enumerate}
In particular, $Z_{\Omega}$ is a vector subspace of
$\Hom(\g,Z(V))$, independent of the section, and
\begin{equation}\label{eq:kernel-isom}
(Z_{\Omega},+)\ \cong\ \ker\tau,\qquad f\longmapsto T_f.
\end{equation}
\end{proposition}

\begin{proof}
The equations in the fourth condition are
$d_\rho f=0$ and $\partial_r f=0$.
By Lemma~\ref{lem:partial-shears}, they say that all mixed
components and the pure component are fixed by $f$, or equivalently
$\Omega^f=\Omega$. The gauge identity makes this equivalent to
$T_f$ preserving every bracket.
The same equations are the pure and mixed derivation identities
for $\delta_f$, as proved in
Theorem~\ref{thm:relative-derivation-exact}.
Hence the four conditions agree.
They are linear, and the equation with $r=n-1$ gives
$f(\g)\subseteq Z(V)$.
Section independence is Proposition~\ref{prop:relative-section}.
Every element of $\ker\tau$ has the unique form $T_f$, and
$T_fT_k=T_{f+k}$, proving \eqref{eq:kernel-isom}.
\end{proof}

\begin{remark}
Although the gauge action on all cocycles is nonlinear, its
stabilizer is described by the linear equations
\eqref{eq:kernel-pure}--\eqref{eq:kernel-mixed}. It is the
additive group of derivations of $E$ that annihilate $V$ and have
image in $V$. This conclusion uses characteristic zero.
\end{remark}

\subsection{The main sequence and two restricted sequences}
\begin{theorem}\label{thm:Wells-exact}
The extension \eqref{eq:extension} gives the exact sequence
\begin{equation}\label{eq:Wells-exact}
1\longrightarrow Z_{\Omega}
 \xrightarrow{\,f\mapsto T_f\,}\Aut_{V}(E)
 \xrightarrow{\tau}\Aut(V)\times\Aut(\g)
 \xrightarrow{\varpi}
 \bigl(H^2_{\nab}(\g,V),[\Omega]\bigr).
\end{equation}
The first three nontrivial terms are groups; exactness at the
third means $\im\tau=\varpi^{-1}([\Omega])$.
Equivalently, there is a short exact sequence of groups
\begin{equation}\label{eq:short-exact}
1\longrightarrow Z_{\Omega}\longrightarrow\Aut_{V}(E)
 \xrightarrow{\tau}K_{\mathcal E}\longrightarrow1.
\end{equation}
The induced $K_{\mathcal E}$-action on $Z_{\Omega}$ is
$ {}^{(\beta,\alpha)}f=\beta f\alpha^{-1}$.
\end{theorem}

\begin{proof}
If $T_f=\id$, then $f(x)=T_f(s(x))-s(x)=0$, so the first
map is injective. Proposition~\ref{prop:kernel-linear} identifies
its image with $\ker\tau$. If $c=\tau(\gamma)$,
Theorem~\ref{thm:Wells-obstruction} supplies $h$ with
$ {}^c\Omega=\Omega^h$, so $\varpi(c)=[\Omega]$.
If $\varpi(c)=[\Omega]$, the same theorem constructs the
automorphism \eqref{eq:normalized-lift} with image $c$.
Thus $\im\tau=K_{\mathcal E}$, proving both sequences.
Formula \eqref{eq:conjugation} gives the asserted action, and
shows directly that it preserves $Z_{\Omega}$.
\end{proof}

\begin{corollary}\label{cor:centerless}
If $Z(V)=0$, each inducible pair has a unique lift and
$\Aut_{V}(E)\cong K_{\mathcal E}$.
\end{corollary}

\begin{proof}
Proposition~\ref{prop:kernel-linear} gives
$Z_{\Omega}\subseteq\Hom(\g,Z(V))=0$. The conclusion follows
from \eqref{eq:short-exact}.
\end{proof}

To fix one of the two prescribed automorphisms, set
\begin{align*}
\Aut_{V}^{V}(E)
 &=\{\gamma\in\Aut_{V}(E):\gamma|_{V}=\id_{V}\},\\
\Aut_{V}^{\g}(E)
 &=\{\gamma\in\Aut_{V}(E):\bar\gamma=\id_{\g}\},\\
\varpi_{\g}(\alpha)&=\varpi(\id_{V},\alpha),\qquad
\varpi_{V}(\beta)=\varpi(\beta,\id_{\g}).
\end{align*}

\begin{theorem}\label{thm:restricted-exact}
There are exact sequences of groups and pointed sets
\begin{align}
1\longrightarrow Z_{\Omega}\longrightarrow
\Aut_{V}^{V}(E)
\xrightarrow{\gamma\mapsto\bar\gamma}\Aut(\g)
\xrightarrow{\varpi_{\g}}
\bigl(H^2_{\nab}(\g,V),[\Omega]\bigr),
\label{eq:restricted-g}\\
1\longrightarrow Z_{\Omega}\longrightarrow
\Aut_{V}^{\g}(E)
\xrightarrow{\gamma\mapsto\gamma|_{V}}\Aut(V)
\xrightarrow{\varpi_{V}}
\bigl(H^2_{\nab}(\g,V),[\Omega]\bigr).
\label{eq:restricted-h}
\end{align}
Thus $\alpha$ lifts while fixing $V$ pointwise precisely when
$\varpi_{\g}(\alpha)=[\Omega]$, and $\beta$ extends while inducing
the identity on $\g$ precisely when
$\varpi_{V}(\beta)=[\Omega]$.
\end{theorem}

\begin{proof}
In $\Aut_{V}^{V}(E)$, having $\bar\gamma=\id_{\g}$ means that
both induced maps are the identity. Its kernel is therefore
$\{T_f:f\in Z_{\Omega}\}$.
If $\varpi_{\g}(\alpha)=[\Omega]$, a solution of
$ {}^{(\id,\alpha)}\Omega=\Omega^h$ gives
\[
\gamma(s(x)+a)=s(\alpha x)+h(\alpha x)+a.
\]
It fixes $V$ and induces $\alpha$. Conversely, any such
automorphism satisfies this cocycle equation by the lifting
criterion. This proves exactness of \eqref{eq:restricted-g}.

In $\Aut_{V}^{\g}(E)$ the kernel of restriction to $V$
again consists of the automorphisms inducing the identity on
both terms, hence of the same shears $T_f$.
When $\varpi_{V}(\beta)=[\Omega]$, a solution of
$ {}^{(\beta,\id)}\Omega=\Omega^h$ gives
\[
\gamma(s(x)+a)=s(x)+h(x)+\beta a.
\]
It induces the identity on $\g$ and restricts to $\beta$.
Every automorphism with these properties supplies such a solution,
which proves exactness of \eqref{eq:restricted-h}.
\end{proof}

\subsection{Splitting and the factor set of chosen lifts}
The stabilizer of the datum itself is
\begin{equation}\label{eq:strict-stabilizer}
C_{\Omega}=\{c\in G:{}^c\Omega=\Omega\}\subseteq K_{\mathcal E}.
\end{equation}
It may be smaller than the stabilizer of its class.

\begin{proposition}\label{prop:strict-splitting}
The inverse image of $C_{\Omega}$ is
\begin{equation}\label{eq:strict-splitting}
\tau^{-1}(C_{\Omega})\cong Z_{\Omega}\rtimes C_{\Omega}.
\end{equation}
The section sends $(\beta,\alpha)$ to
$s(x)+a\mapsto s(\alpha x)+\beta a$.
\end{proposition}

\begin{proof}
For $c\in C_{\Omega}$, the pair $(c,0)$ lies in
$\mathcal A_{\Omega}$ and
$(c,0)(d,0)=(cd,0)$. Every lift of such $c$ is
$(c,f)$ with $f\in Z_{\Omega}$ by \eqref{eq:lifts-torsor}.
Its multiplication is
$ (c,f)(d,k)=(cd,f+{}^ck)$, which is the asserted semidirect
product.
\end{proof}

For the full group $K_{\mathcal E}$ a further obstruction may
remain. Choose, for each $c\in K_{\mathcal E}$, a solution $h_c$
of $ {}^c\Omega=\Omega^{h_c}$, with $h_1=0$. Define
\begin{equation}\label{eq:kappa}
\kappa(c,d)=h_c+{}^ch_d-h_{cd}.
\end{equation}

\begin{theorem}\label{thm:factor-set}
The map $\kappa:K_{\mathcal E}\times K_{\mathcal E}\to Z_{\Omega}$
is a normalized group $2$-cocycle:
\begin{align}
\kappa(1,c)=\kappa(c,1)&=0,\notag\\
\kappa(c,d)+\kappa(cd,e)
 &={}^c\kappa(d,e)+\kappa(c,de).
 \label{eq:kappa-identity}
\end{align}
Its class in the ordinary group cohomology
$H^2_{\mathrm{grp}}(K_{\mathcal E},Z_{\Omega})$ is independent
of the chosen lifts. Sequence \eqref{eq:short-exact} splits
precisely when this class is zero.
\end{theorem}

\begin{proof}
By \eqref{eq:lifts-product}, both $h_c+{}^ch_d$ and $h_{cd}$
solve the lifting equation for $cd$. Their difference lies in
$Z_{\Omega}$ by \eqref{eq:lifts-torsor}.
The normalizations follow from $h_1=0$. Expanding the left side
of \eqref{eq:kappa-identity} gives
\[
h_c+{}^ch_d-h_{cd}+h_{cd}+{}^{cd}h_e-h_{cde}
 =h_c+{}^ch_d+{}^{cd}h_e-h_{cde}.
\]
Expanding its right side gives the same expression:
\begin{align*}
{}^c(h_d+{}^dh_e-h_{de})+h_c+{}^ch_{de}-h_{cde}
 &=h_c+{}^ch_d+{}^{cd}h_e-h_{cde}.
\end{align*}
If another choice is $h'_c=h_c+z_c$,
$z_c\in Z_{\Omega}$, $z_1=0$, then
\begin{equation}\label{eq:kappa-change}
\kappa'(c,d)=\kappa(c,d)+z_c+{}^cz_d-z_{cd}.
\end{equation}
This is precisely a group coboundary change, so the class is
unchanged.

A homomorphic section of $\tau$ has the form
$c\mapsto\gamma_{c,h_c}$ with
$h_{cd}=h_c+{}^ch_d$, by \eqref{eq:lifts-product}.
It therefore makes $\kappa$ zero.
Conversely, if the class of $\kappa$ is zero, choose $z_c$ in
\eqref{eq:kappa-change} so that $\kappa'=0$.
The maps $\gamma_{c,h'_c}$ then satisfy
$\gamma_{c,h'_c}\gamma_{d,h'_d}=\gamma_{cd,h'_{cd}}$ and form
a homomorphic section.
\end{proof}

This group cohomology obstruction is distinct from $\varpi$:
$\varpi$ decides whether each individual pair lifts, while
$[\kappa]$ decides whether the lifts can be chosen multiplicatively.
The existence of an algebra section of \eqref{eq:extension}
asserts $\omega=0$ in suitable coordinates; it does not remove
the mixed-component lifting equations.

\section{Sparse crossed products as a special case}\label{sec:sparse}
This section assumes $n\geq3$.
We impose
\begin{equation}\label{eq:sparse-components}
\rho_1=\rho,\qquad \rho_{n-1}=\vartheta,\qquad
\rho_r=0\quad(2\leq r\leq n-2).
\end{equation}
For $n=3$ the vanishing range is empty, so every extension is
sparse relative to every section. For $n\geq4$ this condition
depends on the chosen section.

\subsection{The bracket and the two representation meanings}
\begin{definition}\label{def:crossed-system}
A sparse crossed system is a triple of maps
\[
\rho:\bigwedge^{n-1}\g\to\End(V),\qquad
\vartheta:\g\to\Hom(\bigwedge^{n-1}V,V),\qquad
\omega:\bigwedge^n\g\to V
\]
whose operation
\begin{align}
&[x_1+u_1,\ldots,x_n+u_n]_{\rho,\vartheta,\omega}\notag\\
&=[x_1,\ldots,x_n]_\g+[u_1,\ldots,u_n]_V+\omega(x_1,\ldots,x_n)
\notag\\
&\quad+\sum_{i=1}^n(-1)^{n-i}
\rho(x_1,\ldots,\widehat x_i,\ldots,x_n)u_i
\notag\\
&\quad+\sum_{i=1}^n(-1)^{i-1}
\vartheta(x_i)(u_1,\ldots,\widehat u_i,\ldots,u_n)
\label{eq:crossed-bracket}
\end{align}
is an $n$-Lie bracket. Its algebra is denoted
$\g\#_{\rho,\vartheta,\omega}V$.
An extension is sparse if it has a section with these components.
\end{definition}
Its exact structural conditions are obtained by substituting
\eqref{eq:sparse-components} in
Theorem~\ref{thm:cocycle-criterion}; equivalently one may use
the sparse crossed-product criterion of~\cite{ZhangExtending}.

The following consequences explain the two nonzero mixed maps.
For $n\geq4$, $\rho_2=0$, and \eqref{eq:curvature} together with
\eqref{eq:NAr0} for $r=1$ gives
\begin{align}
[\rho(X),\rho(y_1,\ldots,y_{n-1})]
 &=\sum_{i=1}^{n-1}\rho(y_1,\ldots,[X,y_i],\ldots,y_{n-1}),
\notag\\
\rho(x_1,\ldots,x_{n-2},[y_1,\ldots,y_n])
 &=\sum_{i=1}^n(-1)^{n-i}
\rho(y_1,\ldots,\widehat y_i,\ldots,y_n)
\rho(x_1,\ldots,x_{n-2},y_i).
\label{eq:sparse-representation}
\end{align}
Thus $\rho$ is a representation on all of $V$, its values are
derivations of $V$, and $d_\rho\omega=0$.
For $n=3$ the two defects are instead
\begin{align}
&\bigl([\rho(x_1,x_2),\rho(y_1,y_2)]
-\rho([x_1,x_2,y_1],y_2)-\rho(y_1,[x_1,x_2,y_2])\bigr)u
\notag\\
&\quad=\vartheta(y_1)(\omega(x_1,x_2,y_2),u)
-\vartheta(y_2)(\omega(x_1,x_2,y_1),u),\notag\\
&\bigl(\rho(x,[y_1,y_2,y_3])
-\rho(y_2,y_3)\rho(x,y_1)+\rho(y_1,y_3)\rho(x,y_2)
-\rho(y_1,y_2)\rho(x,y_3)\bigr)u
\notag\\
&\quad=-\vartheta(x)(\omega(y_1,y_2,y_3),u).
\label{eq:sparse-ternary-defects}
\end{align}
The second equation shows that $\rho$ is a representation
precisely when its displayed right-hand side is zero.

The operation $\vartheta$ is not assumed to be an ordinary
representation of $\g$. Its contractions are derivations of
$V$, and it represents the intrinsic module extension of
Section~\ref{sec:highest-representation}:
\[
D_{x;u_1,\ldots,u_{n-2}}v=\vartheta(x)(u_1,\ldots,u_{n-2},v),
\qquad
\chi_s(U)x=(-1)^{n-1}\vartheta(x)(U).
\]
Their interaction is the following instance of \eqref{eq:NA1}:
\begin{align}
\rho(X)\vartheta(y)(u_1,\ldots,u_{n-1})
&-\sum_j\vartheta(y)(u_1,\ldots,\rho(X)u_j,\ldots,u_{n-1})
\notag\\
&=\vartheta([X,y])(u_1,\ldots,u_{n-1})
  +[\omega(X,y),u_1,\ldots,u_{n-1}]_V.
\label{eq:sparse-rho-theta}
\end{align}
Thus the inner operation determined by $\omega$ is precisely
the equivariance defect of $\vartheta$.

\subsection{Changes between sparse sections}
The general change-of-section formulas apply without requiring
the new section to be sparse. The zero intermediate components
must be checked as well.

\begin{theorem}\label{thm:sparse-section-change}
The extreme and pure components of the new section are
\begin{align}
\vartheta^f(x)(u_1,\ldots,u_{n-1})
 &=
\vartheta(x)(u_1,\ldots,u_{n-1})
 +[f(x),u_1,\ldots,u_{n-1}]_V,
\label{eq:sparse-gauge-theta}\\
\rho^f(x_1,\ldots,x_{n-1})u
 &=\rho(x_1,\ldots,x_{n-1})u\notag\\
 &\quad+\sum_{i=1}^{n-1}(-1)^{i-1}
\vartheta(x_i)(f(x_1),\ldots,\widehat{f(x_i)},\ldots,f(x_{n-1}),u)
\notag\\
 &\quad+[f(x_1),\ldots,f(x_{n-1}),u]_V,
\label{eq:sparse-gauge-rho}\\
\omega^f(x_1,\ldots,x_n)
 &=\omega(x_1,\ldots,x_n)+(d_{\rho}f)(x_1,\ldots,x_n)\notag\\
 &\quad+\sum_{i=1}^n(-1)^{i-1}
\vartheta(x_i)(f(x_1),\ldots,\widehat{f(x_i)},\ldots,f(x_n))
\notag\\
 &\quad+[f(x_1),\ldots,f(x_n)]_V,
\label{eq:sparse-gauge-omega}
\end{align}
where
\[
(d_{\rho}f)(x_1,\ldots,x_n)
 =\sum_i(-1)^{n-i}\rho(x_1,\ldots,\widehat x_i,\ldots,x_n)f(x_i)
 -f([x_1,\ldots,x_n]_\g).
\]
For $2\leq m\leq n-2$, the newly created intermediate bracket is
\begin{align}
&[s'(x_1),\ldots,s'(x_m),u_1,\ldots,u_{n-m}]\notag\\
&=\sum_{i=1}^m(-1)^{i-1}
\vartheta(x_i)(f(x_1),\ldots,\widehat{f(x_i)},\ldots,f(x_m),
                                      u_1,\ldots,u_{n-m})\notag\\
&\quad+[f(x_1),\ldots,f(x_m),u_1,\ldots,u_{n-m}]_V.
\label{eq:sparse-gauge-middle}
\end{align}
Thus $s'$ is sparse precisely when every expression in
\eqref{eq:sparse-gauge-middle} vanishes.
\end{theorem}

\begin{proof}
Apply Proposition~\ref{prop:gauge}. In each mixed expansion only
the terms with one remaining quotient entry, with $n-1$ quotient
entries, or with no quotient entry can survive. Moving the sole
quotient entry from slot $i$ to the first slot contributes
$(-1)^{i-1}$. The one-kernel terms in the pure expansion give
$d_\rho f$. This gives the four displayed formulas.
\end{proof}

The middle equations cannot be discarded when comparing two
sparse triples. They express preservation of the zero
components, not an additional choice of notation.
For instance, the first of them reads
\begin{align}
0={}&\vartheta(x)(f(y),u_1,\ldots,u_{n-2})
 -\vartheta(y)(f(x),u_1,\ldots,u_{n-2})\notag\\
&+[f(x),f(y),u_1,\ldots,u_{n-2}]_V.
\label{eq:sparse-middle-two}
\end{align}

\subsubsection{Arity four and higher arity}
\begin{theorem}\label{thm:sparse-arity}
Suppose $n\geq4$.
For $n=4$, condition \eqref{eq:sparse-middle-two} alone
characterizes the changes between sparse sections, and then
\begin{align}
\rho^f(x,y,z)u
 &=\rho(x,y,z)u-\tfrac12[f(x),f(y),f(z),u]_V,
\label{eq:sparse-four-rho}\\
\omega^f(x_1,x_2,x_3,x_4)
 &=\omega(x_1,x_2,x_3,x_4)+(d_{\rho}f)(x_1,x_2,x_3,x_4)\notag\\
 &\quad-[f(x_1),f(x_2),f(x_3),f(x_4)]_V.
\label{eq:sparse-four-omega}
\end{align}
For $n\geq5$, a change preserves sparsity precisely when
\eqref{eq:sparse-middle-two} and
\begin{equation}\label{eq:sparse-middle-three}
[f(x),f(y),f(z),u_1,\ldots,u_{n-3}]_V=0
\end{equation}
hold. In that case
\begin{equation}\label{eq:sparse-higher-affine}
\rho^f=\rho,\qquad \omega^f=\omega+d_{\rho}f.
\end{equation}
The top mixed operation still changes according to
\eqref{eq:sparse-gauge-theta}.
\end{theorem}

\begin{proof}
Assume \eqref{eq:sparse-middle-two}. Fix $x,y,z$ and
$u_1,\ldots,u_{n-3}$, and abbreviate just these four values:
\begin{align*}
A&=\vartheta(x)(f(y),f(z),u_1,\ldots,u_{n-3}),\\
B&=\vartheta(y)(f(x),f(z),u_1,\ldots,u_{n-3}),\\
C&=\vartheta(z)(f(x),f(y),u_1,\ldots,u_{n-3}),\\
K&=[f(x),f(y),f(z),u_1,\ldots,u_{n-3}]_V.
\end{align*}
Use \eqref{eq:sparse-middle-two} first for $x,y$ with
$f(z)$ as the first kernel entry, then for $x,z$ with
$f(y)$, and finally for $y,z$ with $f(x)$.
Alternation gives, respectively,
\[
A-B+K=0,\qquad -A-C-K=0,\qquad -B+C+K=0.
\]
Thus
\[
B=A+K,\qquad C=-A-K,\qquad -2A-K=0,
\]
and therefore
\begin{equation}\label{eq:sparse-half}
\vartheta(x)(f(y),f(z),u_1,\ldots,u_{n-3})
 =-\tfrac12[f(x),f(y),f(z),u_1,\ldots,u_{n-3}]_V.
\end{equation}

For $m\geq3$, apply \eqref{eq:sparse-half} to each summand
in \eqref{eq:sparse-gauge-middle}, or to the corresponding
sum in the extreme and pure formulas.
Moving $f(x_i)$ from the first kernel slot back to its
original position contributes a second factor $(-1)^{i-1}$.
Consequently
\begin{align}
&\sum_{i=1}^m(-1)^{i-1}
\vartheta(x_i)(f(x_1),\ldots,\widehat{f(x_i)},\ldots,f(x_m),
                                            u_1,\ldots,u_{n-m})
\notag\\
&\quad=-\frac m2
[f(x_1),\ldots,f(x_m),u_1,\ldots,u_{n-m}]_V.
\label{eq:sparse-count}
\end{align}
Adding the all-kernel term gives coefficient $1-m/2$.

When $n=4$, the only intermediate case is $m=2$.
In \eqref{eq:sparse-gauge-rho} put $m=3$, which gives
$1-3/2=-1/2$; in \eqref{eq:sparse-gauge-omega} put $m=4$,
which gives $1-4/2=-1$. This proves the two arity-four formulas.

When $n\geq5$, $m=3$ is an intermediate case.
Its vanishing and \eqref{eq:sparse-count} give
\[
-\tfrac12[f(x),f(y),f(z),u_1,\ldots,u_{n-3}]_V=0.
\]
Every all-kernel term with $m\geq3$ changed entries then
vanishes by substituting further values of $f$ for some
$u_j$. Equation \eqref{eq:sparse-count} also annihilates
the corresponding $\vartheta$-sums.
Thus the equations with $m\geq3$ follow from
\eqref{eq:sparse-middle-three}, and both extra terms in
\eqref{eq:sparse-gauge-rho} and \eqref{eq:sparse-gauge-omega}
are zero. Together with the $m=2$ equation, this proves
necessity, sufficiency, and \eqref{eq:sparse-higher-affine}.
\end{proof}

For $n\geq5$ the affine formula in
\eqref{eq:sparse-higher-affine} concerns changes between sparse
sections; it does not assert that all mixed maps are fixed.
In particular, $f$ need not take values in $Z(V)$, and
$\vartheta^f$ may differ from $\vartheta$.
Nevertheless, for $n\geq5$ the representation $\rho$ and
the ordinary cohomology class
\[
[\omega]\in H_{\ab}^2(\g,V;\rho)
\]
are independent of the sparse section. Indeed, $\rho^f=\rho$,
$d_{\rho}\omega=0$, and $\omega^f-\omega=d_{\rho}f$.
This class retains only part of the extension data:
the equations for $\vartheta$ and the zero middle
components remain necessary for equivalence and lifting.

\subsection{The good criterion and relative cohomology}
The general spaces have the following explicit forms:
\begin{align}
A_\rho=\At
 &=\{a\in Z(V):\vartheta(x)(a,u_2,\ldots,u_{n-1})=0
                         \text{ for all }x,u_j\},\notag\\
\mathcal F_\rho=\Ft
 &=\{f\in\Hom(\g,Z(V)):
\vartheta(x)(f(y),u_1,\ldots,u_{n-2})
=\vartheta(y)(f(x),u_1,\ldots,u_{n-2})\}.
\label{eq:sparse-A-F}
\end{align}
Indeed, the highest mixed condition tests centrality in $V$,
and the two-quotient condition gives the displayed equality.
Every other defining condition is zero by sparsity.

Proposition~\ref{prop:good-criterion} says that a sparse section
is good relative to $\mathcal S$ precisely when
$\mathcal S\subseteq\Ft$, with the additional condition
\[
\vartheta(x)(\omega(y_1,y_2,y_3),u)=0
\]
if $n=3$. For $n\geq4$ the representation condition is automatic.
For every $f\in\Ft$, the finite formula is
$\omega^f=\omega+d_\rho f$, and all mixed components remain fixed.
In particular,
\begin{equation}\label{eq:sparse-relative-H}
\mathcal H_\rho^2=\Ht
=\frac{\{\eta:\bigwedge^n\g\to\At:d_\rho\eta=0\}}
             {d_\rho\Ft}.
\end{equation}
If $\dim\g\geq2$, central goodness forces
$\vartheta(x)(a,u_2,\ldots,u_{n-1})=0$ for every $a\in Z(V)$:
for $x\ne0$ choose $y$ independent of $x$ and take
$f(x)=0,\ f(y)=a$ in \eqref{eq:sparse-A-F}.
Conversely this vanishing gives $\Ft=\Hom(\g,Z(V))$;
in the ternary case the additional pure condition above is
still necessary. Under these conditions $\At=Z(V)$.

The third obstruction of Theorem~\ref{thm:third-obstruction}
is also specialized by \eqref{eq:sparse-components}.
It takes values in $H_{\ext}^3(\g,\At;\rho|_{\At})$.
If $n\geq4$, the full action $\rho$ on $V$ is a representation,
so the obstruction cochain is $d_\rho\omega_0$ in the
$V$-coefficient complex. This equality alone does not imply
vanishing in the $\At$-coefficient complex, since the primitive
$\omega_0$ need not take values in $\At$.

\subsection{Compatible pairs and the specialized Wells sequences}
For a derivation pair $d=(D_V,D_\g)$, use $L_d$ from
\eqref{eq:mixed-Ld}. Its mixed equations reduce to
\begin{equation}\label{eq:sparse-compatible-top}
(L_d\vartheta)(x)(u_1,\ldots,u_{n-1})
=[\lambda(x),u_1,\ldots,u_{n-1}]_V
\end{equation}
and, if $n\geq4$,
\begin{align}
L_d\rho&=0,\notag\\
\vartheta(x)(\lambda(y),u_1,\ldots,u_{n-2})
&=\vartheta(y)(\lambda(x),u_1,\ldots,u_{n-2}).
\label{eq:sparse-compatible-higher}
\end{align}
If $n=3$, the last two equations are replaced by
\begin{equation}\label{eq:sparse-compatible-three}
(L_d\rho)(x,y)u
=\vartheta(x)(\lambda(y),u)-\vartheta(y)(\lambda(x),u).
\end{equation}
To obtain these formulas, insert $\lambda$ in one quotient slot
of each mixed bracket. Only a bracket with two quotient entries
can produce a nonzero $\vartheta$ term; inserting it in the
one-quotient bracket produces the kernel bracket.
This also explains why \eqref{eq:sparse-compatible-higher}
must include preservation of a zero intermediate component.

An automorphism pair $c=(\beta,\alpha)$ lifts precisely when
some $f:\g\to V$ satisfies
\begin{equation}\label{eq:sparse-lift-extremes}
{}^c\vartheta=\vartheta^f,\qquad
{}^c\rho=\rho^f,\qquad{}^c\omega=\omega^f,
\end{equation}
and all expressions in \eqref{eq:sparse-gauge-middle} are zero.
Here transport is \eqref{eq:twist-omega}--\eqref{eq:twist-rho},
restricted to the indicated components.
Its lift is
\[
\gamma(s(x)+u)=s(\alpha x)+f(\alpha x)+\beta u.
\]
For $n\geq5$ Theorem~\ref{thm:sparse-arity} reduces the complete
criterion to
\begin{align}
{}^c\rho&=\rho,\notag\\
({}^c\vartheta-\vartheta)(x)(u_1,\ldots,u_{n-1})
 &=[f(x),u_1,\ldots,u_{n-1}]_V,\notag\\
0&=\vartheta(x)(f(y),u_1,\ldots,u_{n-2})
 -\vartheta(y)(f(x),u_1,\ldots,u_{n-2})\notag\\
&\quad+[f(x),f(y),u_1,\ldots,u_{n-2}]_V,\notag\\
0&=[f(x),f(y),f(z),u_1,\ldots,u_{n-3}]_V,\notag\\
{}^c\omega-\omega&=d_\rho f.
\label{eq:sparse-higher-lift}
\end{align}
All these equations are required; an ordinary cohomology
primitive for the last line alone need not give a lift.

The top mixed equation also implies
\begin{equation}\label{eq:sparse-top-module}
{}^c\chi_s-\chi_s=\partial_V f,\qquad
({}^cq)(U)x=\beta q(\beta^{-1}U)(\alpha^{-1}x).
\end{equation}
For derivations it reads
$L_d\chi_s=\partial_V\lambda$, where $L_d$ acts on the output,
the kernel entries of $U$, and the quotient entry $x$.
These are the representation obstruction carried by
$[\chi_s]$, not substitutes for the remaining mixed equations.

With $\mathfrak G_\rho$ and $\mathcal C_\rho$ defined by these
mixed equations, the general results give
\begin{align}
0\longrightarrow Z_\Omega\longrightarrow\Der_V(E)
 &\longrightarrow\mathfrak G_\rho
 \xrightarrow{\mathcal W_{\mathcal E}}\Ht,\notag\\
1\longrightarrow Z_\Omega\longrightarrow\Aut_V(E)
 &\longrightarrow\mathcal C_\rho
 \xrightarrow{\varpi_\rho}\Ht,
\label{eq:sparse-Wells}
\end{align}
where
\[
Z_\Omega=\{f\in\Ft:d_\rho f=0\},\qquad
\mathcal W_{\mathcal E}(d)=[L_d\omega-d_\rho\lambda],
\qquad
\varpi_\rho(c)=[{}^c\omega-\omega^{f_c}].
\]
The pointed map on all pairs is the restriction of
\eqref{eq:Wells}. Its target may equivalently be restricted to
the subset of $H_{\nab}^2(\g,V)$ represented by sparse data.
No new exactness or splitting proof is needed: the general
theorems already include the zero components.

\subsection{Sparse examples}

\subsubsection{A nonzero representation class of \texorpdfstring{$\vartheta$}{vartheta}}
\begin{example}\label{ex:theta-class}
Let $N$ have basis $x,u_1,\ldots,u_{n-1},z$ and only the
nonzero basic bracket
\[
[x,u_1,\ldots,u_{n-1}]=z.
\]
Every bracket value is central, so all nested brackets vanish.
Let $S$ be centerless and non-abelian, and take
\[
E=S\oplus N,\qquad
V=S\oplus\langle u_1,\ldots,u_{n-1},z\rangle,\qquad
\g=\langle\bar x\rangle.
\]
For $s(\bar x)=x$,
\[
\rho=0,\quad\omega=0,\quad
\vartheta(\bar x)(u_1,\ldots,u_{n-1})=z.
\]
The quotient is one-dimensional, and the section is sparse
and centrally good. Direct calculation gives
\[
\At=\langle z\rangle,\qquad
\Ft=\Hom(\g,\langle u_1,\ldots,u_{n-1},z\rangle),\qquad
\Ht=0,\quad Z_\Omega=\Ft.
\]
Indeed, inserting each $u_i$ and all the other $u_j$ into
$\vartheta(\bar x)$ detects its coefficient, whereas $z$
annihilates every bracket.
The balance equation is automatic in one quotient dimension.
Thus $\Ft$ is strictly larger than $\Hom(\g,\At)$.

For $U=(u_1,\ldots,u_{n-1})$,
\[
\chi_s(U)\bar x=(-1)^{n-1}z\ne0,\qquad
(\partial_V f)(U)\bar x=[U,f(\bar x)]_V=0
\]
for every $f$, since all $u_i$ are central in $V$.
Hence $[\chi_s]\ne0$.
Although $s$ is an algebra section, the representation
extension \eqref{eq:theta-module-extension} does not split.

Let $\beta$ fix $S,u_1,\ldots,u_{n-1}$ and send $z$ to $2z$,
and let $\alpha=\id_\g$.
Then $ {}^c\vartheta(\bar x)(u_1,\ldots,u_{n-1})=2z$,
whereas the top correction
$[f(\bar x),u_1,\ldots,u_{n-1}]_V$ is always zero.
Thus $c$ has no mixed lift, despite $\omega=0$ and $\Ht=0$.
Likewise, the kernel derivation which sends $z$ to $z$ and
kills $S,u_i$, together with $D_\g=0$, is not compatible.
The obstruction in both cases is already carried by
$\vartheta$, before the pure Wells map is formed.
\end{example}

\subsubsection{Zero intermediate brackets can obstruct a lift}
\begin{example}\label{ex:sparse-middle-obstruction}
Let $n\geq4$ and enlarge the preceding $N$ by a central basis
vector $y$. Take the same kernel $V$ and the quotient
$\g=\langle\bar x,\bar y\rangle$.
The standard section has $\rho=\omega=0$ and
$\vartheta(\bar y)=0$.
Set $f(\bar x)=0$, $f(\bar y)=u_1$.
Centrality of $u_1$ in $V$ gives $\vartheta^f=\vartheta$, and
$\bigwedge^{n-1}\g=\bigwedge^n\g=0$ gives
$\rho^f=\rho=0$, $\omega^f=\omega=0$.
Nevertheless,
\[
[s(\bar x)+f(\bar x),s(\bar y)+f(\bar y),u_2,\ldots,u_{n-1}]
 =[x,y+u_1,u_2,\ldots,u_{n-1}]=z\ne0.
\]
Thus $s+f$ is not sparse, and $T_f$ is not an automorphism.
The derivation candidate $\delta_f(y)=u_1$ fails on the same
zero bracket:
\[
\delta_f[x,y,u_2,\ldots,u_{n-1}]=0,\qquad
[x,\delta_f(y),u_2,\ldots,u_{n-1}]=z.
\]
This verifies the need for the middle equations in both theories.
\end{example}

\subsubsection{The exceptional corrections in arity four}
\begin{example}\label{ex:sparse-four-corrections}
Let $V$ have basis $p_1,p_2,p_3,p_4,z$ with
$[p_1,p_2,p_3,p_4]_V=z$, and let $\g$ be abelian with
basis $x_1,x_2,x_3,x_4$.
Define $f(x_i)=p_i$ and
\[
\rho=0,\qquad\omega=0,\qquad
\vartheta(x)(u,v,w)=-\tfrac12[f(x),u,v,w]_V.
\]
Every bracket in $E$ has value in $\F z$ and vanishes if an
argument is $z$, so this is a sparse crossed system.
Its two-entry equation is
\begin{align*}
&\vartheta(x)(f(y),u,v)-\vartheta(y)(f(x),u,v)
                                +[f(x),f(y),u,v]_V\\
&=-\tfrac12[f(x),f(y),u,v]_V
  +\tfrac12[f(y),f(x),u,v]_V+[f(x),f(y),u,v]_V=0.
\end{align*}
Both sections are sparse, but
\[
\rho^f(x_1,x_2,x_3)p_4=-\tfrac12z,\qquad
\omega^f(x_1,x_2,x_3,x_4)=-z,\qquad d_\rho f=0.
\]
The cubic and quartic corrections are therefore essential.

Here $\chi_s=-\tfrac12\partial_V f$, so
$[\chi_s]=0$. The equivariant section $s+\tfrac12 f$
has $\vartheta^{f/2}=0$, but its middle bracket is
\[
[s(x_1)+\tfrac12p_1,s(x_2)+\tfrac12p_2,p_3,p_4]
 =-\tfrac14z\ne0.
\]
Thus splitting the representation extension does not ensure
that the resulting section is sparse.
\end{example}

For $n\geq5$, the analogous kernel
$[p_1,\ldots,p_n]_V=z$ and maps
$f(x_i)=p_i$, $\vartheta(x)=-\tfrac12[f(x),\,\cdot,\ldots,\cdot]_V$
still satisfy the two-entry equation, but
\[
[s(x_1)+p_1,s(x_2)+p_2,s(x_3)+p_3,p_4,\ldots,p_n]
 =-\tfrac12z.
\]
This is now an intermediate bracket, so the changed section
is not sparse.
If $f$ instead has image in $\langle p_1,p_2\rangle$, the
three-entry kernel bracket vanishes by alternation. Both
sections are then sparse, $\rho^f=\rho$ and
$\omega^f=\omega+d_\rho f$, while $\vartheta^f$ can still change.

\section{Specializations and examples}\label{sec:examples}
The following are specializations and examples.
\subsection{The abelian specializations }\label{sec:abelian}
Suppose $[V,V,E,\ldots,E]=0$, so that $V$ is abelian and every $\rho_r$ with $r\geq2$ vanishes.
Then $\rho$ is a representation in every arity $n\geq2$,
\[
A_\rho=V,\qquad \mathcal F_\rho=\Hom(\g,V),\qquad
\mathcal H_\rho^2=H_{\ab}^2(\g,V;\rho),\qquad Z_\Omega=Z^1(\g,V;\rho).
\]
The known semidirect-product and cocycle criteria are recalled
without reproving them.
\begin{proposition}\label{prop:semidirect}
For a representation $\rho$, the operation
\[
[x_1+u_1,\ldots,x_n+u_n]
 =[x_1,\ldots,x_n]_\g
  +\sum_i(-1)^{n-i}\rho(x_1,\ldots,\widehat x_i,\ldots,x_n)u_i
\]
is an $n$-Lie bracket.
\end{proposition}
\begin{proposition}\label{prop:cocycle-bracket}
Adding $\omega(x_1,\ldots,x_n)$ to this bracket gives an
$n$-Lie algebra precisely when $d_\rho\omega=0$.
Two such extensions with the same representation are equivalent
precisely when their cocycles differ by $d_\rho f$.
\end{proposition}
These are the standard representation and abelian extension
constructions; see~\cite{Kasymov,Takhtajan,Xu,TanXu,ZhangExtending}.
Here compatibility becomes
\[
\mathfrak G_\rho=\{d:L_d\rho=0\},\qquad \mathcal C_\rho=\{c:{}^c\rho=\rho\}.
\]
The mixed correction can be chosen zero, and the Wells maps are
\[
\mathcal W_{\mathcal E}(d)=[L_d\omega],\qquad
\varpi_\rho(c)=[{}^c\omega-\omega].
\]
The corresponding exact sequences are
\begin{align*}
0\longrightarrow Z^1(\g,V;\rho)\longrightarrow\Der_V(E)
 &\longrightarrow\mathfrak G_\rho
 \xrightarrow{\mathcal W_{\mathcal E}}H_{\ab}^2(\g,V;\rho),\\
1\longrightarrow Z^1(\g,V;\rho)\longrightarrow\Aut_V(E)
 &\longrightarrow\mathcal C_\rho
 \xrightarrow{\varpi_\rho}H_{\ab}^2(\g,V;\rho).
\end{align*}
They recover the ternary cohomology and Wells constructions
of~\cite{Xu,TanXu} and display their general-arity form.

\subsection{Direct products and corrected compatibility}
Assume $n\geq3$. For the direct product $E=\g\oplus V$ one has
$\rho=\vartheta=\omega=0$ and
\[
A_\rho=Z(V),\quad \mathcal F_\rho=\Hom(\g,Z(V)),\quad
Z_\Omega=\Hom(\g/[\g,\ldots,\g],Z(V)).
\]
Every automorphism or derivation pair lifts diagonally, and
both additive Wells maps vanish. The sequences split over the full
groups and Lie algebras of pairs.

Even for a direct product, strict compatibility depends on the
sparse section. Let $V$ be centerless and non-abelian, let $\g$
be abelian, and choose $b:\g\to V$ of rank one.
The section $s_b(x)=x+b(x)$ is sparse, since every intermediate
kernel bracket contains at least two values in the same
one-dimensional space. Its maps are
\[
\rho^b=0,\qquad\omega^b=0,\qquad
\vartheta^b(x)(u_1,\ldots,u_{n-1})
 =[b(x),u_1,\ldots,u_{n-1}]_V.
\]
For every pair $d=(D_V,D_\g)$ the diagonal derivation, written
in this section, has correction
$\lambda=D_Vb-bD_\g$.
Its top equation is directly
\[
(L_d\vartheta^b)(x)(u_1,\ldots,u_{n-1})
 =[D_Vb(x)-b(D_\g x),u_1,\ldots,u_{n-1}]_V.
\]
Choosing $D_\g=0$ and $D_Vb(x)\ne0$ makes this expression
nonzero for suitable kernel entries, since $Z(V)=0$.
Thus an extensible pair need not be strictly compatible.
Here $A_\rho=\mathcal F_\rho=0$, so every pair has a unique lift.

\subsection{Determinant and trace obstructions with non-abelian kernel}
Let $\g$ be abelian of dimension $m\geq n$,
$C=\bigwedge^n\g$, and $V=S\oplus C$, where $S$ is perfect,
centerless, and non-abelian.
Such kernels are supplied, for example, by the standard
simple $n$-Lie algebras in characteristic zero~\cite{Filippov,Ling}.
Define
\[
E=S\oplus\g\oplus C,\qquad
\rho=\vartheta=0,\qquad
\omega(x_1,\ldots,x_n)=x_1\wedge\cdots\wedge x_n.
\]
The $S$-bracket and the displayed central bracket are the only
nonzero types; their direct sum satisfies the Filippov identity.
Here $A_\rho=C$, $\mathcal F_\rho=\Hom(\g,C)$, and $d_\rho\mathcal F_\rho=0$.

The characteristic subspaces $[V,\ldots,V]=S$ and $Z(V)=C$
imply that every kernel automorphism is $b\oplus q$.
The top lifting equation forces the correction into $C$.
The pure equation then reads
\[
q(x_1\wedge\cdots\wedge x_n)
 =\alpha x_1\wedge\cdots\wedge\alpha x_n.
\]
Since decomposable wedges span $C$, inducibility is exactly
$q=\bigwedge^n\alpha$. All lifts are
\[
x+a+c\longmapsto
\alpha x+f(x)+b(a)+(\bigwedge^n\alpha)c,\qquad f:\g\to C.
\]
These formulas preserve the pure $\g$-brackets, the $S$-brackets,
and every zero mixed bracket, so the criterion is also sufficient.

Similarly, every kernel derivation is $D_S\oplus Q$.
Its top equation forces $\lambda(\g)\subseteq C$, and its
pure equation is
\[
Q(x_1\wedge\cdots\wedge x_n)
 =\sum_i x_1\wedge\cdots\wedge D_\g x_i\wedge\cdots\wedge x_n.
\]
All mixed identities are then zero. When $m=n$, the two criteria
become respectively
\[
q=(\det\alpha)\id_C,\qquad Q=(\tr D_\g)\id_C.
\]
For example, $(\id_S\oplus2\id_C,\id_\g)$ has no automorphism
lift, and $(0\oplus\id_C,0)$ has no derivation lift.
Both pairs satisfy all mixed equations.
These are pure cohomological obstructions, complementary to
the representation obstruction in Example~\ref{ex:theta-class}.

\subsection{A noncentral action and a nonzero derivation obstruction}
\begin{example}\label{ex:noncentral-obstruction}
Let $n\geq3$, let $\g$ have basis $x_1,\ldots,x_n$ with
$[x_1,\ldots,x_n]_\g=x_1$, and put
$A=\langle a_1,a_2\rangle$.
Define an idempotent $P\in\End(A)$ by
\[
P(a_1)=P(a_2)=\tfrac12(a_1+a_2).
\]
Let $S$ be centerless and non-abelian, and put
$V=S\oplus A$, with $A$ an abelian direct summand.
For $w\in A$, define the extension by
\begin{align}
[x_1,\ldots,x_n]&=x_1+w,\notag\\
[x_2,\ldots,x_n,a]&=(-1)^{n-1}Pa\quad(a\in A),
\label{eq:noncentral-example-bracket}
\end{align}
together with the bracket of $S$; all other basic mixed
brackets are zero.

Here is a verification that every $w$ gives an $n$-Lie algebra.
Put $U=\langle x_2,\ldots,x_n\rangle$,
$K=\langle x_1\rangle\oplus A$, and let $\nu$ be the volume
form on $U$ with $\nu(x_2,\ldots,x_n)=1$. Define
\[
T(x_1)=(-1)^{n-1}(x_1+w),\qquad
T|_A=(-1)^{n-1}P.
\]
On the abelian algebra $U$ the map
$\rho_U(u_1,\ldots,u_{n-1})=\nu(u_1,\ldots,u_{n-1})T$
is a representation. The first representation identity
holds because scalar multiples of $T$ commute. The second
follows from
\[
\sum_{i=1}^n(-1)^{n-i}
\nu(v_1,\ldots,\widehat v_i,\ldots,v_n)
\nu(u_1,\ldots,u_{n-2},v_i)=0,
\]
the alternating determinant relation for $n$ vectors in
the $(n-1)$-dimensional space $U$.
Proposition~\ref{prop:semidirect} gives the bracket on
$U\oplus K$, and taking its direct product with $S$ gives
\eqref{eq:noncentral-example-bracket}.

The kernel is non-abelian, while the quotient action on $A$
is nontrivial. For the evident section,
\[
\omega_w(x_1,\ldots,x_n)=w,\qquad
\rho(x_2,\ldots,x_n)|_A=(-1)^{n-1}P,
\]
and $\vartheta=0$. Every intermediate mixed bracket is zero,
while the kernel retains the nonzero bracket of $S$.
Thus
\[
A_\rho=A,\qquad \mathcal F_\rho=\Hom(\g,A).
\]
Since every $w$ gives a cocycle and
$\dim\bigwedge^n\g=1$, evaluation on $x_1,\ldots,x_n$ identifies
$\Zt$ with $A$. For $f\in\mathcal F_\rho$,
\[
(d_\rho f)(x_1,\ldots,x_n)=(P-\id)f(x_1).
\]
Consequently
\begin{equation}\label{eq:noncentral-H}
\mathcal H_\rho^2
 =A/(P-\id)A=A/\langle a_1-a_2\rangle.
\end{equation}

Take $D_V|_S=0$, $D_V|_A=2P$, and
\[
D_\g x_1=x_1,\quad D_\g x_2=x_2,\quad
D_\g x_n=-x_n,\quad
D_\g x_j=0\quad(3\leq j\leq n-1).
\]
The trace of $D_\g$ is $1$, and
\[
D_\g[x_1,\ldots,x_n]_\g=x_1
 =\sum_i[x_1,\ldots,D_\g x_i,\ldots,x_n]_\g.
\]
Thus $D_\g$ is a derivation. The pair is strictly compatible:
$[2P,P]=0$ and the sum of the eigenvalues on $x_2,\ldots,x_n$
is zero. Its correction can be chosen as $\lambda=0$, and
\[
\Theta^{\omega_w}_{d,0}(x_1,\ldots,x_n)=2Pw-w.
\]
For $w=a_1-2a_2$, this equals $-2a_1+a_2$, which is not
in $\langle a_1-a_2\rangle$.
Therefore the pair has no derivation lift.

More precisely, on the quotient \eqref{eq:noncentral-H},
$P$ induces the identity, so
\[
\Phi_{\rho,\vartheta}(d)=\id,\qquad
\mathcal W_{\mathcal E_w}(d)=[w],\qquad
\mathcal W_{\mathcal E_0}(d)=0.
\]
The pair extends in the split member but fails to extend
in every member of the same mixed family. This explicitly
realizes the cohomology-action term in
Theorem~\ref{thm:universal-extensibility}.
\end{example}

\subsection{An intermediate mixed obstruction}
\begin{example}\label{ex:intermediate}
Let $n\geq4$. Let $N$ have basis
\[
x_1,\ldots,x_{n-2},a,b,z
\]
and the single nonzero basic bracket
\begin{equation}\label{eq:intermediate-bracket}
[x_1,\ldots,x_{n-2},a,b]=z.
\end{equation}
All brackets involving $z$ vanish. Hence every nested bracket is
zero, and \eqref{eq:FI} holds.
Let $S$ be any non-abelian $n$-Lie algebra, and set
\[
E=S\oplus N,\qquad
V=S\oplus\operatorname{span}\{a,b,z\},\qquad
\g=\operatorname{span}\{\bar x_1,\ldots,\bar x_{n-2}\}.
\]
The subspace $V$ is an ideal, and $\g$ is abelian.
The section $s(\bar x_i)=x_i$ is an algebra section.
Its datum satisfies
\begin{equation}\label{eq:intermediate-datum}
\omega=0,\qquad \rho_1=0,\qquad \rho_{n-1}=0,\qquad
\rho_2(\bar x_1,\ldots,\bar x_{n-2})(a,b)=z,
\end{equation}
with all other mixed components zero.

Define $\alpha=\id_{\g}$ and let $\beta\in\Aut(V)$ fix $S,a,b$
and send $z$ to $2z$. This is an automorphism because
$\operatorname{span}\{a,b,z\}$ is a central direct summand of $V$.
If a lift existed, it would have
$\gamma(x_i)=x_i+\lambda(\bar x_i)$ and
$\gamma(a)=a$, $\gamma(b)=b$. Preservation of
\eqref{eq:intermediate-bracket} would give
\begin{align*}
2z
 &=\gamma([x_1,\ldots,x_{n-2},a,b])\\
 &=[x_1+\lambda(\bar x_1),\ldots,
    x_{n-2}+\lambda(\bar x_{n-2}),a,b]\\
 &=z.
\end{align*}
Indeed, any term containing $\lambda(\bar x_i)\in V$ has fewer
than $n-2$ of the required $x$-entries in its $N$-component,
and its $S$-component is zero because of the last two entries
$a,b$. This is a contradiction. Hence
\[
\varpi(\beta,\id_{\g})\ne[\Omega].
\]
In coordinates, the transformed nonzero mixed operation is
$ {}^{(\beta,\id)}\rho_2=2\rho_2$.
Thus this obstruction is present in a split extension with
non-abelian kernel, and it is invisible to $\omega$, $\rho_1$,
and $\rho_{n-1}$ alone.
\end{example}

\section*{Acknowledgments}
ChatGPT (6-Astra) was used in the process of this manuscript.


\begin{thebibliography}{99}

\bibitem{AfiBasdouri}
M.~Afi and O.~Basdouri,
On the isomorphism of non-abelian extensions of $n$-Lie algebras,
\emph{J. Geom. Phys.} \textbf{173} (2022), 104427.
\url{https://doi.org/10.1016/j.geomphys.2021.104427}.


\bibitem{AlekseevskyMichorRuppert}
D.~Alekseevsky, P.~W.~Michor, and W.~A.~F.~Ruppert,
Extensions of super Lie algebras,
\emph{J. Lie Theory} \textbf{15} (2005), no.~1, 125--134.
\url{https://arxiv.org/abs/math/0101190}.

\bibitem{AzcarragaIzquierdo}
J.~A.~de Azc\'arraga and J.~M.~Izquierdo,
$n$-ary algebras: a review with applications,
\emph{J. Phys. A: Math. Theor.} \textbf{43} (2010), no.~29, 293001.
\url{https://arxiv.org/abs/1005.1028}.

\bibitem{BardakovSingh}
V.~G.~Bardakov and M.~Singh,
Extensions and automorphisms of Lie algebras,
\emph{J. Algebra Appl.} \textbf{16} (2017), no.~9, 1750162, 15 pp.
\url{https://doi.org/10.1142/S0219498817501626}.

\bibitem{CasasLodayPirashvili}
J.~M.~Casas, J.-L.~Loday, and T.~Pirashvili,
Leibniz $n$-algebras,
\emph{Forum Math.} \textbf{14} (2002), no.~2, 189--207.

\bibitem{DaletskiiTakhtajan}
Y.~L.~Daletskii and L.~A.~Takhtajan,
Leibniz and Lie algebra structures for Nambu algebra,
\emph{Lett. Math. Phys.} \textbf{39} (1997), 127--141.

\bibitem{Filippov}
V.~T.~Filippov,
$n$-Lie algebras,
\emph{Siberian Math. J.} \textbf{26} (1985), 879--891.

\bibitem{Fregier}
Y.~Fr\'egier,
Non-abelian cohomology of extensions of Lie algebras as Deligne groupoid,
\emph{J. Algebra} \textbf{398} (2014), 243--257.

\bibitem{InassaridzeKhmaladzeLadra}
N.~Inassaridze, E.~Khmaladze, and M.~Ladra,
Non-abelian cohomology and extensions of Lie algebras,
\emph{J. Lie Theory} \textbf{18} (2008), 413--432.

\bibitem{Kasymov}
Sh.~M.~Kasymov,
On a theory of $n$-Lie algebras,
\emph{Algebra and Logic} \textbf{26} (1987), 155--166.

\bibitem{Ling}
W.~X.~Ling,
\emph{On the Structure of $n$-Lie Algebras},
Ph.D. thesis, University-GHS-Siegen, Siegen, 1993.


\bibitem{SongMakhloufTang}
L.~Song, A.~Makhlouf, and R.~Tang,
On non-abelian extensions of $3$-Lie algebras,
\emph{Commun. Theor. Phys.} \textbf{69} (2018), no.~4, 347--356.
\url{https://arxiv.org/abs/1708.02847}.

\bibitem{Takhtajan}
L.~A.~Takhtajan,
Higher order analog of Chevalley--Eilenberg complex and deformation
theory of $n$-gebras,
\emph{St. Petersburg Math. J.} \textbf{6} (1995), 429--438.

\bibitem{TanXu}
Y.~Tan and S.~Xu,
The Wells map for abelian extensions of $3$-Lie algebras,
\emph{Czechoslovak Math. J.} \textbf{69} (144) (2019),
no.~4, 1133--1164.
\url{https://doi.org/10.21136/CMJ.2019.0098-18}.

\bibitem{Wells}
C.~Wells,
Automorphisms of group extensions,
\emph{Trans. Amer. Math. Soc.} \textbf{155} (1971), 189--194.

\bibitem{Xu}
S.~Xu,
Cohomology, derivations and abelian extensions of $3$-Lie algebras,
\emph{J. Algebra Appl.} \textbf{18} (2019), no.~7, 1950130, 26 pp.
\url{https://doi.org/10.1142/S0219498819501305}.


\bibitem{Zhang3Lie}
T.~Zhang,
Extending structures for \(3\)-Lie algebras,
\emph{Comm. Algebra} \textbf{50} (2022), 1469--1497.


\bibitem{ZhangExtending}
T.~Zhang,
Extending structures for $n$-Lie algebras,
preprint (2026), arXiv:2609.05051.
\url{https://arxiv.org/abs/2609.05051}.

\end{thebibliography}
\end{document}